\documentclass[preprint,12pt]{elsarticle}

\usepackage{amssymb}
\usepackage{amsmath,amsthm,mathtools}
\usepackage{graphicx}
\usepackage{caption}
\usepackage{subcaption}
\usepackage{booktabs}
\usepackage{textcomp}
\usepackage{lipsum}
\usepackage{afterpage}
\usepackage{natbib}
\usepackage{extarrows}
\RequirePackage{tikz}
\usetikzlibrary{shapes,arrows}

\usepackage{csl}
\usepackage{soul,ulem}

\newtheorem{proposition}{Proposition}[section]

\newtheorem{definition}[proposition]{Definition}
\newtheorem{remark}[proposition]{Remark}
\newtheorem{assumption}[proposition]{Assumption}
\newtheorem{lemma}[proposition]{Lemma}

\journal{Applied Numerical Mathematics}

\begin{document}

\par\vfill\break 

\advance\vsize by 1cm 
\advance\voffset by -0.5cm 
\csltitle{A Patankar predictor-corrector approach for positivity-preserving time integration}

\cslauthor{Kamila Nurkhametova, Reid J. Gomillion, Amit N. Subrahmanya, and Adrian Sandu} 

\cslemail{kamila@vt.edu, rjg18@vt.edu, famitnagesh@anl.gov, asandu7@vt.edu}

\cslyear{26}
\cslreportnumber{1}

\csltitlepage

\par\vfill\break 

\advance\vsize by -1cm 
\advance\voffset by 0.5cm 

\begin{frontmatter}



\title{A Patankar predictor-corrector approach for positivity-preserving time integration}


\author[1]{Kamila Nurkhametova} 
\author[1]{Reid J. Gomillion}
\author[2]{Amit N. Subrahmanya}
\author[1]{Adrian Sandu}

\affiliation[1]{organization={Computational Science Laboratory, Department of Computer Science, Virginia Tech},
            addressline={620 Drillfield Dr.},
            city={Blacksburg},
            postcode={24061},
            state={VA},
            country={USA}}
\affiliation[2]{organization={Mathematics and Computer Science Division, Argonne National Laboratory},
            addressline={9700 S Cass Ave.},
            city={Lemont},
            postcode={60439},
            state={IL},
            country={USA}}

\begin{abstract}

Many natural processes, such as chemical reactions and wave dynamics, are modeled as production-destruction (PD) systems that obey positivity and linear conservation laws. Classical time integrators do not guarantee positivity and can produce negative or nonphysical numerical solutions. This paper presents a modular correction strategy that can be applied to stiffly implicit Runge-Kutta schemes, in particular stiffly accurate SDIRK methods. The strategy combines stage-wise clipping with a ratio-based scaling that enforces invariants and is guaranteed to yield non-negative, conservative solutions. We provide a theoretical analysis of the corrected schemes and characterize their worst-case order of accuracy relative to the underlying base method. Numerical experiments on stiff ODE systems (Robertson, MAPK, stratospheric chemistry) and a nonlinear PDE (the Korteweg-De Vries equation) demonstrate that the corrected SDIRK methods preserve positivity and invariants without significant loss of accuracy. Importantly, corrections applied only to the final stage are sufficient in practice, while applying them at all stages may distort dynamics in some cases.  These results show that the proposed framework provides a simple and effective way to construct positivity-preserving integrators for stiff PD systems.
\end{abstract}




\end{frontmatter}



\section{Introduction}
\label{sec:introduction}
Many systems are described by processes that naturally involve strictly non-negative variables such as chemical concentrations, densities, biological populations, or energy levels. These processes are often modeled using systems of ordinary differential equations (ODEs), especially ones that are structured as production-destruction systems, where components engage with each other through transfer, conversion, or decay processes. Two fundamental requirements for the numerical simulation of such systems are positivity, which ensures that all variables remain non-negative throughout the numerical trajectory, and conservation, which ensures that the linear invariants of the system---such as total mass---are preserved at any time.

Positivity of numerical solutions is not only a physical faithfulness requirement, but also a mathematical necessity to avoid producing non-physical or unstable dynamics. As shown by Hundsdorfer and Verwer \cite{hundsdorfer2013numerical}, even simple reaction systems such as $A + B \xrightarrow{\text{k}} C$, where $A(t)$, $B(t)$, and $C(t)$ denote species concentrations governed by mass-action kinetics, can exhibit severe numerical issues if positivity is violated. They show that negative initial concentrations can lead to unstable solutions that blow up in finite-time, even though the exact solution remains well-behaved for any non-negative initial values.  From a computational standpoint, this implies that even a small negative value introduced by a numerical error can cascade into instability, especially in stiff or nonlinear systems. Furthermore, positivity violations may disrupt conservation laws and invalidate any meaningful physical interpretation of the simulation results \cite{hundsdorfer2013numerical}. Therefore, time integration methods for such systems must be carefully constructed to preserve positivity, to avoid results that are physically meaningless or numerically divergent.

However, standard numerical integration methods, including Runge-Kutta schemes, do not guarantee these properties when applied to stiff or nonlinear systems. This can lead to non-physical results such as negative concentrations or non-conservation of mass under naive fixes such as clipping. To address these limitations, a variety of positivity-preserving schemes have been developed. The most widely used approaches are modified Patankar methods \cite{Burchard_2003_high-order,Burchard_2005_application} and nonstandard finite difference schemes including geometric conservative (GeCo) integrators \cite{Martiradonna_2020_geco, Izgin_2023_on-the-dynamics}. These methods include algebraic modifications, scaling techniques, or stability-aware formulations to maintain both positivity and invariant preservation. While these approaches can guarantee positivity, they often suffer from order reduction, additional computational cost, or difficulty achieving high orders of accuracy, particularly for stiff systems.

In this paper, we propose a general positivity-preserving correction framework that can be applied as a predictor-corrector post-processing step to existing time integration methods. The key idea is to use the underlying integrators, such as singly diagonally implicit Runge-Kutta (SDIRK) schemes \cite{wanner1996solving}, as solution predictors, and then apply algebraic corrections clipping negative components and applying scaling matrices to enforce positivity and conservation. This approach is simple and modular, requiring no modifications of the base solvers. Similar such correction procedures have been used in the application to the Shallow Water Equations \cite{meisterUnconditionallyPositiveImplicit2014a}.

The proposed approach is tested on a diverse set of benchmark problems: the stiff Robertson reaction system \cite{Blanes_2022_positivity,wanner1996solving}, the mitogen-activated protein kinase (MAPK) cascade \cite{Blanes_2022_positivity,hadavc2017minimal}, the stratospheric chemistry model \cite{Sandu_2001_positiveChemistry,Sandu_2002_positivity-Favoring,Blanes_2022_positivity}, and a nonlinear dispersive partial differential equation, the Korteweg-De Vries (KdV) equation \cite{Jager_2011_originKdV,drazin1989solitons}. These examples span linear and nonlinear PD structures and include both ODE and PDE settings. Our experiments demonstrate that the correction mechanism enforces positivity and conservation for highly stiff multiscale problems. It also maintains, in practice, the expected order of accuracy for SDIRK integrators and improves solver robustness with minimal overhead. We also identify several limitations of the approach: for the KdV problem, applying corrections at all stages distorted the dynamics, and for explicit Runge-Kutta methods the correction reduced convergence to first order. These results delineate the regime of effectiveness of the framework.

The remainder of this paper is organized as follows. Section \ref{sec:lit_review} reviews existing work on production-destruction systems, existing positivity-preserving schemes, and SDIRK methods. Section \ref{sec:theoretical_background} describes the proposed correction framework and its theoretical properties. Section \ref{sec:test_problems_and_methods} introduces the test problems and the numerical methods used in the experiments. Section \ref{sec:results} presents numerical experiments on a range of stiff problems, and confirms positivity preservation, invariant preservation, order of accuracy, and efficiency of the approach. Section \ref{sec:conclusions} summarizes the main findings of this work.

\section{Review of literature on positive time integration} 
\label{sec:lit_review}

Numerical methods for ordinary differential equations must always preserve physical invariants such as the system's total mass or energy levels. This is especially important when modeling real-world processes and evolutions where these components cannot be negative. However, classical numerical methods, such as Runge-Kutta, may require severe time step restrictions to preserve positivity when applied to stiff production-destruction systems. Naive positivity corrections can, in turn, destroy mass conservation and lead to non-physical solutions. This has led to the development of specific positivity preserving schemes, discussed next.

\subsection{Patankar-type methods}

\emph{Production-destruction systems (PDS)} describe the time evolution of $N$ interacting species through balanced production and destruction processes \cite{Burchard_2003_high-order,Burchard_2005_application}. They naturally arise in many applications, such as chemical kinetics, biochemical signaling networks, discrete hyperbolic conservation laws, and ecological models. In these systems, the value of each variable (e.g., the concentration of each species) changes due to production from other species and destruction into others.

A general PDS can be written as
\begin{subequations}
\label{eqn:PDS}
\begin{equation}
\label{eqn:PDS-dynamics}
\frac{\textnormal{d} \mathbf{y}_i}{\textnormal{d}t} = P_i(\mathbf{y}) - D_i(\mathbf{y}), \quad i = 1, \dots, N, \quad t \ge t_0,
\end{equation}
where $P_i(\mathbf{y}) \ge 0$ and $D_i(\mathbf{y}) \ge 0$ denote the total production and destruction rates of species $i$. Moreover, if a component is zero then its destruction rate is also zero, $\mathbf{y}_i = 0$  $\Rightarrow$ $D_{i}(\mathbf{y})=0$.

A PDS is called \emph{conservative} if production and destruction are pairwise balanced, i.e.,
\begin{equation}
\label{eqn:pairwise-balance}
P_i(\mathbf{y}) = \sum_{j=1}^N p_{i,j}(\mathbf{y}), \quad
D_i(\mathbf{y}) = \sum_{j=1}^N d_{i,j}(\mathbf{y}), \quad
p_{i,j}(\mathbf{y}) = d_{j,i}(\mathbf{y})~\forall\, i,j,\mathbf{y},
\end{equation}
\end{subequations}
where $d_{i,j}(\mathbf{y}) \ge 0$ is the rate at which the $i$-th constituent transforms into the $j$-th component, while $p_{i,j}(\mathbf{y}) \ge 0$ is the corresponding production rate of the $i$-th component due to contributions from the $j$-th species \cite{Kopecz_2018_on-order}. 

Such pairwise balance \eqref{eqn:pairwise-balance} in \eqref{eqn:PDS-dynamics} directly implies total mass conservation:
\[
\frac{\textnormal{d}}{\textnormal{d}t} \sum_{i=1}^N \mathbf{y}_i = \sum_{i,j=1}^N p_{i,j}(\mathbf{y}) - \sum_{i,j=1}^N d_{i,j}(\mathbf{y}) \equiv 0 
\quad \Rightarrow \quad 
\sum_{i=1}^N \mathbf{y}_i(t) = \text{const}.
\]
Moreover, if the initial data $\mathbf{y}(0) \ge 0$, the exact solution of a PDS remains non-negative for all $t>0$.

Since $P_i(\mathbf{y}) \ge 0$ and the destruction rate of a species vanishes whenever that species concentration is zero, $\mathbf{y}_i = 0 \Rightarrow D_i(\mathbf{y}) = 0$, it follows that $\frac{\mathrm{d}\mathbf{y}_i}{\mathrm{d}t} \ge 0$ for $\mathbf{y}_i \to 0$. Therefore, the set of non-negative states is invariant under the PDS dynamics, and solutions starting from non-negative initial data, $\mathbf{y}(t_0) \succeq 0$, remain non-negative for all $t \ge t_0$ \cite{Burchard_2003_high-order}.

However, standard explicit and implicit time integrators may violate these structural properties, producing negative concentrations or artificial mass gain/loss when large time steps are used. This observation has motivated the development of structure-preserving time integration methods designed to enforce positivity and conservation at the discrete level.

Among the earliest and most influential of these approaches are Patankar-type schemes, originally developed for chemical kinetics. Notably, the Modified Patankar-Euler (MPE) and Modified Patankar-Runge-Kutta (MPRK) schemes introduced by \cite{Burchard_2003_high-order} and later extended by \cite{Burchard_2005_application} guarantee unconditional positivity and conservation for stiff systems. These schemes modify destruction terms to depend linearly on the current solution, resulting in $M$-matrix structures that ensure non-negativity.

Based on this, Kopecz and Meister \cite{Kopecz_2018_on-order} derived order conditions for MPRK schemes. They introduced the second-order MPRK22 family and showed that Patankar-weight denominators play an important role in improving accuracy. Their following work \cite{Kopecz_2019_on-the-existence} demonstrated the non-existence of third-order three-stage MPRK schemes using standard Patankar weights. In \cite{Kopecz_2018_unconditionally}, they derived necessary and sufficient conditions for third-order MPRK schemes and introduced the first family of such methods, namely four-stage third-order MPRK43 schemes. They also noted that the existence of a third-order three-stage MPRK scheme remains an open question. Subsequently, Izgin et al. \cite{Izgin_2022_on-lyapunov, Izgin_2022_on-the-stability} developed a Lyapunov-based framework for analyzing the stability of such schemes, proving the stability of MPRK22 methods.

For higher-order positivity preserving {\"O}ffner and Torlo \cite{offnerArbitraryHighorderConservative2020} developed arbitrary order deferred correction based schemes. For MPRK schemes order conditions for arbitrary order have been derived \cite{izginOrderConditionsRungeKuttalike2025}.
However, third-order accuracy cannot be achieved with only three stages under standard Patankar weights, and requires additional stages and modified denominators.

\subsection{Nonstandard finite difference methods}

In order to address higher-order accuracy for positivity-preserving schemes, Martiradonna et al. \cite{Martiradonna_2020_geco} developed the Geometric Conservative (GeCo) and modified GeCo (mGeCo) methods. These methods are explicit, positive, and linear-invariant conservative, and are nonstandard finite difference schemes. Order conditions for methods with solution dependent coefficients are derived in \cite{izginOrderConditionsRungeKuttalike2025}. These integrators show stability and their suitability for stiff biochemical systems. However, higher order variants such as GeCo2 show bounded stability, limiting their use in highly stiff problems \cite{Izgin_2023_on-the-dynamics}.

\subsection{Strong stability preserving methods}

Recent works have focused on combining strong stability frameworks (SSPs) with the Patankar approach to handle problems that involve both convection and stiff reactive source terms. Such problems frequently arise in convection-reaction systems, chemically reacting flows, and hyperbolic conservation laws with stiff kinetics. In these cases, maintaining positivity and mass conservation is critical even under large time-step restrictions.

Huang, Juntao et al. \cite{Huang_2023_on-the-stability} introduced strong-stability-preserving (SSP) modified Patankar-Runge-Kutta schemes, which integrate the unconditional positivity and conservation properties of Patankar methods into the SSP Runge-Kutta framework. These schemes preserve positivity and conservation under the standard CFL conditions dictated by convection, while remaining independent of the stiffness of the source terms.
Stability is analyzed through the Lyapunov-based framework, offering parameter-dependent guarantees while maintaining high-order accuracy in convection-reaction systems.

\subsection{Exponential and Magnus integrators}

Other positivity-preserving strategies include exponential and Magnus integrators \cite{Blanes_2022_positivity}. These methods exploit the inherent graph Laplacian structures present in certain ODEs to preserve positivity and mass conservation, thereby bypassing the order limitations typically encountered in classical methods. Their second- and third-order variants show robust performance on stiff problems while following the principles of geometric integration.

\subsection{Optimization based approaches}

A different approach to obtaining positive solutions was taken in \cite{Sandu_2001_positiveChemistry,Sandu_2002_positivity-Favoring}, where it was observed that the set of non-negative values that obey all linear invariants forms a convex set in the state space. Projecting the solution produced by a standard numerical scheme onto this convex set results in a new solution that is non-negative, conservative, and at least as accurate as the original standard solution. This approach checks all the desired solution properties, but increases computational cost, since one optimization problem needs to be solved at each step for solution projection.

\subsection{Summary}

In summary, the positive time integration field has evolved from low-order, robust schemes to a wide spectrum of methods ensuring accuracy, structure preservation, and efficiency. The ongoing challenge is to extend these properties to higher-order schemes applicable to complex, stiff systems without losing theoretical guarantees or practical usability.

This paper introduces a Patankar predictor-corrector strategy that can be applied to stiffly accurate implicit Runge-Kutta (RK) and Singly-Diagonally Implicit Runge-Kutta (SDIRK) methods. The approach has two-phases: an initial predictor using standard RK or SDIRK steps, followed by a corrector that ensures positivity and conservation through a clip-and-scale mechanism. Stage values are clipped to eliminate negativity and then scaled using diagonal matrices that preserve the structure of systems. This method allows the application of classical time integration methods while preserving the properties required for the physical system and formal order of accuracy.

\section{Theoretical background}
\label{sec:theoretical_background}

\subsection{Graph Laplacian systems}

We consider systems with a graph Laplacian structure of the form \cite{Blanes_2022_positivity}:
\begin{equation}
\label{eqn:graph-laplacian-system}
\mathbf{y}' = \mathbf{f}(t,\mathbf{y}) = \mathbf{G}(t,\mathbf{y})\, \mathbf{y}, \quad \mathbf{y}(t_0) = \mathbf{y}_0 \succeq 0 \in \mathbb{R}^d, \quad \mathbf{G}(t,\mathbf{y}) \in \mathbb{R}^{d \times d},
\end{equation}
where the symbol $\succeq 0$ denotes component-wise non-negativity. Structure \eqref{eqn:graph-laplacian-system} allows to model systems that enjoy positivity (state variables must remain non-negative for all time) and conservation laws (some linear invariants remain constant throughout the evolution).

The matrix $\mathbf{G}(t,\mathbf{y})$ in \eqref{eqn:graph-laplacian-system} has the following properties.

\begin{assumption}[Stability]
\label{ass:stability}
The matrix $\mathbf{G}(t,\mathbf{y})$ has eigenvalues with non-positive real parts for any $t,\mathbf{y}$.
\end{assumption}

\begin{assumption}[Sign structure] 
\label{ass:positivity}
The matrix $\mathbf{G}(t,\mathbf{y})$ has non-positive diagonal entries, and non-negative off-diagonal entries:
\begin{equation}
\label{eqn:G-entries-sign}
\begin{cases}
\mathbf{G}_{i,i}(t,\mathbf{y}) \le 0, & \forall\, i = 1,\dots,d, \\
\mathbf{G}_{i,j}(t,\mathbf{y}) \ge 0, &\forall\, i, j = 1,\dots,d ~~\textnormal{with}~~ i \ne j, 
\end{cases}
\quad \forall\, t \ge t_0,\, \forall\, \mathbf{y}\succeq 0.
\end{equation} 
This assumption guarantees that any component which becomes zero cannot be further destroyed, thus ensuring non-negativity of the exact solution of \eqref{eqn:graph-laplacian-system}.
\end{assumption}

\begin{assumption}[Strong sign entries assumption] 
\label{ass:strong-positivity}
The matrix $\mathbf{G}(t,\mathbf{y})$ has non-positive diagonal entries, and non-negative off-diagonal entries, for any argument:

\begin{equation}
\text{Entry sign property} \quad 
\eqref{eqn:G-entries-sign} \quad \text{holds for } \quad \forall\, t \ge t_0 , \, \forall\, \mathbf{y} \in \mathbb{R}^d.
\label{eqn:strong-G-entries-sign}
\end{equation}

Note that \eqref{eqn:strong-G-entries-sign} states that one does not need $\mathbf{y} \succeq 0$ for the sign entries property to hold.
\end{assumption}

As a consequence of Assumption \ref{ass:positivity} (and of the stronger Assumption \ref{ass:strong-positivity}) the solution of \eqref{eqn:graph-laplacian-system} is non-negative:
\begin{equation}
\label{eqn:positive-solution}
\begin{split}
\mathbf{y}(t_0) \succeq 0 ~~\textnormal{with}~~ \mathbf{y}_i(t_0) = 0 \quad & \Rightarrow \quad \mathbf{y}'_i(t_0) = \sum_{j \ne i} \mathbf{G}_{i,j}\big((t_0,\mathbf{y}(t_0)\big)\,\mathbf{y}_j(t_0) \ge 0 \\
& \Rightarrow \quad \mathbf{y}(t) \succeq \mathbf{0}, ~~\forall\, t \ge t_0.
\end{split}
\end{equation}

\begin{assumption}[Linear invariants] 
\label{ass:invariants}
There exist vectors $\mathbf{w}_1,\dots,\mathbf{w}_M \in \mathbb{R}^d$ with $M < d$ such that:
\begin{equation}
\label{eqn:linear-invariants}
\mathbf{w}_j^T\,\mathbf{G}(t,\mathbf{y}) = \mathbf{0}, \quad j=1,\dots,M, \quad \forall\, t,\mathbf{y}.
\end{equation}
\end{assumption}
As a consequence of Assumption \ref{ass:invariants} the vectors $\mathbf{w}_j$ are linear invariants of the ODE system \eqref{eqn:graph-laplacian-system}:
\begin{equation}
\label{eqn:invariant}
\mathbf{w}_j^T\,\mathbf{y}' = \mathbf{0} \quad \Rightarrow \quad \mathbf{w}_j^T\,\mathbf{y}(t) = \mathbf{w}_j^T\,\mathbf{y}_0 = \textnormal{const}, \quad j=1,\dots,M, \quad \forall\, t \ge t_0.
\end{equation}

\begin{remark}[Graph Laplacian]
\label{rem:graph-laplacian}
The matrix $\mathbf{G}(t,\mathbf{y})$ is called a graph Laplacian if it satisfies two key structural properties:
\begin{enumerate}
    \item Sign structure Assumption \ref{ass:positivity}, and
    \item 
    Assumption \ref{ass:invariants} with $\mathbf{w}_1 = \boldsymbol{1}_{d}$, where $\boldsymbol{1}_d \in \mathbb{R}^d$ is the column vector of ones. Consequently, each column entries sums up to zero, $\sum_{i=1}^d \mathbf{G}_{i,j}(t,\mathbf{y}) = 0$ for all $j$, and the system preserves the total mass $\sum_{i=1}^d y_i(t) = \sum_{i=1}^d y_i(0)$ \(\forall t \ge t_0\).
\end{enumerate}
Systems \eqref{eqn:graph-laplacian-system} under consideration are part of the class of graph Laplacian systems \cite{Blanes_2022_positivity}.
\end{remark}

To design positivity-preserving integrators, we rely on some fundamental matrix properties.

\begin{remark}[Positivity of the inverse]

Let $\mathbf{I}_{d \times d} \in \mathbb{R}^{d \times d}$ denote the identity matrix. 
For any step size $h \ge 0$, the matrix $\mathbf{I}_{d \times d} - h\, \mathbf{G}$ has positive diagonal entries, and non-positive off-diagonal entries, due to Assumption \ref{ass:positivity}. Moreover, the eigenvalues $\sigma(\mathbf{I}_{d \times d} - h\, \mathbf{G}) \subset \mathbb{C}^{+}$ due to Assumption \ref{ass:stability}. Consequently $\mathbf{I}_{d \times d} - h\, \mathbf{G}$ is an $M$-matrix \cite{Blanes_2022_positivity} and
\[
(\mathbf{I}_{d \times d} - h\, \mathbf{G})^{-1} \succeq \mathbf{0}.
\]
This property is proven in \cite{Blanes_2022_positivity} for $M=1$ and $\mathbf{w}_1= \boldsymbol{1}_{d}$.
\end{remark}

\begin{remark}[Scaling the columns of graph Laplacian matrix]
\label{rem:scaling}
Consider a diagonal matrix $\boldsymbol{\Sigma}$ with non-negative entries:
\begin{subequations}
\label{eqn:scaling}
\begin{equation}
\label{eqn:sigma}
 \boldsymbol{\Sigma} = \operatorname{diag}_{i=1,\dots,d} \sigma_{i,i}, \quad \sigma_{i,i} \ge 0~~\forall\,i.
\end{equation}
Multiplying the matrix $\mathbf{G}$ \eqref{eqn:graph-laplacian-system} from the right by $\boldsymbol{\Sigma}$ scales each column by the non-negative diagonal entry:
\begin{equation}
\label{eqn:G-scaled}
\overline{\mathbf{G}} \coloneqq \mathbf{G} \cdot \boldsymbol{\Sigma}, \quad \overline{\mathbf{G}}_{i,j} = \mathbf{G}_{i,j}\,\sigma_{j,j}~~~\forall\, i,j.
\end{equation}
We see immediately that the entries of the scaled matrix \eqref{eqn:G-scaled} have the same signs as the entries of the original matrix \eqref{eqn:G-entries-sign},
\begin{equation}
\overline{\mathbf{G}}_{i,i}(t,\mathbf{y}) \le 0, ~~ \forall\, i; \quad 
\overline{\mathbf{G}}_{i,j}(t,\mathbf{y}) \ge 0, ~~\forall\, i \ne j. 
\end{equation}
Moreover, the scaled matrix \eqref{eqn:G-scaled} admits the same left kernel vectors as the original matrix \eqref{eqn:linear-invariants}, 
\begin{equation}
\mathbf{w}_j^T\,\overline{\mathbf{G}} = \mathbf{0} \quad \forall\,j.
\end{equation}
\end{subequations}
\end{remark}

\begin{remark}[Linear combinations of graph Laplacian matrices]
\label{rem:linear-combinations}
Let $\mathbf{G}_1, \dots, \mathbf{G}_s$ be matrices whose entries have the sign structure \eqref{eqn:G-entries-sign}, and that all have the same left kernel vectors \eqref{eqn:linear-invariants}. Let $b_1,\dots,b_s \ge 0$ be non-negative numbers. Then the linear combination
\begin{equation}
\label{eqn:G-linear-combination}
\overline{\mathbf{G}} = b_1\,\mathbf{G}_1 + \dots + b_s\,\mathbf{G}_s
\end{equation}
has the entry sign property \eqref{eqn:G-entries-sign} and shares the same left kernel vectors \eqref{eqn:linear-invariants} as the individual matrices. Consequently, a non-negative linear combination of graph Laplacian matrices \eqref{eqn:G-linear-combination} is itself a graph Laplacian matrix.
\end{remark}

Remarks \ref{rem:scaling} and \ref{rem:linear-combinations} lead to the following.

\begin{remark}[Linear combinations of scaled graph Laplacian matrices]
\label{rem:scaled-linear-combinations}
Let $\mathbf{G}_1, \dots, \mathbf{G}_s$ be graph Laplacian matrices, with the same left kernel vectors \eqref{eqn:linear-invariants}. Let $b_1,\dots,b_s \ge 0$ be non-negative numbers. Let $\boldsymbol{\Sigma}_1,\dots,\boldsymbol{\Sigma}s$ be diagonal scaling matrices with non-negative diagonal entries \eqref{eqn:sigma}. Then the non-negative linear combination of the scaled matrices
\begin{equation}
\label{eqn:G-scaled-linear-combination}
\overline{\mathbf{G}} = b_1\,\mathbf{G}_1\,\boldsymbol{\Sigma}_1 + \dots + b_s\,\mathbf{G}_s\,\boldsymbol{\Sigma}_s
\end{equation}
is itself a graph Laplacian matrix with the same left kernel vectors \eqref{eqn:linear-invariants}.
\end{remark}

\subsection{Production-destruction systems in graph Laplacian form}

We consider conservative production-destruction models \eqref{eqn:PDS}, where the production rates $p_{i, j}$ and destruction rates $d_{i,j}$ have the following properties:
\begin{subequations}
\label{eqn:PDS-properties}
\begin{alignat}{3}
\label{eqn:PDS-property-a}
p_{i,j}(\mathbf{y}) &\equiv d_{j,i}(\mathbf{y}), && \forall i,j=1,\dots,N, \quad \forall \mathbf{y} \in \mathbb{R}^d, \\
\label{eqn:PDS-property-b}
d_{i,j}(\mathbf{y}) &= \ell_{i,j}(\mathbf{y})\, \mathbf{y}_{i}, \quad&& \forall i,j=1,\dots,N, \quad \forall\, \mathbf{y} \in \mathbb{R}^d, \\
\label{eqn:PDS-property-c}
\ell_{i,j}(\mathbf{y}) &\ge 0, && \forall i,j=1,\dots,N, \quad \forall\, \mathbf{y} \succeq 0.
\end{alignat}
\end{subequations} 
The coefficients $\ell_{i,j}(\mathbf{y}) \ge 0$ are non-negative transition rates, meaning that destruction terms can never generate mass or become negative. In particular, \eqref{eqn:PDS-property-b} implies
\[
\mathbf{y}(t) \succeq 0,\,\mathbf{y}_i(t) = 0 \quad \xRightarrow{\eqref{eqn:PDS-property-b}} \quad d_{i,j}(\mathbf{y}) = 0,~\forall\,j 
\quad \xRightarrow{\eqref{eqn:pairwise-balance}} \quad D_{i}(\mathbf{y}) = 0,
\]
so a vanishing species cannot be further destroyed, ensuring positivity.

PDS systems \eqref{eqn:PDS-properties} have the detailed form \cite{Blanes_2022_positivity}:
\begin{equation}
\label{eqn:production-destruction}
\begin{split}
\mathbf{y}'_i & = \sum_{j=1}^N p_{i,j}(\mathbf{y}) - \sum_{j=1}^N d_{i,j}(\mathbf{y}) \\
&\xlongequal{\eqref{eqn:PDS-property-a}} \sum_{j=1}^N d_{j,i}(\mathbf{y}) - \sum_{j=1}^N d_{i,j}(\mathbf{y}) \\
&\xlongequal{\eqref{eqn:PDS-property-b}} \sum_{j=1}^N \ell_{j,i}(\mathbf{y})\, \mathbf{y}_{j} - \sum_{j=1}^N \ell_{i,j}(\mathbf{y})\, \mathbf{y}_{i}.
\end{split}
\end{equation}
Using the matrix of non-negative transition rates 
\[
\mathbf{L}(\mathbf{y}) \coloneqq \big[ \ell_{i,j}(\mathbf{y}) \big]_{1 \le i,j \le N}
\]
the system \eqref{eqn:production-destruction} can be written in matrix form as
\begin{equation}
\label{eqn:pd-G-definition}
\begin{split}
\mathbf{y}' &= \mathbf{L}^T\, \mathbf{y} - \textnormal{diag}(\mathbf{L}\,\boldsymbol{1}_{d})\, \mathbf{y}
= \big( \mathbf{L}^T - \textnormal{diag}(\mathbf{L}\,\boldsymbol{1}_{d}) \big)\, \mathbf{y} \\
 &= \mathbf{G}(t,\mathbf{y})\, \mathbf{y} \quad \text{with} \quad \mathbf{G}(t,\mathbf{y}) \coloneqq \mathbf{L}^T - \textnormal{diag}(\mathbf{L}\,\boldsymbol{1}_{d}).
\end{split}
\end{equation}
From the definition \eqref{eqn:pd-G-definition} the matrix $\mathbf{G}$ entries have the following sign properties:
\begin{subequations}
\label{eqn:pd-G-properties}
\begin{eqnarray}
\label{eqn:pd-G-ii}
\mathbf{G}_{i,i} &=& \big( \mathbf{L}^T- \textnormal{diag}(\mathbf{L}\,\boldsymbol{1}_{d}) \big)_{i,i} = \ell_{i,i} - \sum_{j=1}^N \ell_{i,j} \\
\nonumber
&=& - \sum_{j \ne i} \ell_{i,j} \le 0, \quad \forall\,i,\\
\label{eqn:pd-G-ij}
\mathbf{G}_{i,j} &=& \ell_{j,i} \ge 0, \quad \forall\,j \ne i.
\end{eqnarray}

Thus, all diagonal entries are non-positive because they represent the total destruction rate of species $i$ taken with a negative sign, while all off-diagonal entries are non-negative because they correspond to production rates from other species into $i$.
Hence, the sign structure in Equation \eqref{eqn:G-entries-sign} is satisfied.

Moreover,
\begin{eqnarray}
\sum_{i=1}^N \mathbf{G}_{i,j} &=& \mathbf{G}_{i,i} + \sum_{i \ne j} \mathbf{G}_{i,j} \\
&\stackrel{\eqref{eqn:pd-G-properties}}{=} & - \sum_{j \ne i} \ell_{i,j} + \sum_{i \ne j} \ell_{j,i} = 0, \\
 \Leftrightarrow && \quad 
\boldsymbol{1}_{d}^T\,\mathbf{G} \equiv \boldsymbol{0}_{d}^T.
\end{eqnarray}
\end{subequations}
Therefore, \(\mathbf{G}(t,\mathbf{y})\) satisfies the zero column sum property of Remark \ref{rem:graph-laplacian}.
So, by the argument in Equation \eqref{eqn:positive-solution}, no negative values can appear in the exact solution since:

\[
\mathbf{y}(t) \succeq 0,\,\mathbf{y}_i(t) = 0  \quad \Rightarrow \quad \mathbf{y}'_i(t) = \sum_{j \neq i} \mathbf{G}_{i,j}\,\mathbf{y}_j \ge 0.
\] 

By Assumptions \ref{ass:invariants} and \eqref{eqn:invariant}, the total mass is conserved:
\[
\frac{\mathrm{d}}{\mathrm{d}t} \big( \boldsymbol{1}_d^T \mathbf{y}(t) \big) 
= \boldsymbol{1}_d^T \mathbf{y}'(t) 
= \boldsymbol{1}_d^T\, \mathbf{G}(t,\mathbf{y}) \, \mathbf{y}(t) =  0.
\]

\begin{remark}[Other linear invariants] 
\label{rem:invariants-other}
The system \eqref{eqn:pd-G-definition} can also admit other linear invariants \eqref{eqn:linear-invariants}.
Specifically, consider a vector $\mathbf{w} \in \mathbb{R}^d$ such that:
\begin{equation}
\label{eqn:linear-invariants-general}
\begin{split}
& \left( \sum_{j=1}^N  \ell_{i,j}\,\mathbf{w}_j \right)
= \mathbf{w}_i\,\sum_{j=1}^N  \ell_{i,j}, ~~ \forall i \quad 
\xLeftrightarrow{\eqref{eqn:production-destruction}} \quad
\mathbf{w}^T\, \mathbf{G} = 0,
\end{split}
\end{equation}
which implies that $\mathbf{w}^T \,\mathbf{y}(t) = $const. Clearly \eqref{eqn:linear-invariants-general} holds for $\mathbf{w} = \boldsymbol{1}_d$, but the system \eqref{eqn:pd-G-definition}  may have additional linear invariants $\mathbf{w}$.
\end{remark}

\begin{remark}[All linear invariants] 
\label{rem:invariants-all}
Clearly, all left null vectors of the matrix $\mathbf{G}$ are linear invariants of \eqref{eqn:graph-laplacian-system}:
\[
\mathbf{w}^T\, \mathbf{G}(t,\mathbf{y}) = 0~~\forall t,\mathbf{y} \quad \xRightarrow{\eqref{eqn:graph-laplacian-system}} \quad 
\mathbf{w}^T\,\mathbf{y}' = \mathbf{w}^T\,\mathbf{G}(t,\mathbf{y})\,\mathbf{y} = 0.
\]
However, not all linear invariants of the nonlinear system of \eqref{eqn:graph-laplacian-system} are necessarily left null vectors of $\mathbf{G}$
\[
\mathbf{w}^T\,\mathbf{G}(t,\mathbf{y}(t))\,\mathbf{y}(t) = 0 ~~\text{along solution}~~\mathbf{y}(t) 
\quad \nRightarrow \quad \mathbf{w}^T\, \mathbf{G}(t,\mathbf{y}) = 0~~\forall\, t,\mathbf{y},
\]
i.e., the nonlinear system \eqref{eqn:graph-laplacian-system} can have additional linear invariants that may not be preserved by the methods discussed herein.
\end{remark}

Thus, production-destruction systems \eqref{eqn:production-destruction} are a special case of nonlinear graph Laplacian systems that satisfy the positivity and conservation properties required for Assumption \ref{ass:positivity}.

\begin{remark}[A different form of production-destruction]
\label{rem:D-formulation}
Using the matrix of non-negative destruction rates 
\[
\mathbf{D}(\mathbf{y}) \coloneqq \big[ d_{i,j}(\mathbf{y}) \big]_{1 \le i,j \le N}
\]
the system \eqref{eqn:production-destruction} can be written in matrix form as
\begin{equation}
\label{eqn:pd-D-definition}
\begin{split}
\mathbf{y}' &= \mathbf{H}(\mathbf{y})\, \boldsymbol{1}_{d} \quad \text{with} \quad \mathbf{H}(\mathbf{y}) \coloneqq  \mathbf{D}^T(\mathbf{y}) - \textnormal{diag}(\mathbf{D}(\mathbf{y})\,\boldsymbol{1}_{d}).
\end{split}
\end{equation}
Note that $\mathbf{H}(\mathbf{y})$ is also a graph Laplacian matrix:
\[
\boldsymbol{1}_{d}^T\,\mathbf{H}(\mathbf{y}) \equiv \boldsymbol{0} \quad \Rightarrow \quad \boldsymbol{1}_{d}^T\, \mathbf{y}(t) = \textnormal{const},
\]
and:
\[
\mathbf{H}_{i,i}(\mathbf{y}) = - \sum_{j \ne i} d_{i,j}(\mathbf{y}) \le 0,~~ \forall\,i; \quad
\mathbf{H}_{i,j}(\mathbf{y}) = d_{j,i}(\mathbf{y}) \ge 0,~~ \forall\,j \ne i.
\]
\end{remark}

\subsection{The new positivity preserving predictor-corrector Runge-Kutta methods}
\label{subsec:DIRK-solution-PC}

We now seek to solve numerically the system \eqref{eqn:graph-laplacian-system} that satisfies Assumption \ref{ass:positivity} (or the stronger Assumption \ref{ass:strong-positivity}). Assume that we have computed a numerical approximation $\mathbf{y}_n \succeq 0$ for the exact solution $\mathbf{y}(t_n) \succeq 0$.

\subsubsection{Clipping and scaling operations}

The next step solution is computed by a SDIRK method with only non-negative weights $b_i \ge 0$ as follows:
\begin{subequations}
\label{eqn:SDIRK}
\begin{eqnarray}
\label{eqn:SDIRK-stage}
\mathbf{Y}_i & = & \mathbf{y}_n + h\,\sum_{j=1}^i\,a_{i,j}\,\mathbf{G}(\mathbf{Y}_j)\, \mathbf{Y}_j, \quad i=1,\dots,s; \\
\label{eqn:SDIRK-solution}
\mathbf{y}_{n+1} & = & \mathbf{y}_n + h\,\sum_{j=1}^s\,b_{j}\,\mathbf{G}(\mathbf{Y}_j)\, \mathbf{Y}_j.
\end{eqnarray}
\end{subequations}
We consider methods \eqref{eqn:SDIRK} that are stiffly accurate, meaning that the last stage is the next step solution, i.e., $b_{j} = a_{s,j} \ge 0$ for all $j$ and $\mathbf{y}_{n+1} = \mathbf{Y}_{s}$.

We start by defining special clipping and scaling operations.

\begin{definition}[Clipping]
\label{def:clipping}
Let $\mathbf{Y} \in \mathbb{R}^d$. The following clipping operation sets the negative entries of a vector to zero: 

\begin{equation}
\label{eqn:SDIRK-clipping}
\breve{\mathbf{Y}}_\ell \coloneqq \left(\textnormal{clip}(\mathbf{Y})\right)_{\ell} \coloneqq \max(\mathbf{Y}_\ell,0), 
\quad \ell=1,\dots,d.
\end{equation}

\end{definition}
\begin{lemma}
\label{lemma:clipping-errors}
For all monotonic norms \(\Vert \cdot \Vert\)  \cite{bhatiaMatrixAnalysis1997}, e.g p-norms, then clipping does \textbf{not} increase the local errors of stages and solution constructed from \eqref{eqn:SDIRK}.
\end{lemma}

\begin{proof}

Consider an exact \emph{local} solution $u(t)$ of the system \eqref{eqn:graph-laplacian-system} started from the numerical value $u(t_n) = \mathbf{y}_n \succeq 0$. The properties of the exact solution imply that $u(t) \succeq 0$ for $t_{n} \le t\le t_{n+1}$.

Consider a component \(\ell\) such that \((\mathbf{Y}_i)_{\ell} \le 0\). Since the SDIRK scheme \eqref{eqn:SDIRK} has stage order $q=1$ and the local exact solution is always non-negative we have
\[
-(\mathbf{Y}_i)_{\ell} \le u_{\ell}(t_{n} + c_i h) - (\mathbf{Y}_i)_{\ell} = \mathcal{O}(h^{q+1}).
\]
Similarly, if the solution has a non-positive entry \((\mathbf{y}_{n+1})_{\ell} \le 0\), using the order $p$ of scheme \eqref{eqn:SDIRK} yields
\[
-(\mathbf{y}_{n+1})_{\ell} \le u_{\ell}(t_{n+1} ) - (\mathbf{y}_{n+1})_{\ell}= \mathcal{O}(h^{p+1}).
\]
Given that the norm is monotone and the negative parts of the solution entries are bounded we have
\begin{equation}
    \begin{split}
        \Vert \breve{\mathbf{Y}}_i - \mathbf{Y}_i\Vert &= \mathcal{O}(h^{q+1}),\\
        \Vert \breve{\mathbf{y}}_{n+1} - \mathbf{y}_{n+1}\Vert &= \mathcal{O}(h^{p+1}).
    \end{split}
\end{equation}
This gives the errors of the clipped solutions:
\begin{equation}
    \begin{split}
        \Vert \breve{\mathbf{Y}}_i - u(t_n + c_i\,h)\Vert &\le \Vert \mathbf{Y}_i - u(t_n + c_i\,h)\Vert + \Vert \breve{\mathbf{Y}}_i - \mathbf{Y}_i\Vert = \mathcal{O}(h^{q+1}),\\
        \Vert \breve{\mathbf{y}}_{n+1} - u(t_n + \,h)\Vert &\le \Vert \mathbf{y}_{n+1} - u(t_n + \,h)\Vert + \Vert \breve{\mathbf{y}}_{n+1} - \mathbf{y}_{n+1}\Vert = \mathcal{O}(h^{p+1}).
    \end{split}
\end{equation}
\end{proof}

Figure \ref{fig:clipping} illustrates the effect of clipping in 2D. 
Consider the points black (for exact solution), red (for numerical solution), and blue (for clipped solution). Clipping changes only component $(\mathbf{Y}_i)_{\ell}$ but leaves other components $(\mathbf{Y}_i)_{k}$ untouched, and therefore the red-blue edge is parallel to the $\ell$ axis. Since $u_\ell, u_k \ge 0$, the triangle with black-red-blue vertices has an obtuse angle at  the blue vertex. Therefore the red edge (local error of the original solution) is longer in Euclidian norm than the blue edge (local error of the clipped solution.)

\begin{figure}[htbp]
\centering
\begin{tikzpicture}[scale=1]
\draw[thick, ->] (-0.5,0) -- (4.5,0);
\draw[thick, ->] (0,-1.5) -- (0,3);
\draw [thick,red, -] (1.5,1.5) -- (2.9,-1);
\draw [thick,blue, -] (1.5,1.5) -- (2.9,0);
\draw [thick,purple, ->] (2.85,-1) -- (2.85,-0.1);
\draw (0,3) node[above] {$\mathbf{y}_\ell$};
\draw (4.5,0) node[right] {$\mathbf{y}_k$};
\draw (1.5,1.8) node[circle] {$u(t_n + c_i\,h)$};
\draw (1.3,0.3) node[circle,red] {$\mathcal{O}(h^{q+1})$};
\draw (1.5,1.5) node[circle,fill=black,inner sep=2.5pt] {};
\draw (3,-1) node[circle,left,fill=red,inner sep=2.5pt] {};
\draw (3,-1) node[circle,below] {$\mathbf{Y}_i$};
\draw (3,0) node[circle,left,fill=blue,inner sep=2.5pt] {};
\draw (3,0) node[circle,above] {$\breve{\mathbf{Y}}_i$};
\draw (3.7,-0.5) node[circle,purple] {$\mathcal{O}(h^{q+1})$};
\end{tikzpicture}
\caption{Graphical representation of Lemma \ref{lemma:clipping-errors} proof demonstrates the local truncation errors of the stages of the Runge-Kutta do not increase after clipping is applied.}
\label{fig:clipping}
\end{figure}

\begin{definition}[Ratio scaling]
\label{def:ratio-scaling}
Let $\mathbf{Y}, \, \mathbf{Z} \in \mathbb{R}^d$.
The ratio scaling matrix is a positive semidefinite diagonal matrix, whose diagonal entries are the result of vector component-wise division operation:
\begin{equation}
\label{eqn:ratio-scaling}
\boldsymbol{\Sigma}_{\mathbf{Y} / \mathbf{Z}} \coloneqq \underset{\ell=1,\dots,d}{\operatorname{diag}}\, \frac{\max(\mathbf{Y}_\ell,0)}{\max(\mathbf{Z}_\ell,\varepsilon)} \in \mathbb{R}^{d \times d}.
\end{equation}
The small fixed number $\varepsilon > 0$ in denominator is needed to avoid division by a number too close to zero.
\end{definition}

\begin{remark}[Application of ratio scaling]
\label{rem:ratio-scaling-applied}
We will use Definition \ref{def:ratio-scaling} with scaling vectors $\mathbf{Y}$ and $\mathbf{Z}$ that are two numerical solutions in the interval $[t_n,t_{n+1}]$, and consequently $\mathbf{Z} = \mathbf{Y} + \mathcal{O}(h)$.

We assume that a single component, $\ell$, take values close to zero during the current time step. Numerical stage and solution values have components $\ell$ close to zero, $\mathbf{Y}_\ell \sim 0$ and $\mathbf{Z}_\ell \sim 0$. These numerical values can become negative during the step and need special attention.

Let
\begin{equation}
\widehat{\mathbf{Y}}_\ell = \left(\boldsymbol{\Sigma}_{\mathbf{Y} / \mathbf{Z}}\,\breve{\mathbf{Z}}\right)_{\ell} = 
\begin{cases}
\breve{\mathbf{Y}}_\ell, & \mathbf{Z}_\ell \ge \varepsilon, \\
\breve{\mathbf{Y}}_\ell\, \frac{\mathbf{Z}_\ell}{\varepsilon}, & 0 < \mathbf{Z}_\ell < \varepsilon, \\
0, & \mathbf{Z}_\ell \le 0.
\end{cases}
\end{equation}
The application of scaling on the same vector has the following effect.
\begin{enumerate}
\item For components that are not close to zero
$\mathbf{Z}_\ell > \varepsilon$ and $\widehat{\mathbf{Y}}_\ell = \breve{\mathbf{Y}}_\ell = \mathbf{Y}_\ell$; such components are not changed.

\item For components $\ell$ that are close to zero, if $\mathbf{Z}_\ell \le 0$ then $\widehat{\mathbf{Y}}_\ell = 0$.

\item If both components are small and positive, $0 < \mathbf{Z}_\ell < \varepsilon$ and $\mathbf{Y}_\ell < \varepsilon + \mathcal{O}(h)$,
then the output is a scaled value of the small $\mathbf{Y}_\ell$ component, with a multiplier smaller than one:
\[
0 \le \frac{\mathbf{Z}_\ell}{\varepsilon} < 1, \quad
\widehat{\mathbf{Y}}_\ell =  \frac{\mathbf{Z}_\ell}{\varepsilon}\,\mathbf{Y}_\ell,
\]
and
\[
\widehat{\mathbf{Y}}_\ell - \breve{\mathbf{Y}}_\ell = \left(1 - \frac{\mathbf{Z}_\ell}{\varepsilon}\right)\,\breve{\mathbf{Y}}_\ell
\le \breve{\mathbf{Y}}_\ell  =  \mathbf{Y}_\ell\,(1+\mathcal{O}(h^{q+1}))
= \mathcal{O}(\varepsilon) + \mathcal{O}(h).
\]
\end{enumerate}
\end{remark}

\subsubsection{Predictor-corrector approach}

The proposed predictor-corrector algorithm proceeds as follows.

\paragraph{Predictor step} 
Consider a SDIRK method \eqref{eqn:SDIRK} with only non-negative weights $b_i \ge 0$ and appy it to solve \eqref{eqn:graph-laplacian-system} and obtain a ``predicted'' solution $\mathbf{y}^{\{p\}}_{n+1}$:
\begin{subequations}
\label{eqn:SDIRK-predictor}
\begin{eqnarray}
\label{eqn:SDIRK-predictor-stage}
\mathbf{Y}_i & = & \mathbf{y}_n + h\,\sum_{j=1}^i\,a_{i,j}\,\mathbf{G}(\mathbf{Y}_j)\, \mathbf{Y}_j, \quad i=1,\dots,s; \\
\label{eqn:SDIRK-predictor-solution}
\mathbf{y}^{\{p\}}_{n+1} & = & \mathbf{y}_n + h\,\sum_{j=1}^s\,b_{j}\,\mathbf{G}(\mathbf{Y}_j)\, \mathbf{Y}_j.
\end{eqnarray}
\end{subequations}
The computed stage values $\mathbf{Y}_i$, as well as the new predicted solution $\mathbf{y}^{\{ p \} }_{n+1} $, may have negative entries.

\begin{remark}
If in \eqref{eqn:SDIRK-predictor} all stages $\mathbf{Y}_j \succeq 0$ have non-negative entries, the scheme is stiffly accurate, and the solution $\mathbf{y}^{\{p\}}_{n+1} \succeq \varepsilon$ has entries greater than the threshold, then the predictor solution \eqref{eqn:SDIRK-predictor-solution} reads:
\begin{equation}
\begin{split}
\mathbf{y}^{\{p\}}_{n+1} & = \mathbf{y}_n + h\,\left(\sum_{j=1}^{s-1}\,b_{j}\,\mathbf{G}(\mathbf{Y}_j)\, \boldsymbol{\Sigma}_{\mathbf{Y}_j /\mathbf{y}^{\{ p \} }_{n+1} } + b_{s}\,\mathbf{G}(\mathbf{y}^{\{p\}}_{n+1})\right)\, \mathbf{y}^{\{p\}}_{n+1} \\
  & = \left(\mathbf{I}_{d \times d} - h\,\Bar{\Bar{\mathbf{G}}}\right)^{-1}\,\mathbf{y}_n,\\
 \Bar{\Bar{\mathbf{G}}} &:= \sum_{j=1}^{s-1}\,b_{j}\,\mathbf{G}(\mathbf{Y}_j)\, \boldsymbol{\Sigma}_{\mathbf{Y}_j /\mathbf{y}^{\{ p \} }_{n+1} } + b_{s}\,\mathbf{G}(\mathbf{y}^{\{p\}}_{n+1}).
\end{split}
\end{equation}
Under these assumptions $\boldsymbol{\Sigma}_{\mathbf{Y}_j /\mathbf{y}^{\{ p \}}_{n+1} }$ are matrices of non-negative weights, and $\Bar{\Bar{\mathbf{G}}}$ is a graph Laplacian matrix per Corollary \ref{rem:scaled-linear-combinations}.
\end{remark}

The goal of the corrector is to construct a modified matrix $\overline{\mathbf{G}}$ that satisfies graph Laplacian assumptions even when some stage or solution entries are negative, and compute a corrected solution $\left(\mathbf{I}_{d \times d} - h\,\overline{\mathbf{G}}\right)^{-1}\,\mathbf{y}_n$.

\paragraph{Corrector step} 
The predictor intermediate stage values $\mathbf{Y}_j$ and the solution $\mathbf{y}^{\{ p \}}_{n+1}$ may contain negative components.
In this case the nonlinear matrix $\mathbf{G}(\mathbf{Y}_j)$ may loose its graph Laplacian properties. which are not guaranteed to remain positive. To correct this, each stage argument is replaced by its clipped version, such that each $\mathbf{G}(\breve{\mathbf{Y}}_j)$ satisfies the required sign pattern.

We compute a new, positive solution as follows:
\begin{subequations}
\label{eqn:SDIRK-corrector}
\begin{equation}
\label{eqn:SDIRK-corrector-solution}
\begin{split}
\breve{\mathbf{Y}}_j & = \textnormal{clip}(\mathbf{Y}_j) \quad \textnormal{(clipped version of stage computed by \eqref{eqn:SDIRK-predictor})} \\
\mathbf{y}_{n+1} & = \mathbf{y}_n + h\, \left(  \sum_{j=1}^s\,b_{j}\,\mathbf{G}(\breve{\mathbf{Y}}_j)\, \boldsymbol{\Sigma}_{\mathbf{Y}_j / \mathbf{y}^{\{ p \}}_{n+1} } \right) \, \mathbf{y}_{n+1} \equiv \mathbf{y}_n + h\,\overline{\mathbf{G}}\,\mathbf{y}_{n+1},
\end{split}
\end{equation}
where the ``averaged'' $\mathbf{G}$ matrix is
\begin{equation}
\label{eqn:SDIRK-corrector-Gbar}
\overline{\mathbf{G}} = \sum_{j=1}^s\,b_{j}\,\mathbf{G}(\breve{\mathbf{Y}}_j)\, \boldsymbol{\Sigma}_{\mathbf{Y}_j / \mathbf{y}^{\{ p \}}_{n+1} }.
\end{equation}
Each $\mathbf{G}(\breve{\mathbf{Y}}_j)$ is evaluated at clipped arguments with non-negative entries, such that it is a graph Laplacian matrix per equations \eqref{eqn:PDS-property-c} and \eqref{eqn:pd-G-properties}. 

Each $\mathbf{G}(\breve{\mathbf{Y}}_j)$ is scaled by a diagonal matrix with non-negative entries $\boldsymbol{\Sigma}_{\mathbf{Y}_j / \mathbf{y}^{\{ p \} }_{n+1}}$, applied from the right.

Using Remark \ref{rem:scaled-linear-combinations}, the averaged matrix $\overline{\mathbf{G}}$ \eqref{eqn:SDIRK-corrector-Gbar} is itself a graph Laplacian matrix with the same left kernel vectors \eqref{eqn:linear-invariants}.
Consequently the matrix $\overline{\mathbf{G}}$ satisfies the assumptions posed in \eqref{eqn:G-entries-sign} and \eqref{eqn:linear-invariants} that ensure positivity and conservation.

The corrector solution \eqref{eqn:SDIRK-corrector} is computed via a linear system solution:
\begin{equation}
\label{eqn:SDIRK-corrector-formula}
\mathbf{y}_{n+1} = \left( \mathbf{I}_{d \times d} - h\,\overline{\mathbf{G}} \right)^{-1}\, \mathbf{y}_{n}.
\end{equation}
The accuracy of  \(\mathbf{y}_{n+1}\) produced by the predictor-corrector approach is discussed in Section \ref{subsec:Patankar-accuracy}.
Since $\overline{\mathbf{G}}$ satisfies Assumption~\ref{ass:stability}, the matrix $\mathbf{I}_{d \times d} - h\,\overline{\mathbf{G}}$ is an $M$-matrix (positive diagonals, non-positive off-diagonals, eigenvalues in the open right half-plane). From classical $M$-matrix theory its inverse has only non-negative entries,
    \[
    (\mathbf{I}_{d \times d} - h\,\overline{\mathbf{G}})^{-1} \succeq 0,
    \]
and consequently the solution \eqref{eqn:SDIRK-corrector-formula} is non-negative:
\[
\mathbf{y}_{n} \succeq 0 \quad \Rightarrow \quad \mathbf{y}_{n+1} \succeq 0.
\]
Additionally, the inverse of the matrix \((\mathbf{I}_{d \times d} - h\,\overline{\mathbf{G}})^{-1}\) is nonsingular and therefore each row has at least one strictly positive entry; this leads to the strict positivity property 
\begin{equation}
\label{eq:strict-positivity}
\mathbf{y}_{n} \succ 0 \quad \Rightarrow \quad \mathbf{y}_{n+1} \succ 0.
\end{equation}
Moreover, solution \eqref{eqn:SDIRK-corrector-formula} preserves all linear invariants $\mathbf{w}_j$ of the system: 
\begin{equation}
\begin{split}
& \mathbf{w}_j^T\,\mathbf{G}(\cdot) \equiv 0 ~~\Rightarrow ~~\mathbf{w}_j^T\,\overline{\mathbf{G}} = 0 \\
& \Rightarrow ~~ \mathbf{w}_j^T\,\mathbf{y}_{n+1} = \mathbf{w}_j^T\,\left( \mathbf{I}_{d \times d} - h\,\overline{\mathbf{G}} \right)\,\mathbf{y}_{n+1} = \mathbf{w}_j^T\,\mathbf{y}_{n}.
\end{split}
\end{equation}
\end{subequations}

\begin{remark}[Strong sign assumption]
For systems satisfying the strong sign entries Assumption \ref{ass:strong-positivity}, the matrices  $\mathbf{G}(\mathbf{Y}_j)$ evaluated at arguments with possible small negative entries are still graph Laplacian matrices.  Arguments need not be clipped to preserve the sign property, and the averaged matrix \eqref{eqn:SDIRK-corrector-Gbar} used in the corrector  takes the simpler form:
\begin{equation}
\overline{\mathbf{G}} = \sum_{j=1}^s\,b_{j}\,\mathbf{G}(\mathbf{Y}_j)\, \boldsymbol{\Sigma}_{\mathbf{Y}_j / \mathbf{y}^{\{ p \}}_{n+1} }.
\end{equation}
\end{remark}

\begin{remark}[More general form]
For systems in the more general form \eqref{eqn:pd-D-definition}
the predictor step \eqref{eqn:SDIRK-predictor-solution} reads
\[
\mathbf{y}^{\{p\}}_{n+1} = \mathbf{y}_n + h\,\left( \sum_{j=1}^s\,b_{j}\,\mathbf{H}(\mathbf{Y}_j) \right)\,\boldsymbol{1}_{d},
\]
and the corrector step \eqref{eqn:SDIRK-corrector-solution} is
\[
\begin{split}
\mathbf{y}_{n+1} & = \mathbf{y}_n + h\, \left(  \sum_{j=1}^s\,b_{j}\,\mathbf{H}(\breve{\mathbf{Y}}_j)\, \boldsymbol{\Sigma}_{\boldsymbol{1}_{d} / \mathbf{y}^{\{ p \}}_{n+1} } \right) \, \mathbf{y}_{n+1} \\
& = \mathbf{y}_n + h\,\overline{\mathbf{H}}\,\mathbf{y}_{n+1}.
\end{split}
\]
\end{remark}

\begin{remark}[Adaptive integration and positivity corrections]
\label{rem:adaptive-clipping}
In practical computations, the underlying SDIRK methods are used together with an embedded scheme solution $\hat{\mathbf{y}}_{n+1}$ for adaptive step-size control.
In the proposed framework, the embedded error estimate $LTE$ uses the uncorrected predictor solution  produced by the base SDIRK method, i.e., $LTE \coloneqq \mathbf{y}^{\{p\}}_{n+1} - \hat{\mathbf{y}}_{n+1}$.

The positivity-preserving correction (clipping and scaling) is applied only to the solution returned by the time step, and does not enter the error estimation procedure.
Consequently, the adaptive step-size controller is unaffected by the correction mechanism, while the corrected solution $\mathbf{y}_{n+1}$ is guaranteed to be non-negative.

Since the difference between the predictor and corrected solutions is of higher order (Lemma \ref{lemma:clipping-errors}), using the predictor solution for local error control remains effective in practice.

\end{remark}

\begin{remark}[More conservative error estimation]
\label{rem:alternative-adaptive-clipping}
Alternatively, the error estimate for step size control can include the positivity correction magnitude $\|\mathbf{y}_{n+1} - \mathbf{y}^{\{p\}}_{n+1}\|$:
\[
LTE \coloneqq \|\mathbf{y}^{\{p\}}_{n+1} - \hat{\mathbf{y}}_{n+1}\| + \|\mathbf{y}_{n+1} - \mathbf{y}^{\{p\}}_{n+1}\| \ge \|\mathbf{y}_{n+1} - \hat{\mathbf{y}}_{n+1}\|.
\]
Note that this is an upper bound for the overall local truncation error, and its use ensures that the corrections remain small throughout the simulation.
\end{remark}

\subsection{Patankar stage predictor-corrector DIRK methods}
\label{subsec:DIRK-stage-PC}

We now apply the same approach to make all stage vectors non-negative.
Consider a stiffly accurate (S)DIRK method with non-negative coefficients $a_{i,j} \ge 0$ and weights $b_i \ge 0$ applied to solve \eqref{eqn:graph-laplacian-system}:
\begin{equation}
\label{eqn:DIRK}
\begin{split}
\mathbf{Y}_i & = \mathbf{y}_n + h\,\sum_{j=1}^i\,a_{i,j}\,\mathbf{G}(\mathbf{Y}_j)\, \mathbf{Y}_j, \quad i=1,\dots,s; \\
\mathbf{y}_{n+1} & = \mathbf{Y}_s = \mathbf{y}_n + h\,\sum_{j=1}^s\,a_{s,j}\,\mathbf{G}(\mathbf{Y}_j)\, \mathbf{Y}_j \\
&\equiv \mathbf{y}_n + h\,\sum_{j=1}^s\,b_{j}\,\mathbf{G}(\mathbf{Y}_j)\, \mathbf{Y}_j.
\end{split}
\end{equation}

\begin{enumerate}
\item {\it Stage predictor step.} We apply the regular DIRK stage \eqref{eqn:DIRK}: 
\[
\begin{split}
\mathbf{Y}^{\{ p \} }_i & = \mathbf{y}_n + h\,\sum_{j=1}^{i-1}\,a_{i,j}\,\mathbf{G}(\mathbf{Y}_j)\, \mathbf{Y}_j + h\,a_{i,i}\,\mathbf{G}(\mathbf{Y}^{\{ p \} }_i)\, \mathbf{Y}^{\{ p \} }_i.
\end{split}
\]
However, the computed predictor stage values $\mathbf{Y}^{\{ p \} }_i$ may have negative entries.
We note that, if all predicted entries are positive $\mathbf{Y}^{\{ p \} }_i \succeq \varepsilon \boldsymbol{1}_{d}$, then:
\[
\begin{split}
\mathbf{Y}^{\{ p \} }_i & = \mathbf{y}_n + h\,\left( \sum_{j=1}^{i-1}\,a_{i,j}\,\mathbf{G}(\mathbf{Y}_j)\, \boldsymbol{\Sigma}_{\mathbf{Y}_j / \mathbf{Y}^{\{ p \} }_i} + h\,a_{i,i}\,\mathbf{G}(\mathbf{Y}^{\{ p \} }_i)\, \right) \mathbf{Y}^{\{ p \} }_i.
\end{split}
\]
\item {\it Stage corrector step.} We compute a new, positive stage vector as follows:
\begin{equation}
\label{eqn:DIRK-corrector}
\begin{split}
\mathbf{Y}_i &= \mathbf{y}_n + h\, \left( \sum_{j=1}^{i-1}\,a_{i,j}\,\mathbf{G}(\mathbf{Y}_j)\, \boldsymbol{\Sigma}_{\mathbf{Y}_j / \breve{\mathbf{Y}}^{\{ p \} }_i} 
+ h\,a_{i,i}\,\mathbf{G}( \breve{\mathbf{Y}}^{\{ p \} }_i ) \right)\,\mathbf{Y}_i \\
&\equiv \mathbf{y}_n + h\, \overline{\mathbf{G}}_i\, \mathbf{Y}_i.
\end{split}
\end{equation}
By the same reasoning as before, $\overline{\mathbf{G}}_i$ is a graph Laplacian matrix and each corrected stage $\mathbf{Y}_i \succeq \mathbf{0}$ is a conservative and positive vector.
\item Using the stiff accuracy property, the new solution is the last corrected stage $\mathbf{y}_{n+1}=\mathbf{Y}_s$.
\end{enumerate}

\begin{remark}[Additional LU Factorization for Correction.]
\label{rem:AdditionalLU}
The post positivity correction step involves a linear solve with an additional LU factorization of the \(\bar{\mathbf{G}}\)  (\ref{eqn:SDIRK-corrector-Gbar}). This adds a non-trivial overhead to the computational cost since this is a different LU factorization than what the nonlinear solver uses. However, note that the correction avoids stiff dynamics in non-physical negative regimes that require very small integration steps. In this case the correction can decrease the computational effort despite the higher cost per-step, as seen here in the Roberston's example.

Note that MPRK schemes start with an explicit RK method and use one LU factorization per stage for the positive solution computation. Here we start with a SDIRK scheme that has it's own computation cost and we use one extra LU factorization per step for the positive correction. This correction can be viewed as an extra MPRK type stage.

\end{remark}

\paragraph{Example}

Consider the second order SDIRK2 method \cite{wanner1996solving}: 
\begin{equation}
\label{eqn:SDIRK2}
\begin{split}
\mathbf{Y}_1 & = \mathbf{y}_n + h\,\gamma\,\mathbf{G}(\mathbf{Y}_1)\, \mathbf{Y}_1 \\
\mathbf{Y}_2 & = \mathbf{y}_n + h\,(1-\gamma)\,\mathbf{G}(\mathbf{Y}_1)\, \mathbf{Y}_1 + h\,\gamma\,\mathbf{G}(\mathbf{Y}_2)\, \mathbf{Y}_2 \\
\mathbf{y}_{n+1} & = \mathbf{Y}_2,
\end{split}
\end{equation}
with $\gamma = 1 - \frac{\sqrt{2}}{2}$ and where $\mathbf{Y}_1$ and $\mathbf{Y}_2$ are the stages of SDIRK2 method. 

\begin{enumerate}

\item \textit{Predictor step.} The regular SDIRK2 stages are computed as follows \eqref{eqn:DIRK}:
\begin{equation}
\begin{split}
\mathbf{Y}_1^{\{p\}} & = \mathbf{y}_n + h\,\gamma\,\mathbf{G}(\mathbf{Y}_1^{\{p\}})\, \mathbf{Y}_1^{\{p\}} \\
\mathbf{Y}_2^{\{p\}} & = \mathbf{y}_n + h\,(1-\gamma)\,\mathbf{G}(\mathbf{Y}_1^{\{p\}})\, \mathbf{Y}_1^{\{p\}} + h\,\gamma\,\mathbf{G}(\mathbf{Y}_2^{\{p\}})\, \mathbf{Y}_2^{\{p\}},
\end{split}
\end{equation}
where the predicted stages $\mathbf{Y}_1^{\{p\}}$ and $\mathbf{Y}_2^{\{p\}}$ might contain negative components.

\item \textit{Corrector step.} To enforce positivity, we apply clipping \eqref{eqn:SDIRK-clipping} and scaling operations \eqref{eqn:DIRK-corrector}:
\begin{equation}
\begin{split}
\mathbf{Y}_1 & = \mathbf{y}_n + h\,\gamma\,\mathbf{G}(\breve{\mathbf{Y}}^{\{p\}}_1)\,\mathbf{Y}_1; \\
\mathbf{Y}_2 & = \mathbf{y}_n + h\,\left( (1-\gamma)\,\mathbf{G}(\mathbf{Y}_1)\,\boldsymbol{\Sigma}_{\mathbf{Y}_1/\breve{\mathbf{Y}}^{\{p\}}_2} + \gamma\,\mathbf{G}(\breve{\mathbf{Y}}^{\{p\}}_2) \right)\, \mathbf{Y}_2.
\end{split}
\end{equation}

Each corrected stage $\mathbf{Y}_1, \,\mathbf{Y}_2 \succeq \mathbf{0}$ is now guaranteed to have non-negative entries.

\item \textit{Final solution.} The corrected solution is obtained directly from the last stage:
\begin{equation}
\mathbf{y}_{n+1} = \mathbf{Y}_2.
\end{equation}
This ensures both positivity and conservation properties of the numerical method.

\end{enumerate}

\subsection{Accuracy considerations}
\label{subsec:Patankar-accuracy}

\begin{proposition}[Local truncation error of corrected solution]
Consider the predicted solution \eqref{eqn:SDIRK-predictor-solution} and the Patankar corrector \eqref{eqn:SDIRK-corrector}
started from the exact solution $\mathbf{y}_n = \mathbf{y}(t_n)$. Let $\varepsilon$ be the threshold factor used in the matrix scaling \eqref{def:ratio-scaling}. Then the local truncation error of the corrected solution is:
\begin{equation}
\label{eqn:patankar-local-order}
\mathbf{y}_{n+1}-\mathbf{y}(t_{n+1}) = \mathcal{O}(h^{\min(q+2,p+1)}) +  \mathcal{O}(\varepsilon).
\end{equation}
At each step one can select a threshold factor that depends on the current step size, i.e., $\varepsilon = C\,h^{\min(q+2,p+1)}$ for some constant $C$ independent of the step. With this choice the order of the Patankar corrected method is $\min(p,q+1)$.
\end{proposition}

\begin{proof}
Since the small negative numerical values are of the order of local truncation errors we have that
\[
\max \left( \mathbf{y}^{\{ p \}}_{n+1},\varepsilon\right) =  \mathbf{y}^{\{ p \}}_{n+1} + \mathcal{O}(h^{p+1}) + \mathcal{O}(\varepsilon).
\]
Using the solution formula \eqref{eqn:SDIRK-predictor-solution}:
\begin{subequations}
\label{eqn:accuracy}
\begin{eqnarray}
\label{eqn:accuracy-predictor}
\max \left( \mathbf{y}^{\{ p \}}_{n+1},\varepsilon\right) & = &   \mathbf{y}_n + h\,\sum_{j=1}^s\,b_{j}\,\mathbf{G}(\mathbf{Y}_j)\, \mathbf{Y}_j + \mathcal{O}(h^{p+1}) + \mathcal{O}(\varepsilon) \\
\nonumber
& = & \mathbf{y}_n + h\,\left( \sum_{j=1}^s\,b_{j}\,\mathbf{G}(\mathbf{Y}_j)\, \boldsymbol{\Sigma}_{\mathbf{Y}_j / \mathbf{y}^{\{ p \}}_{n+1} } \right) \,\max \left( \mathbf{y}^{\{ p \}}_{n+1},\varepsilon\right)  \\
\nonumber && + \mathcal{O}(h^{p+1}) + \mathcal{O}(\varepsilon) \\
\nonumber
& = & \mathbf{y}_n + h\,\left( \sum_{j=1}^s\,b_{j}\,\mathbf{G}(\breve{\mathbf{Y}}_j)\, \boldsymbol{\Sigma}_{\mathbf{Y}_j / \mathbf{y}^{\{ p \}}_{n+1} } \right) \,\max \left( \mathbf{y}^{\{ p \}}_{n+1},\varepsilon\right) \\
\nonumber && +  \mathcal{O}(h^{\min(p+1,q+2)}) + \mathcal{O}(\varepsilon),
\end{eqnarray}
and therefore, using \eqref{eqn:SDIRK-corrector-Gbar} and \eqref{eqn:SDIRK-corrector-formula}, we have
\begin{equation}
\begin{split}
\max \left( \mathbf{y}^{\{ p \}}_{n+1},\varepsilon\right) &= \left( \mathbf{I}_{d \times d} - h\,\overline{\mathbf{G}} \right)^{-1}\, \mathbf{y}_{n} + \mathcal{O}(h^{\min(p+1,q+2)}) + \mathcal{O}(\varepsilon) \\
&= \mathbf{y}_{n+1} + \mathcal{O}(h^{\min(p+1,q+2)}) + \mathcal{O}(\varepsilon).
\end{split}
\end{equation}

\end{subequations}
Since $\mathbf{y}^{\{p\}}_{n+1}-\mathbf{y}(t_{n+1}) =  \mathcal{O}(h^{p+1})$, we have that \eqref{eqn:patankar-local-order} holds.
\end{proof}

\section{Test problems and methods}
\label{sec:test_problems_and_methods}

To assess the proposed positivity-preserving correction strategy, we employed four representative problems: the Robertson reaction, the MAPK cascade, the stratospheric reaction system, and the Korteweg-De Vries (KdV) equation. Each of these models can be expressed in the general graph Laplacian form \eqref{eqn:graph-laplacian-system}
with $\mathbf{G}(t,\mathbf{y})$ satisfying \eqref{eqn:G-entries-sign} and \eqref{eqn:linear-invariants}. It is noted that the graph Laplacian forms for a problem are not unique, as discussed in \cite{Blanes_2022_positivity}. Different Laplacian forms can capture different subsets of the nonlinear system invariants per Remark \ref{rem:invariants-all}. Here we focus on specific factorizations of the problems and do not discuss how to optimize splittings for preserving linear invariants.

\subsection{Test Problems}

This subsection provides a brief description of each test problem used in the numerical experiments. The selected models span a range of stiffness levels, nonlinearities, and structural properties, including mass conservation and positivity, and are standard benchmarks in the literature for assessing structure-preserving time integration methods. For each problem, we summarize the model equations, initial conditions, and relevant invariants.

\subsubsection{Robertson Reaction System}

The classical Robertson reaction models a stiff three-species chemical network \cite{Blanes_2022_positivity,wanner1996solving}:
\[
A \rightarrow B, \quad 
B + B \rightarrow B + C \rightarrow A + C,
\]
where $A$, $B$, and $C$ denote the three chemical species. Their concentrations are collected in
$\mathbf{y}(t) = (\mathbf{y}_1,\mathbf{y}_2,\mathbf{y}_3)^\top$, with
\begin{equation}
\label{eqn:robertson-species}
A \equiv \mathbf{y}_1, \qquad B \equiv \mathbf{y}_2, \qquad C \equiv \mathbf{y}_3.
\end{equation}
This yields the following system of ODEs for the concentrations $ \mathbf{y} \in \mathbb{R}^3$:
\begin{subequations}
\label{eqn:Robertson}
\begin{align}
\mathbf{y}_1' &= -0.04\, \mathbf{y}_1 + 10^4\, \mathbf{y}_2\, \mathbf{y}_3, \\
\mathbf{y}_2' &= 0.04\, \mathbf{y}_1 - 10^4\, \mathbf{y}_2\, \mathbf{y}_3 - 3\cdot 10^7\, \mathbf{y}_2^2, \\
\mathbf{y}_3' &= 3\cdot 10^7\, \mathbf{y}_2^2.
\end{align}
\end{subequations}
This can be written compactly in the Laplacian form
\begin{equation}
\frac{\textnormal{d}}{\textnormal{d}t} \begin{bmatrix} \mathbf{y}_1 \\ \mathbf{y}_2 \\ \mathbf{y}_3 \end{bmatrix} =
\underbrace{\begin{bmatrix}
-0.04 & 10^4\,\mathbf{y}_3 & 0 \\
0.04 & - 3\cdot 10^7\,\mathbf{y}_2 -10^4\,\mathbf{y}_3 & 0 \\
0 & 3\cdot 10^7\,\mathbf{y}_2 & 0
\end{bmatrix}}_{\mathbf{G}(t,\mathbf{y})}
\begin{bmatrix} \mathbf{y}_1 \\ \mathbf{y}_2 \\ \mathbf{y}_3 \end{bmatrix}.
\end{equation}

The initial conditions are:
\begin{equation}
\mathbf{y}(0) = [1, 0, 0]^\top.
\end{equation}

Here, $\mathbf{G}(t,\mathbf{y})$ has the graph Laplacian structure and 
satisfies a linear invariant corresponding to mass conservation \cite{Blanes_2022_positivity}.
This invariant is
\begin{equation}
\label{eqn:robertson-invariant}
\mathbf{w}^\top\,\mathbf{y} = \mathbf{y}_1 + \mathbf{y}_2 + \mathbf{y}_3 = \mathrm{constant}.
\end{equation}
This means that the total mass of all species is preserved by the continuous system, even though the individual species concentrations evolve over multiple time scales.

\subsubsection{MAPK Cascade Model}
\label{subsubsec:mapk-theory}

The mitogen-activated protein kinase (MAPK) cascade is a fundamental biochemical signaling network exhibiting multistability and oscillations. In the reduced six-dimensional form (following \cite{Blanes_2022_positivity}), the system reads:

\begin{equation}
\frac{\mathrm{d}}{\mathrm{d}t}
\begin{bmatrix}
\mathbf{y}_1 \\ \mathbf{y}_2 \\ \mathbf{y}_3 \\ \mathbf{y}_4 \\ \mathbf{y}_5 \\ \mathbf{y}_6
\end{bmatrix}
=
\underbrace{
\begin{bmatrix}
-(k_7 + k_1 \mathbf{y}_2) & 0 & 0 & k_2 & 0 & 0 \\
0 & -k_1 \mathbf{y}_1 & k_5 & 0 & 0 & 0 \\
0 & 0 & -(k_3 \mathbf{y}_1 + k_5) & k_2 & k_4 & 0 \\
(1-\alpha)k_1 \mathbf{y}_2 & \alpha k_1 \mathbf{y}_1 & 0 & -k_2 & 0 & 0 \\
0 & 0 & k_3 \mathbf{y}_1 & 0 & -k_4 & 0 \\
k_7 & 0 & 0 & 0 & 0 & -k_6
\end{bmatrix}
}_{\mathbf{G}(t,\mathbf{y})}
\begin{bmatrix}
\mathbf{y}_1 \\ \mathbf{y}_2 \\ \mathbf{y}_3 \\ \mathbf{y}_4 \\ \mathbf{y}_5 \\ \mathbf{y}_6
\end{bmatrix}.
\label{eqn:mapk-system}
\end{equation}

Here, the parameter $\alpha \in [0,1]$ controls how the coupling between $\mathbf{y}_1$ and $\mathbf{y}_2$ is distributed between different reaction pathways. While $\mathbf{G}(t,\mathbf{y})$ does not always have column sums equal to zero, it retains the sign pattern of a Laplacian, thus preserving positivity of the exact solution \cite{Blanes_2022_positivity}.

The typical rate parameter values are used here: 
\[
 k_1 = \frac{100}{3}, \,k_2 = \frac{1}{3}, \,k_3 = 50, \,k_4 = \frac{1}{2}, \,k_5 = \frac{10}{3}, \,k_6 = \frac{1}{10}, \,k_7 = \frac{7}{10}.
\] 
The initial condition for this model is given by:
\[
\mathbf{y}(0) = [ 0.1, \,0.175, \,0.15, \,1.15, \,0.81, \,0.5 ]^\top.
\]
For $\alpha \in (0,1)$ this MAPK cascade possesses two independent linear conservation laws
given by the following left kernel vectors of $\mathbf{G}(t,\mathbf{y})$:
\[
\mathbf{w}_1^\top = [ 1, 0, 0, 1, 0, 1 ], 
\quad 
\mathbf{w}_2^\top = [ 0, 1, 1, 1, 1, 0 ].
\]
The corresponding conserved quantities are
\begin{equation}
C_1(t) = \mathbf{w}_1^\top \mathbf{y}(t), \qquad
C_2(t) = \mathbf{w}_2^\top \mathbf{y}(t).
\label{eqn:mapk-invariant}
\end{equation}

Special cases behave differently and can be summarized as follows:  
\begin{itemize}
    \item For $\alpha = 0$, $\mathbf{w}_1$ is an invariant, but $\mathbf{w}_2$ is not.
    \item For $\alpha = 1$, $\mathbf{w}_2$ is an invariant, but $\mathbf{w}_1$ is not.
\end{itemize}
Therefore, it is impossible for \(\alpha \in \{0,1\}\) to simultaneously enforce both invariants with methods that rely solely on matrix exponentials. This reflects a typical limitation in the geometric integration of nonlinear systems with multiple invariants \cite{Blanes_2022_positivity}. We chose $\alpha = 1$ to ensure exact preservation of the conservation law corresponding to $\mathbf{w}_2$, which is typically more relevant for the signaling interpretation of the cascade.

\subsubsection{Stratospheric Reaction System}
\label{subsubsec:stratospheric-theory}

The stratospheric reaction model from \cite{Sandu_2001_positiveChemistry,Sandu_2002_positivity-Favoring,Sandu_2011_chemSolverOverview} represents photochemical interactions in atmospheric chemistry. 
The stratospheric reaction problem involves six chemical tracers:
\[
\mathbf{y} = \big[ [O^{1D}], \,[O], \,[O_3], \,[O_2], \,[NO], \,[NO_2] \big]^\top = [\mathbf{y}_1, \dots, \mathbf{y}_6]^\top,
\]
where each component represents the concentration of a specific species.
The time evolution of the concentrations is governed by the following stiff system of ODEs \cite{Sandu_2001_positiveChemistry,Sandu_2002_positivity-Favoring}:
\begin{equation}
\label{eqn:strato-ode}
\begin{aligned}
\mathbf{y}_1' &= k_5 \mathbf{y}_3 - k_6 \mathbf{y}_1 - k_7 \mathbf{y}_1 \mathbf{y}_3, \\
\mathbf{y}_2' &= 2k_1 \mathbf{y}_4 - k_2 \mathbf{y}_2 \mathbf{y}_4 + k_3 \mathbf{y}_3 - k_4 \mathbf{y}_2 \mathbf{y}_3 + k_6 \mathbf{y}_1 - k_9 \mathbf{y}_2 \mathbf{y}_6 + k_{10} \mathbf{y}_6, \\
\mathbf{y}_3' &= k_2 \mathbf{y}_2 \mathbf{y}_4 - k_3 \mathbf{y}_3 - k_4 \mathbf{y}_2 \mathbf{y}_3 - k_7 \mathbf{y}_1 \mathbf{y}_3 - k_8 \mathbf{y}_3 \mathbf{y}_5, \\
\mathbf{y}_4' &= -k_1 \mathbf{y}_4 - k_2 \mathbf{y}_2 \mathbf{y}_4 + k_3 \mathbf{y}_3 + 2k_4 \mathbf{y}_2 \mathbf{y}_3 + k_5 \mathbf{y}_3 + 2k_7 \mathbf{y}_1 \mathbf{y}_3 + k_8 \mathbf{y}_3 \mathbf{y}_5 + k_9 \mathbf{y}_2 \mathbf{y}_6, \\
\mathbf{y}_5' &= -k_8 \mathbf{y}_3 \mathbf{y}_5 + k_9 \mathbf{y}_2 \mathbf{y}_6 + k_{10} \mathbf{y}_6, \\
\mathbf{y}_6' &= k_8 \mathbf{y}_3 \mathbf{y}_5 - k_9 \mathbf{y}_2 \mathbf{y}_6 - k_{10} \mathbf{y}_6.
\end{aligned}
\end{equation}
The rate coefficients are given by:
\[
\begin{aligned}
k_1 &= 2.643 \cdot 10^{-10} \, \sigma^3(t), & \quad k_2 &= 8.018 \cdot 10^{-17}, & \quad k_3 &= 6.120 \cdot 10^{-4} \, \sigma(t), \\
k_4 &= 1.576 \cdot 10^{-15}, & \quad k_5 &= 1.070 \cdot 10^{-3} \, \sigma^2(t), & \quad k_6 &= 7.110 \cdot 10^{-11}, \\
k_7 &= 1.200 \cdot 10^{-10}, & \quad k_8 &= 6.062 \cdot 10^{-15}, & \quad k_9 &= 1.069 \cdot 10^{-11}, \\
k_{10} &= 1.289 \cdot 10^{-2} \, \sigma(t).
\end{aligned}
\]
Here, $\sigma(t)$ is a periodic function modeling the diurnal cycle:
\[
\sigma(t) =
\begin{cases}
\displaystyle \frac{1}{2} + \frac{1}{2}\cos\!\left( \pi \left| \frac{2T_L - T_R - T_S}{T_S - T_R}\right| \frac{2T_L - T_R - T_S}{T_S - T_R} \right), & \text{if } T_R \leq T_L \leq T_S, \\[6pt]
0, & \text{otherwise,}
\end{cases}
\]
where:
\[
T_L = \left( \frac{t}{3600} \right) \bmod 24, \quad
T_R = 4.5, \quad
T_S = 19.5.
\]
The time is measured in seconds. The initial time is taken at noon, $t_0 = 12 \cdot 3600$, and the integration is typically performed over a three-day period $t_f = t_0 + 72 \cdot 3600$.

The initial concentrations (in molecules per cubic centimeter) are:
\[
y_0 = \big[ 9.906 \cdot 10, ~ 6.624 \cdot 10^8, ~ 5.326 \cdot 10^{11}, ~ 1.697 \cdot 10^{16}, ~ 8.725 \cdot 10^8, ~ 2.240 \cdot 10^8 \big]^\top.
\]
This is a non-autonomous system that can be written in the general graph Laplacian form \eqref{eqn:graph-laplacian-system}
with $\mathbf{G}(t, \mathbf{y})$ an explicitly time-dependent graph Laplacian matrix.

A possible choice for the graph Laplacian $\mathbf{G}$ is:
\[
\resizebox{\linewidth}{!}{$
\begin{bmatrix}
-(k_6 + k_7 \mathbf{y}_3) & 0 & k_5 & 0 & 0 & 0 \\
k_6 & -(k_2 \mathbf{y}_4 + k_4 \mathbf{y}_3 + k_9 \mathbf{y}_6) & k_3 & 2k_1 & 0 & k_{10} \\
0 & \frac{1}{3}k_2 \mathbf{y}_4 & -\gamma & \frac{2}{3}k_2 \mathbf{y}_2 & 0 & 0 \\
\frac{1}{2}k_7 \mathbf{y}_3 & k_4 \mathbf{y}_3 + \frac{1}{2}k_9 y_6 & \gamma + \frac{1}{2}k_7 \mathbf{y}_1 & -(k_1 + k_2 \mathbf{y}_2) & 0 & \frac{1}{2}k_9 \mathbf{y}_2 \\
0 & 0 & 0 & 0 & -k_8 \mathbf{y}_3 & k_{10} + k_9 \mathbf{y}_2 \\
0 & 0 & 0 & 0 & k_8 \mathbf{y}_3 & -(k_{10} + k_9 \mathbf{y}_2)
\end{bmatrix}
$}
\]
with $\gamma = k_3 + k_5 + k_4 \mathbf{y}_2 + k_7 \mathbf{y}_1 + k_8 \mathbf{y}_5$.

This problem has two linear invariants corresponding to the conservation of the total number of oxygen and nitrogen atoms, respectively:
\[
\mathbf{w}_1 = [1,\,1,\,3,\,2,\,1,\,2]^\top, \quad \mathbf{w}_2 = [0,\,0,\,0,\,0,\,1,\,1]^\top.
\]

The corresponding conserved quantities are
\begin{equation}
M_O(t) = \mathbf{w}_1^\top \mathbf{y}(t), \qquad
M_N(t) = \mathbf{w}_2^\top \mathbf{y}(t).
\label{eqn:strato-invariant}
\end{equation}

However, due to the structure of $\mathbf{G}$, it is impossible to find a single $\mathbf{G}$ that simultaneously preserves both invariants exactly:
\[
\mathbf{w}_1^\top\,\mathbf{G}(t, \mathbf{y}) \neq 0, \quad \mathbf{w}_2^\top\,\mathbf{G}(t, \mathbf{y}) = 0.
\]
Thus, one invariant can be preserved exactly (here, total nitrogen), while the other (total oxygen) is only preserved to high accuracy numerically \cite{Blanes_2022_positivity}.

In summary, all three problems, Robertson reaction, MAPK cascade, and stratospheric chemistry fit the $\mathbf{G}(t,\mathbf{y})\mathbf{y}$ framework of production-destruction systems, enabling the application of positivity-preserving time integrators.

\subsubsection{Korteweg-De Vries (KdV) equation}

The Korteweg-De Vries (KdV) equation is a nonlinear dispersive partial differential equation originally derived as a model for the propagation of weakly nonlinear waves \cite{Jager_2011_originKdV}. In its generalized form, it reads:

\begin{equation}
\frac{\partial u}{\partial t} + \alpha u \frac{\partial u}{\partial x} + \rho \frac{\partial u}{\partial x} + \nu \frac{\partial^3 u}{\partial x^3} = 0,
\label{eqn:kdv-pde}
\end{equation}

where $\alpha$ is the nonlinear advection coefficient, $\rho$ is a linear transport (convection) coefficient, $\nu$ is the dispersion parameter.

For numerical experiments, we discretize the KdV equation using a finite volume formulation on a periodic domain and recast it in the general production-destruction (PD) form \eqref{eqn:pd-D-definition}:
\[
\mathbf{y}' = \mathbf{H}(\mathbf{y})\, \mathbf{1}, \quad \mathbf{H}(\mathbf{y}) = \mathbf{D}^\top - \operatorname{diag}(\mathbf{D}\,\mathbf{1}).
\]
Here $\mathbf{y} \in \mathbb{R}^d$ is the cell-averaged discrete solution (shifted to ensure positivity), and $\mathbf{D}(\mathbf{y}) \in \mathbb{R}^{d \times d}$ is a non-negative matrix of destruction rates. 
The discrete flux function is defined as:
\[
\mathbf{f} = -\left( \alpha\, \mathbf{y}^2 + \rho\, \mathbf{y} + \nu\, L_x \mathbf{y} \right),
\]
where $L_x$ is the discrete second derivative operator.

The matrix $\mathbf{D}$ is constructed from the numerical fluxes at cell interfaces.
For each interface flux $f_{i+1/2}$ we define the corresponding production and
destruction rates in a conservative, pairwise manner:
\[
\begin{cases}
p_{i,i+1} = d_{i+1,i} = \tfrac{1}{\Delta x}\, \mathbf{f} _{i+1/2}, & \mathbf{f} _{i+1/2} \ge 0, \\[1ex]
p_{i+1,i} = d_{i,i+1} = -\tfrac{1}{\Delta x}\, \mathbf{f} _{i+1/2}, & \mathbf{f} _{i+1/2} < 0,
\end{cases}
\]
so that each positive flux contributes to outflow from the upwind cell and inflow
to its neighbor. Collecting these rates, we obtain
\[
\mathbf{D} = \operatorname{diag}(\mathbf{y})\,\mathbf{L}, \qquad
\mathbf{L}_{i,j} = 
\begin{cases}
\frac{(\mathbf{S}_p\, \mathbf{f})_i}{\mathbf{y}_i} & \text{when $j$ is the upwind neighbor of $i$}, \\
0 & \text{otherwise},
\end{cases}
\]
with periodic indexing on the domain. Here $\mathbf{S}_p$ denotes the finite-volume operator
for the first derivative in conservative flux form.

The initial condition is chosen as a solitary wave:
\[
\mathbf{y}(0) = 6\, \mathrm{sech}^2(x),
\]
and the system is evolved in time according to the structure-preserving framework \eqref{eqn:pd-D-definition}, making this a nonlinear, non-diagonal example of the general class of production-destruction systems.

The semi-discrete KdV equation preserves the discrete mass
\begin{equation}
M(t) = \Delta x \sum_{j=1}^N u_j(t).
\label{eqn:kdv-invariant}
\end{equation}

Equivalently, $M(t) = \mathbf{w}^\top \mathbf{y}(t)$ with
$\mathbf{w} = \Delta x (1,\dots,1)^\top$.

This test problem is used to assess the ability of numerical methods to handle nonlinear wave dynamics, while respecting the structural constraints required for positivity-preserving integration. Similar positivity-preserving production-destruction formulations for PDEs have been studied in \cite{Huang_2019_positivity}, where unconditional positivity and conservation were established in the context of non-equilibrium flows.

\subsection{Base SDIRK methods}

To integrate stiff production-destruction systems, we use a family of singly diagonally implicit Runge-Kutta (SDIRK) methods with embedded error estimators.

An $s$-stage SDIRK method for the autonomous system
\[
\mathbf{y}' = f(\mathbf{y}), \quad \mathbf{y}(t_n) = \mathbf{y}_n,
\]
takes the form
\begin{align}
\mathbf{Y}_i &= \mathbf{y}_n + h \sum_{j=1}^{i} a_{ij}\, f(\mathbf{Y}_j), \quad i = 1,\dots,s, \\
\mathbf{y}_{n+1} &= \mathbf{y}_n + h \sum_{j=1}^s b_j\, f(\mathbf{Y}_j),
\end{align}
where the Butcher matrix $A = (a_{ij})$ is lower triangular with identical diagonal entries $a_{ii} = \gamma$ \cite{Hairer_1993_solvingODE}.

We use three classical SDIRK methods of increasing order, each with an embedded scheme for adaptive step-size control \cite{Hairer_1993_solvingODE}.
In this work we carry out experiments with fixed step sizes. In general, the positivity holds even when adaptive step sizes are used. The embedded scheme only provides truncation error estimates and does not need to be positive.

\subsubsection{SDIRK21 scheme}

The SDIRK21 scheme \cite{Kennedy_2016} is a two-stage, second-order, L-stable method with an embedded first-order solution for local error estimation.
The specific numerical values for coefficients used in this work are:
\[
\begin{array}{c|cc}
1 - \frac{1}{\sqrt{2}} & 1 - \frac{1}{\sqrt{2}} & 0 \\
1 & \frac{1}{\sqrt{2}} & 1 - \frac{1}{\sqrt{2}} \\
\hline
 & \frac{1}{\sqrt{2}} & 1 - \frac{1}{\sqrt{2}} \\
\hline
 & \frac{2}{3} & \frac{1}{3}
\end{array}
\]
This method is second-order accurate, L-stable, making it robust for stiff problems, and has an embedded Euler-like first-order scheme for adaptive step-size control \cite{Kennedy_2016}.

\subsubsection{SDIRK32 scheme}

The SDIRK32 scheme  \cite{Kennedy_2016} is a four-stage, third-order, L-stable method with an embedded second-order solution.
The coefficients used in this work are:
\[
\renewcommand{\arraystretch}{1.5}
\begin{array}{c|cccc}
\frac{9}{40} & \frac{9}{40} & 0 & 0 & 0 \\
\frac{7}{13} & \frac{163}{520} & \frac{9}{40} & 0 & 0 \\
\frac{11}{15} & -\frac{6481433}{8838675} & \frac{87795409}{70709400} & \frac{9}{40} & 0 \\
1 & \frac{4032}{9943} & \frac{6929}{15485} & -\frac{723}{9272} & \frac{9}{40} \\
\hline
 & \frac{4032}{9943} & \frac{6929}{15485} & -\frac{723}{9272} & \frac{9}{40} \\
\hline
 & \frac{20}{51} & -\frac{410011317313}{22571052756900} & -\frac{2075562131}{281561189810} & \frac{29595333}{2429345900}
\end{array}
\]

This method is third-order accurate, L-stable, suitable for moderately stiff problems, and has an embedded second-order method for error estimation \cite{Kennedy_2016}.

\subsubsection{SDIRK43 scheme}

The SDIRK43 scheme  \cite{Kennedy_2016} is a five-stage, fourth-order method with an embedded third-order solution.
The Butcher tableau is:
\[
\renewcommand{\arraystretch}{1.5}
\begin{array}{c|ccccc}
\frac{1}{4} & \frac{1}{4} & 0 & 0 & 0 & 0 \\
\frac{9}{10} & \frac{13}{20} & \frac{1}{4} & 0 & 0 & 0 \\
\frac{2}{3} & \frac{580}{1287} & -\frac{175}{5148} & \frac{1}{4} & 0 & 0 \\
\frac{3}{5} & \frac{12698}{37375} & -\frac{201}{2990} & \frac{891}{11500} & \frac{1}{4} & 0 \\
1 & \frac{944}{1365} & -\frac{400}{819} & \frac{99}{35} & -\frac{575}{252} & \frac{1}{4} \\
\hline
 & \frac{944}{1365} & -\frac{400}{819} & \frac{99}{35} & -\frac{575}{252} & \frac{1}{4} \\
\hline
 & \frac{41911}{60060} & -\frac{83975}{144144} & \frac{3393}{1120} & -\frac{27025}{11088} & \frac{103}{352} \\
\end{array}
\]

This method is fourth-order accurate, L-stable but the embedded lower-order method may be A-stable, and provides an embedded third-order estimate for adaptive control \cite{Kennedy_2016}.

\subsubsection{Summary and usage}

These methods provide a hierarchy of accuracy and stiffness robustness. SDIRK21 is the simplest, fully L-stable, and ideal for highly stiff problems. SDIRK32 offers a balance between accuracy and stability. SDIRK43 gives high accuracy while retaining good A-stability properties.

Before applying the positivity-preserving predictor-corrector correction (Section \ref{subsec:DIRK-solution-PC}), we first verify that the baseline SDIRK methods achieve their theoretical order on standard test problems. The second step is to apply the correction from Sec. \ref{subsec:DIRK-solution-PC}, i.e., correct only $\mathbf{y}_{n+1} = \mathbf{Y}_2$, and evaluate the resulting positivity and order of accuracy. Finally, we apply the correction from Sec. \ref{subsec:DIRK-stage-PC}, i.e., correct both $\mathbf{Y}_1$ and $\mathbf{y}_{n+1} = \mathbf{Y}_2$, and again assess positivity preservation and convergence order.

This framework generalizes to SSP SDIRK methods and higher-order schemes \cite{Ketcheson_2009_SSP_RK}, ensuring positivity and conservation without degrading in practice the formal order of accuracy.

\section{Numerical results}
\label{sec:results}

In this section we evaluate the proposed positivity-preserving predictor-corrector schemes for four main aspects:
\begin{enumerate}
\item Positivity enforcement (see Section \ref{subsec:positivity}).
\item Invariant preservation. Verifying that intrinsic conservation laws are respected by the numerical scheme, even after corrections are applied (see Section \ref{subsec:invariant}).
\item Order. Assessing the empirical order of convergence and possible degradation of the theoretical order of SDIRK methods (see Section \ref{subsec:order}).
\item Efficiency. Quantifying the computational overhead (or potential benefits) of applying corrections (see Section \ref{subsec:efficiency}).
\end{enumerate}

\subsection{Positivity preservation}
\label{subsec:positivity}
This subsection examines the ability of the proposed correction strategy to prevent the appearance of negative solution components in practice. For each test problem, we compare the baseline SDIRK integrator with its positivity-corrected variants and illustrate whether non-negativity is maintained throughout the simulation. The examples highlight situations in which uncorrected methods violate positivity and demonstrate how the proposed corrections restore physically meaningful solutions.

\subsubsection{MAPK Cascade}
The MAPK cascade did not exhibit any negative concentrations when simulated using the uncorrected SDIRK methods. This behavior suggests that the intrinsic dynamics of the MAPK system while moderately stiff do not drive the solution outside the positive region, at least under the integration settings used in our tests.

Therefore, positivity-preserving corrections had no visible effect on the MAPK results in terms of enforcing non-negativity. However, corrections were still applied in subsequent sections to test their impact on invariants, convergence order, and computational performance.

\subsubsection{Robertson Reaction}

The classical stiff Robertson reaction problem is known to exhibit significant numerical challenges due to rapid changes. Standard integrators can produce negative concentrations, particularly in the intermediate species $B$ (see \eqref{eqn:robertson-species}). 

\begin{figure}[htbp]
\centering
\includegraphics[scale=0.4]{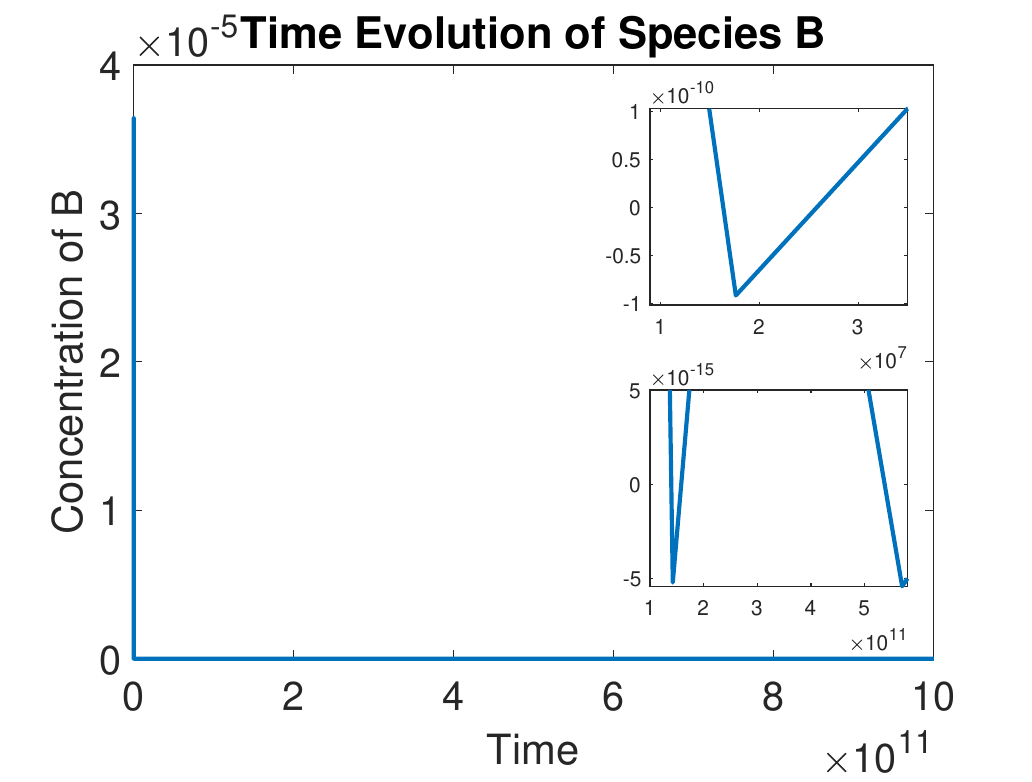}
\caption{Robertson reaction system \eqref{eqn:Robertson}. Time evolution of numerical $B$ computed by the base version of SDIRK21. Negative values are observed, indicating a violation of physical constraints.}
\label{fig:robertson-uncorrected}
\end{figure}

Figure \ref{fig:robertson-uncorrected} shows that when the SDIRK21 scheme is applied without correction, the concentration of $B$ becomes negative during the integration, which is physically unrealistic. 
After applying our proposed positivity correction strategies either at the final stage or at all stages we obtained fully non-negative trajectories as shown in Figures \ref{fig:robertson-last_stage-corrected} and \ref{fig:robertson-all_stages-corrected}, respectively.

\begin{figure}[htbp]
\centering
\includegraphics[scale=0.4]{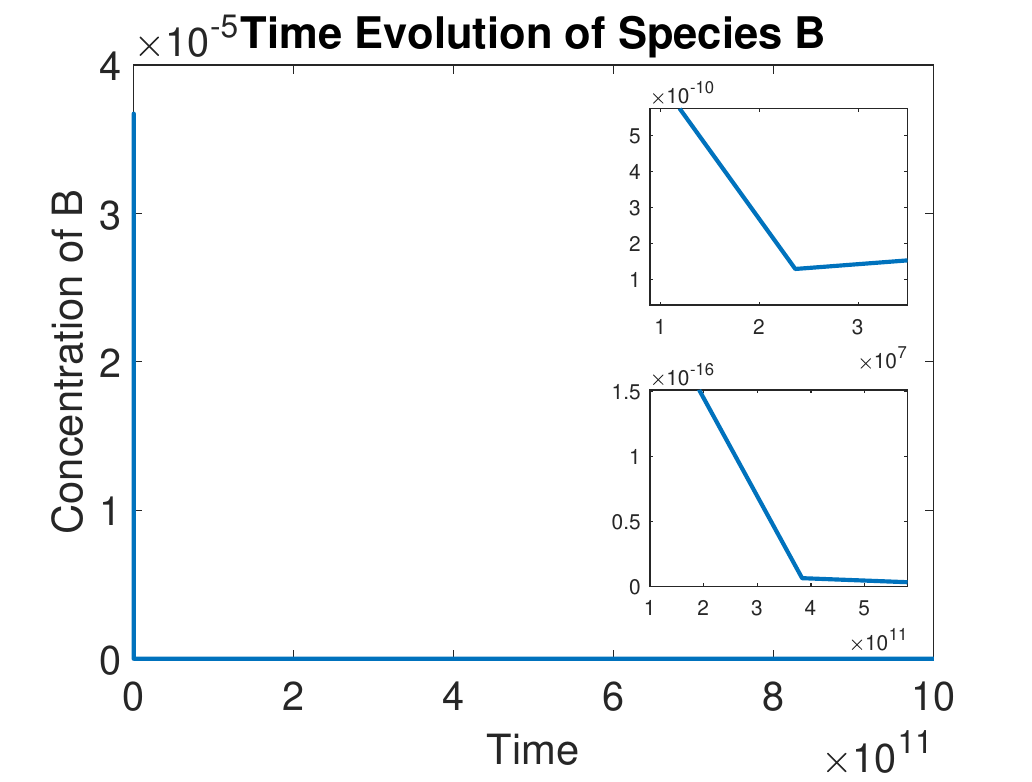}
\caption{Robertson reaction system \eqref{eqn:Robertson}. Time evolution of $B$ using SDIRK21 with positivity correction applied to $\mathbf{y}_{n+1}$. All values remain strictly non-negative.}
\label{fig:robertson-last_stage-corrected}
\end{figure}

\begin{figure}[htbp]
\centering
\includegraphics[scale=0.4]{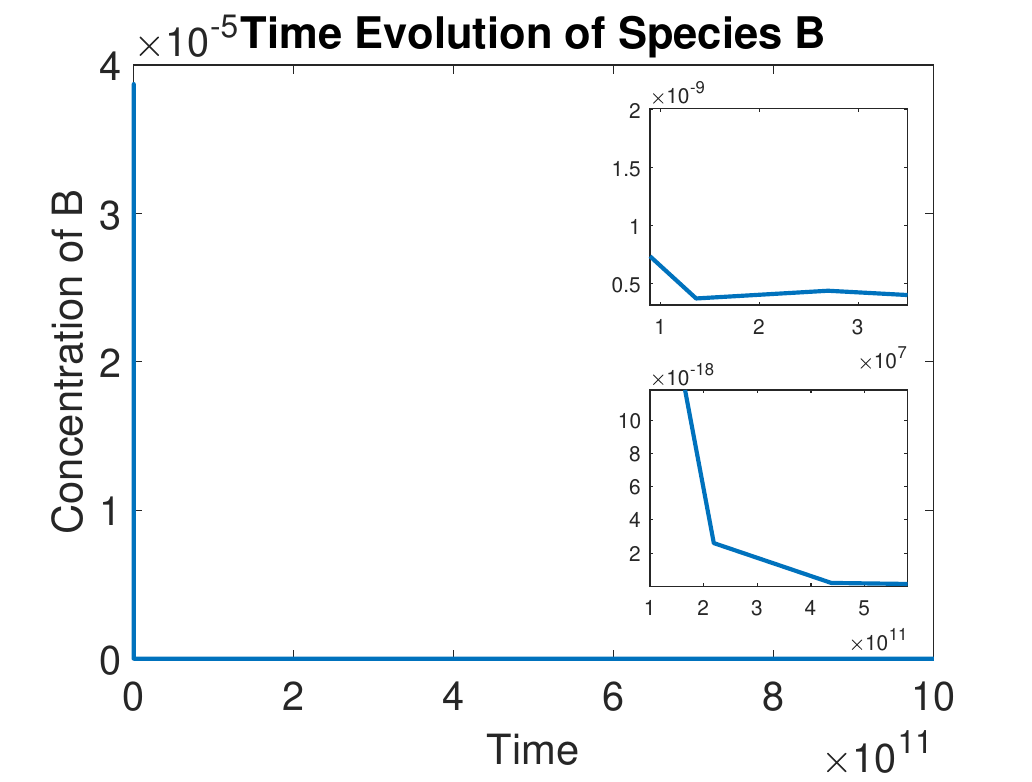}
\caption{Robertson reaction system \eqref{eqn:Robertson}. Time evolution of $B$ using SDIRK21 with positivity correction applied to each stage. The solution remains strictly non-negative.}
\label{fig:robertson-all_stages-corrected}
\end{figure}

These results demonstrate that positivity preserving schemes are to be preferred for preserving the physical integrity of stiff reaction models like Robertson's. Importantly, the corrected methods do not introduce instability or noticeable error in the solution profile over the long simulation interval.

\subsubsection{Stratospheric Reaction}

To verify the effectiveness of the proposed correction scheme, we conducted a one-day simulation of the stratospheric reaction system using the SDIRK21 method, focusing on the concentrations of species O and O\textsuperscript{1D}. These two concentrations are particularly tend to numerical negativity in standard time integration.

\begin{figure}[htbp]
\centering
\begin{subfigure}[h]{0.45\textwidth}
\centering
\includegraphics[scale=0.4]{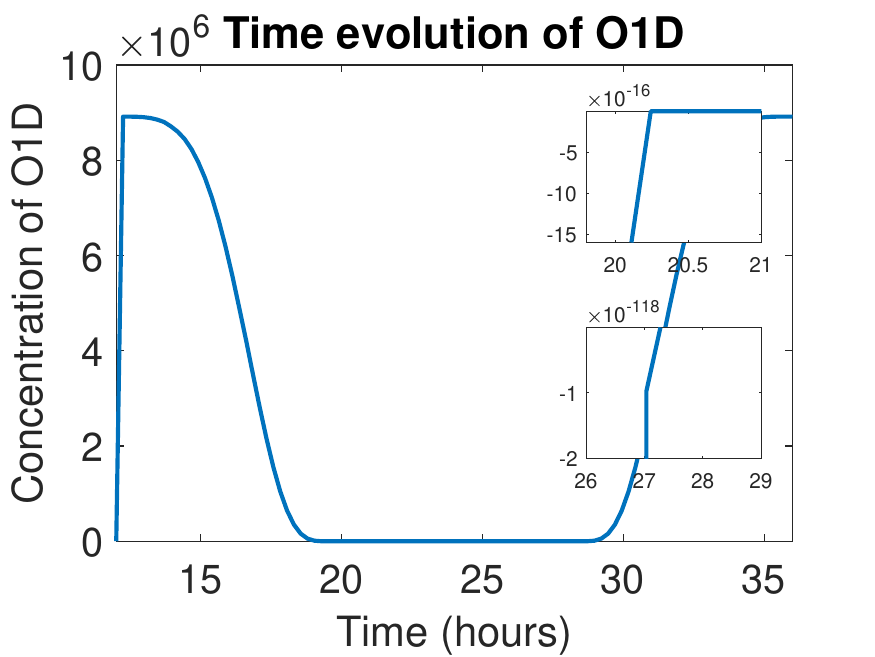}
\caption{Concentration of O\textsuperscript{1D}}
\label{sfig:sdirk-O1D}
\end{subfigure}
\begin{subfigure}[h]{0.45\textwidth}
\centering
\includegraphics[scale=0.4]{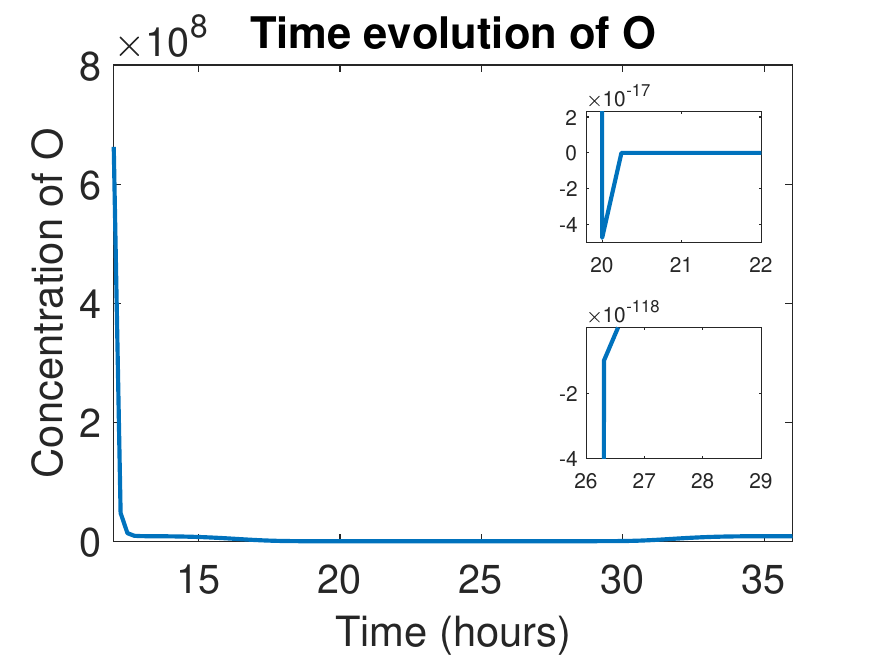}
\caption{Concentration of O}
\label{sfig:sdirk-O}
\end{subfigure}
\caption{Stratospheric Reaction \eqref{eqn:strato-ode}. Concentrations of O\textsuperscript{1D} and O over time using base version of SDIRK21. Negative values are observed, indicating a violation of physical constraints.}
\label{fig:sdirk-uncorrected}
\end{figure}

Figure \ref{fig:sdirk-uncorrected} illustrates the results of integrating the stratospheric reaction system using the SDIRK21 method without applying positivity corrections. In both cases concentrations become negative during the integration interval, which is physically meaningless and violates the positivity of chemical concentrations.

After applying the proposed positivity-preserving correction mechanism, the new results in Figure \ref{fig:sdirk-last_stage-corrected} and \ref{fig:sdirk-all_stages-corrected} show fully non-negative solutions throughout the simulation interval.

\begin{figure}[htbp]
\centering
\begin{subfigure}[h]{0.45\textwidth}
\centering
\includegraphics[scale=0.4]{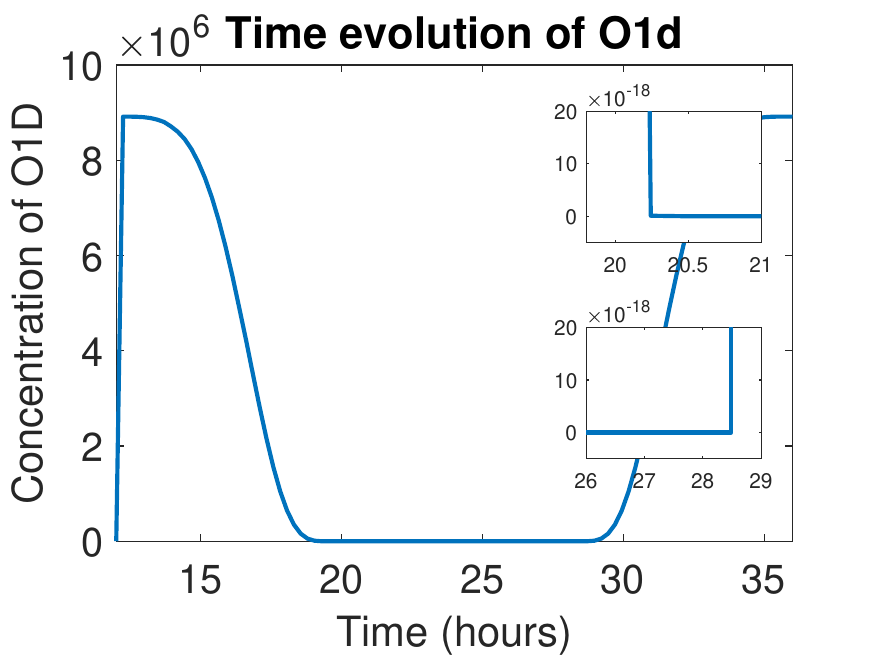}
\caption{Concentration of O\textsuperscript{1D}}
\label{sfig:sdirk-O1D-corr_last-stage}
\end{subfigure}
\begin{subfigure}[h]{0.45\textwidth}
\centering
\includegraphics[scale=0.4]{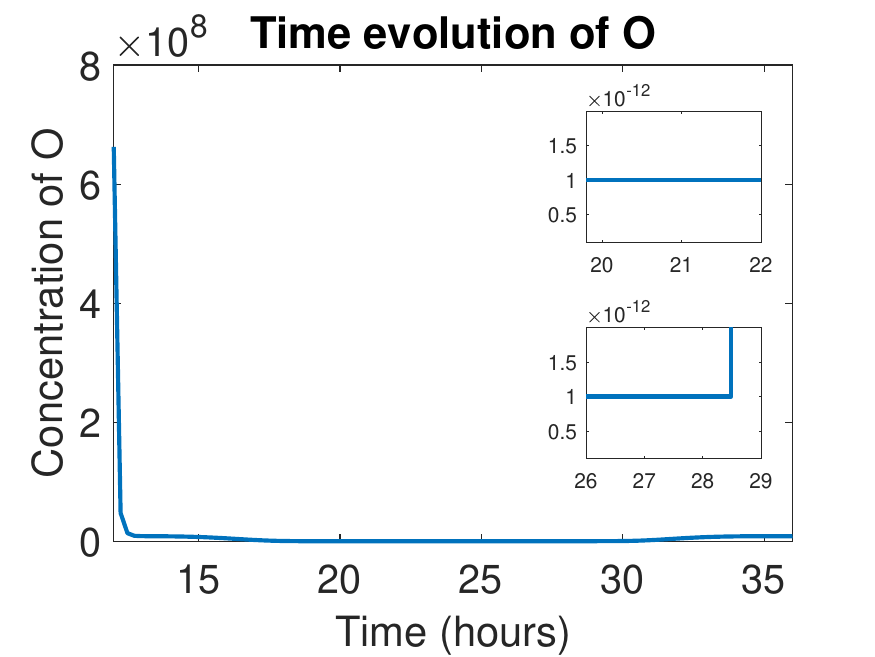}
\caption{Concentration of O}
\label{sfig:sdirk-O-corr_last-stage}
\end{subfigure}
\caption{Stratospheric Reaction \eqref{eqn:strato-ode}. Concentrations of O\textsuperscript{1D} and O over time using SDIRK21 with positivity correction applied to $\mathbf{y}_{n+1}$. All values remain strictly non-negative.}
\label{fig:sdirk-last_stage-corrected}
\end{figure}

\begin{figure}[htbp]
\centering
\begin{subfigure}[h]{0.45\textwidth}
\centering
\includegraphics[scale=0.4]{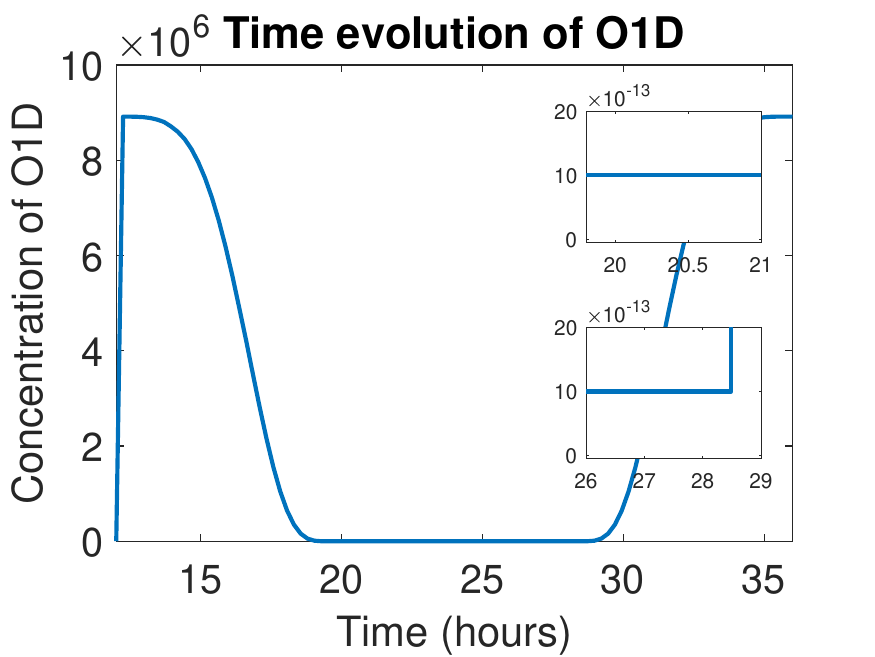}
\caption{Concentration of O\textsuperscript{1D}}
\label{sfig:sdirk-O1D-corr_all-stages}
\end{subfigure}
\begin{subfigure}[h]{0.45\textwidth}
\centering
\includegraphics[scale=0.4]{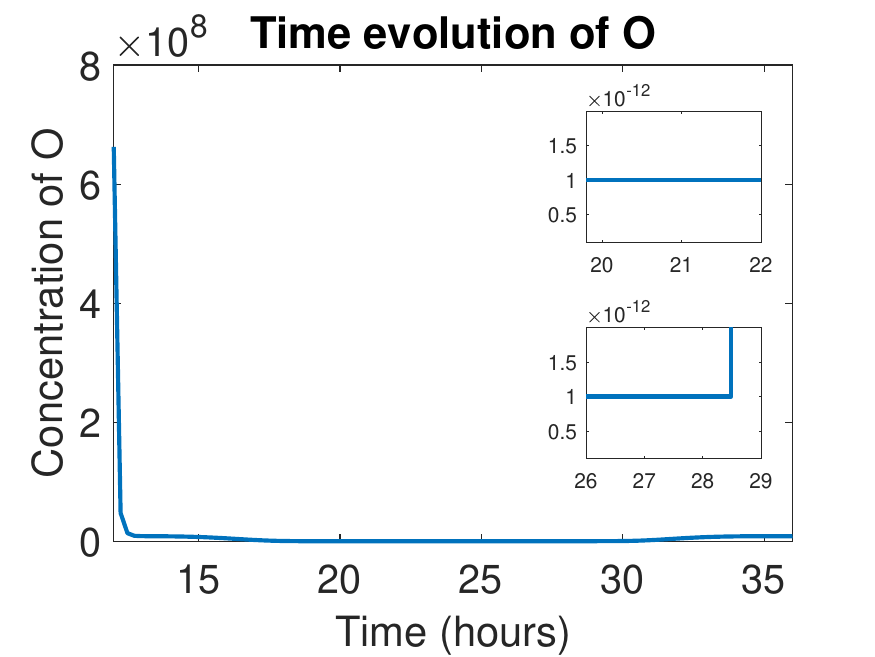}
\caption{Concentration of O}
\label{sfig:sdirk-O-corr_all-stages}
\end{subfigure}
\caption{Stratospheric Reaction \eqref{eqn:strato-ode}. Concentrations of O\textsuperscript{1D} and O over time using SDIRK21 with positivity correction applied to each stage. All values remain strictly non-negative.}
\label{fig:sdirk-all_stages-corrected}
\end{figure}

This experiment highlights the effectiveness of the correction strategy. By enforcing positivity at each stage of the integration, the corrected method maintains the physical realism of the solution.

A natural question is whether positivity could be enforced simply by rejecting steps that produce negative components and retrying with smaller step sizes. We tested this approach by halving the step size $h$ whenever a negative value appeared, while keeping all other solver settings unchanged. While this strategy initially prevents negativity, in practice, the integrator quickly drives the step size below useful limits. In stiff regimes, the solver spends many consecutive attempts with extremely small steps without making appreciable progress in time, eventually hitting the ``step size too small'' termination condition. This experiment shows that simply decreasing the step size is not a viable solution and underscores the necessity of dedicated positivity-preserving corrections.

\begin{figure}[htbp]
    \centering
    \begin{subfigure}[b]{0.45\textwidth}
        \centering
        \includegraphics[width=\textwidth]{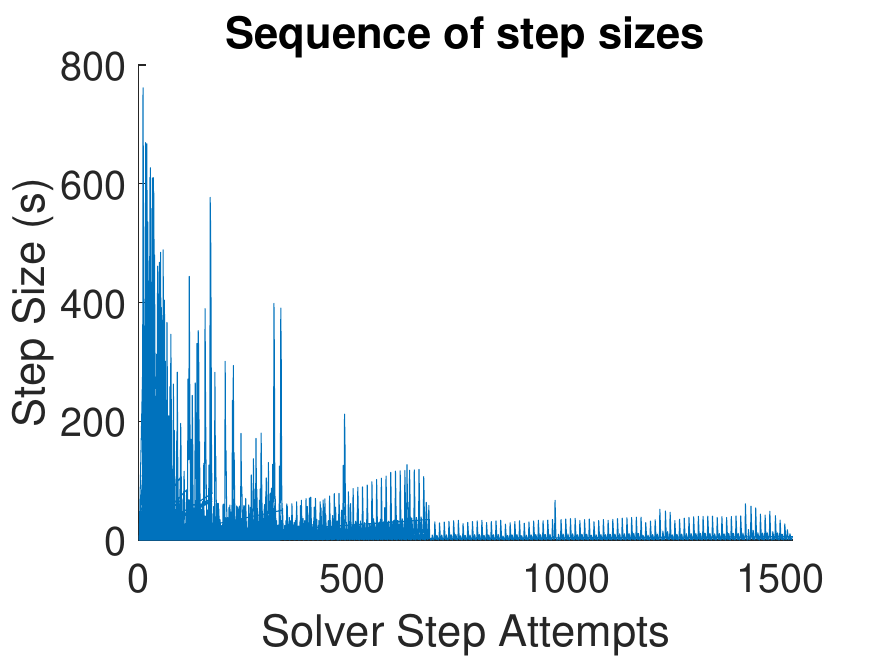}
        \caption{}
        \label{sfig:step-sizes-sequence-with-changes}
    \end{subfigure}
    \begin{subfigure}[b]{0.45\textwidth}
        \centering
        \includegraphics[width=\textwidth]{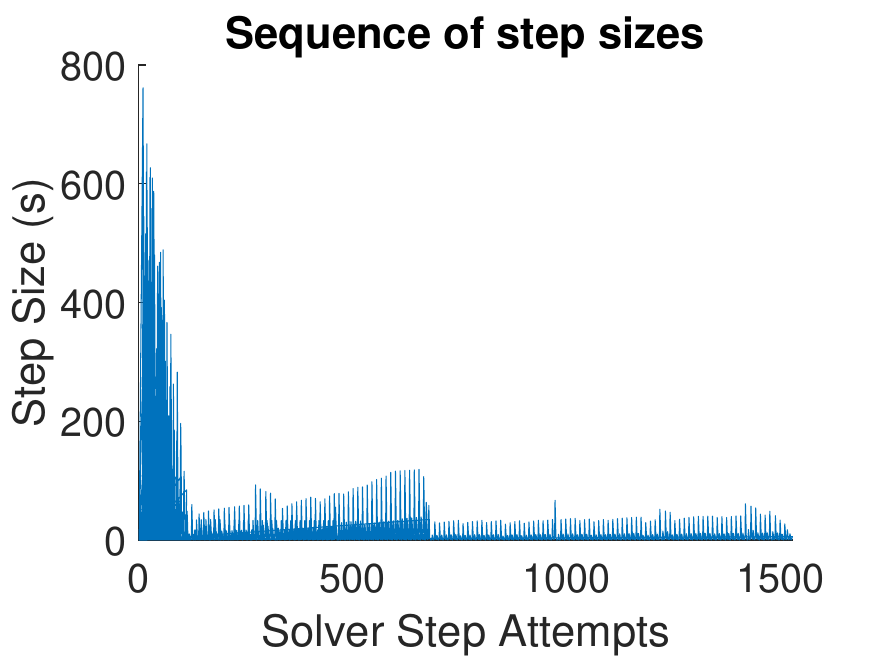}
        \caption{}
        \label{sfig:step-sizes-sequence-without-changes}
    \end{subfigure}
       \caption{Sequence of attempted step sizes $h$ (in seconds) plotted against solver step attempts for the stratospheric reaction model \eqref{eqn:strato-ode} using the base SDIRK method. Figure \ref{sfig:step-sizes-sequence-with-changes} positivity enforced by rejecting steps that produce negative components and halving the step size. Figure \ref{sfig:step-sizes-sequence-without-changes} same configuration, but without positivity-based step rejection. When step rejection is used to enforce positivity, the step size rapidly collapses and the solver becomes trapped in repeated retries; without the positivity guard, the step size remains governed by error control and the solver advances normally.}
       \label{fig:step-sizes-sequence}
\end{figure}

Figure \ref{fig:step-sizes-sequence} shows the sequence of solver step attempts when positivity is enforced by rejecting negative solutions and halving the step size, and compares it with the same solver run without positivity-based step rejection. Each point corresponds to one attempted step, whether accepted or rejected. In the guarded case, after a few hundred attempts, the step size $h$ (measured in seconds) falls by several orders of magnitude and remains clustered near the numerical floor, indicating repeated rejections with negligible progress in physical time. By contrast, when negativity does not trigger step rejection, the step size remains controlled by the error estimator and the solver advances normally. Such behavior illustrates why enforcing positivity by step rejections alone is impractical, and motivates the need for dedicated positivity-preserving corrections that maintain progress without choking the step size.

\subsubsection{KdV Equation}

To further evaluate the performance of the proposed positivity-preserving corrections, we applied the SDIRK21 method to the Korteweg-De Vries (KdV) equation. This nonlinear dispersive equation models the evolution of one-dimensional wave profiles and is widely used as a benchmark for testing numerical integrators. In particular, the KdV system can generate steep gradients and localized soliton structures, making it a suitable test for assessing whether the positivity of the solution is maintained during the simulation.

Unlike the Robertson, MAPK, and stratospheric problems, for the KdV equation we applied the correction only to the final stage $\mathbf{y}_{n+1}$. When the correction was enforced at all intermediate stages $\mathbf{Y}_i$, the numerical solution exhibited severely distorted dynamics, so those results are omitted.

This behavior is consistent with the correction mechanism described in Section \ref{sec:theoretical_background}. The stage correction modifies the intermediate stage values through clipping and scaling, with $\breve{\mathbf{Y}}_i - \mathbf{Y}_i = \mathcal{O}(h^{q+1})$ for SDIRK methods (Lemma \ref{lemma:clipping-errors}). For the semi-discrete KdV equation, these stage values enter the nonlinear spatial operator. As a result, repeated modification of the internal stages can perturb the intermediate wave evolution and lead to noticeable distortion of the solution profile. Applying the correction only to the final stage $\mathbf{y}_{n+1}$ avoids altering the internal stage dynamics while still enforcing non-negativity of the returned solution.

Thus, the KdV results below compare only the baseline SDIRK integrator and the final-stage correction.

\begin{figure}[htbp]
\centering
\includegraphics[scale=0.5]{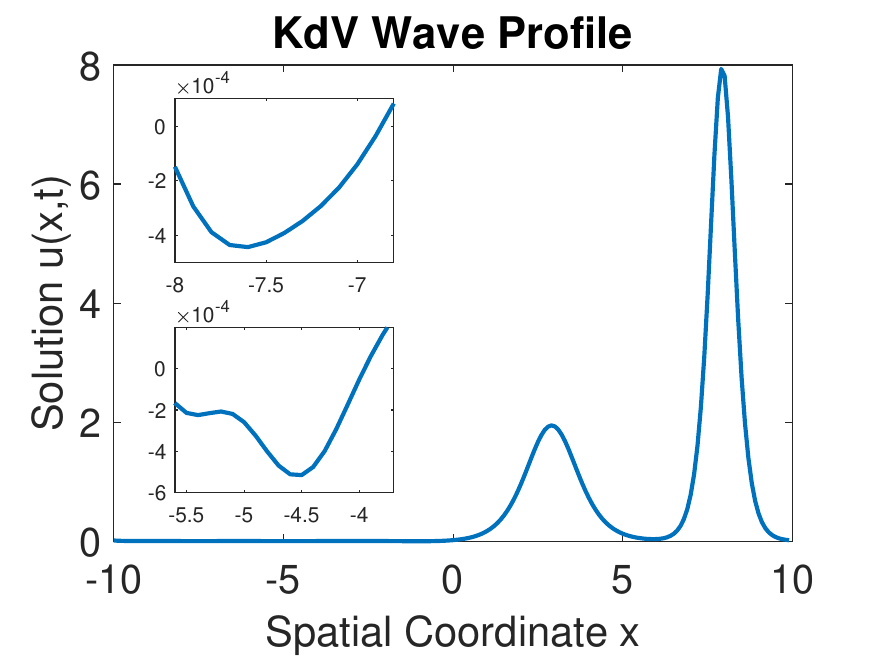}
\caption{KdV equation \eqref{eqn:kdv-pde}. KdV wave profile at the final time $t$ = 0.35 base version of SDIRK21. Negative values are observed, indicating a violation of physical constraints.}
\label{fig:kdv-uncorrected}
\end{figure}

Figure \ref{fig:kdv-uncorrected} illustrates the result of integrating the KdV equation with the uncorrected SDIRK21 method. As shown in the zoomed-in regions, the numerical solution exhibits small but clearly visible negative undershoots. These oscillations are unphysical, as the KdV wave amplitude should remain non-negative.

\begin{figure}[htbp]
\centering
\includegraphics[scale=0.5]{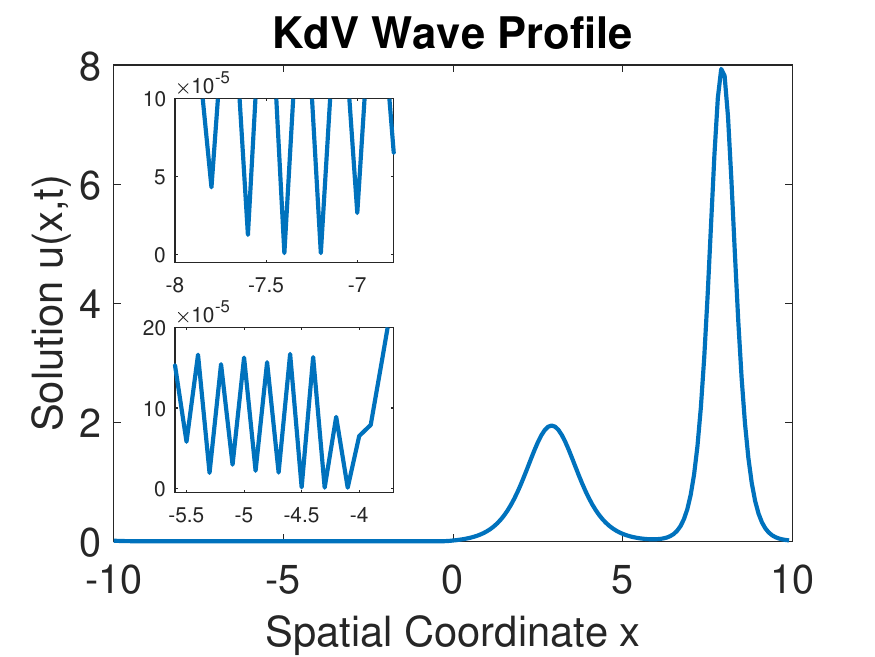}
\caption{KdV equation \eqref{eqn:kdv-pde}. KdV wave profile at the final time $t$ = 0.35 using SDIRK21 with positivity correction applied to $\mathbf{y}_{n+1}$. All values remain strictly non-negative.}
\label{fig:kdv_corr_last-stage}
\end{figure}

After applying the proposed positivity-preserving correction at the final stage, the solution remains strictly non-negative, as shown in Figure \ref{fig:kdv_corr_last-stage}. Importantly, the correction eliminates spurious negative components without introducing visible distortion to the soliton structure.

These results confirm that the correction strategy is not only effective for stiff reaction kinetics but also for discrete nonlinear PDEs such as KdV, where maintaining positivity is essential for preserving the physical integrity of the solution \cite{Huang_2019_positivity,Huang_2019_third-orders,Huang_2023_on-the-stability}.

\subsection{Invariant preservation}
\label{subsec:invariant}

To evaluate the ability of the proposed SDIRK schemes with positivity-preserving Patankar corrections to maintain the invariants of production-destruction systems, we evaluated four benchmark problems. These are the MAPK cascade, the stratospheric reaction system, the Robertson reaction, and the KdV equation. For each problem, we report the relative error in invariant preservation for three cases:
\begin{enumerate}
    \item The baseline SDIRK integrator without correction.
    \item The same integrator with correction applied only to the final stage $\mathbf{y}_{n + 1}$.
    \item The integrator with correction applied to all internal stages $\mathbf{Y}_i$.
\end{enumerate}

To quantify invariant preservation, we measure the maximum relative deviation of a conserved quantity over the integration interval. Let $C(t)$ denote a linear invariant of the continuous or semi-discrete system, and $C_0 := C(0)$ its initial value. We define the relative invariant error as

\begin{equation}
\label{eqn:invariant-error}
E_I := \max_{t \in [0,T]} \frac{\lvert C(t) - C_0 \rvert}{\lvert C_0 \rvert}.
\end{equation}

\subsubsection{Robertson Reaction}

The Robertson reaction preserves the total concentration defined in \eqref{eqn:robertson-invariant}.

Over the time interval $[0,10^4]$, the maximum relative mass deviation
$E_I$ defined in \eqref{eqn:invariant-error} was:
\begin{itemize}
    \item Baseline SDIRK: $2.44 \times 10^{-15}$.
    \item Final-stage correction: $2.22 \times 10^{-15}$.
    \item All-stage correction: $3.99 \times 10^{-15}$.
\end{itemize}

\begin{figure}[htbp]
\centering
\includegraphics[scale=0.4]{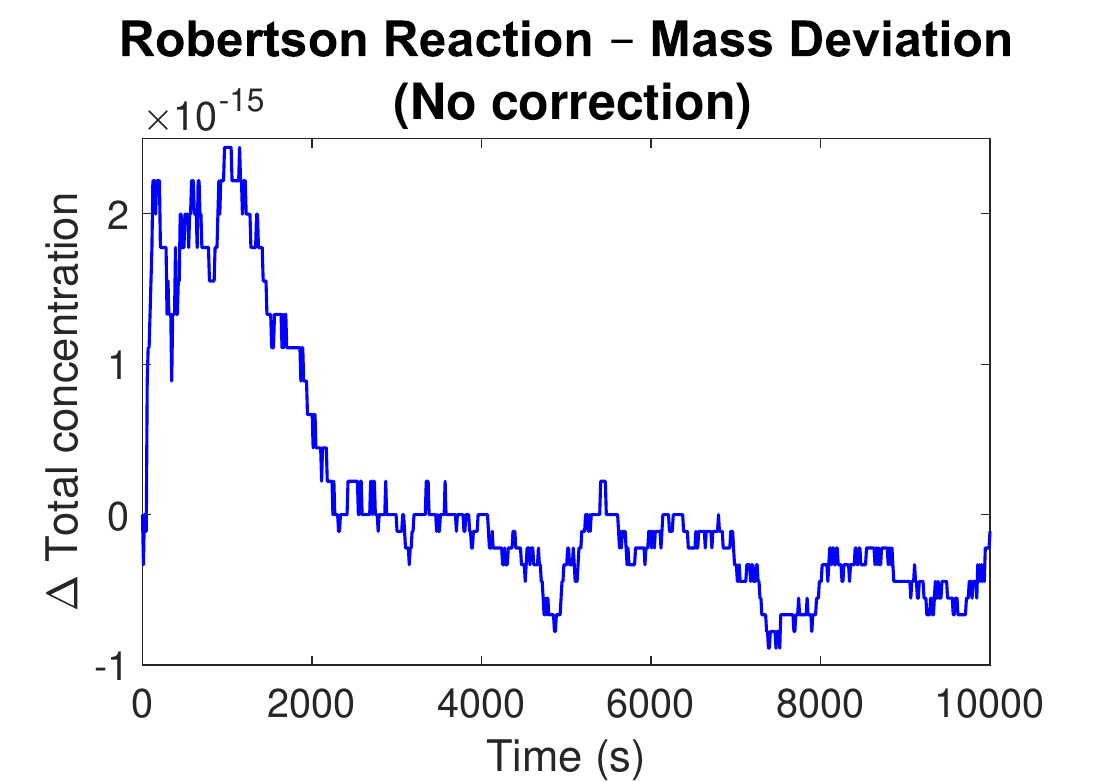}
\caption{Robertson reaction system \eqref{eqn:Robertson}. Deviation of the conserved total concentration for the baseline SDIRK scheme. The plotted quantities show deviations from the initial value, reported in concentration units.}
\label{fig:robertson-mass}
\end{figure}

\begin{figure}[htbp]
\centering
\includegraphics[scale=0.4]{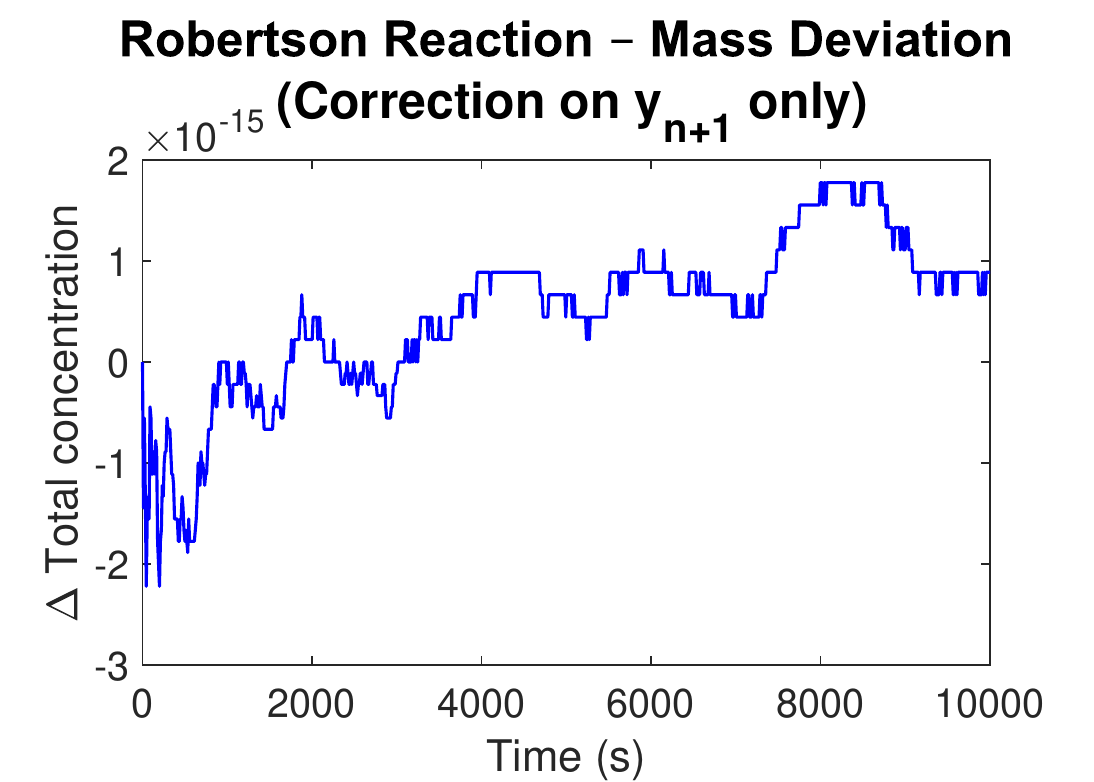}
\caption{Robertson reaction system \eqref{eqn:Robertson}. Deviation of the conserved total concentration with final-stage correction. The plotted quantities show deviations from the initial value, reported in concentration units.}
\label{fig:robertson-mass-last-stage}
\end{figure}

\begin{figure}[htbp]
\centering
\includegraphics[scale=0.4]{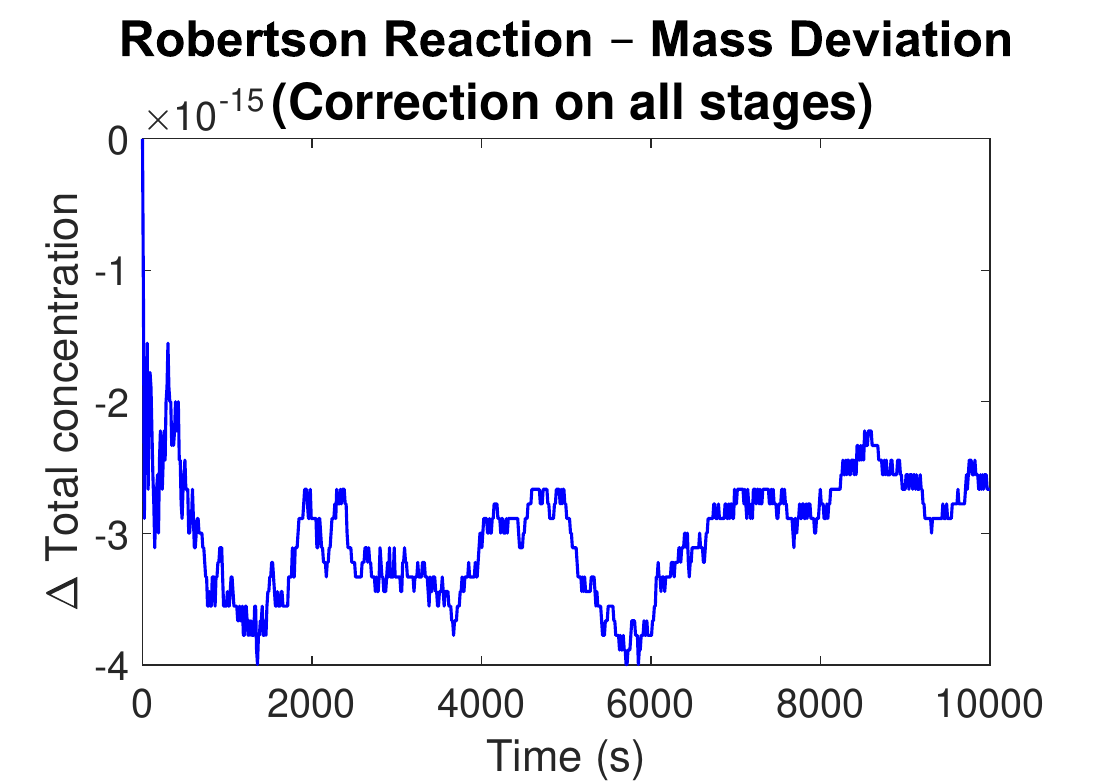}
\caption{Robertson reaction system \eqref{eqn:Robertson}. Deviation of the conserved total concentration with all-stage correction. The plotted quantities show deviations from the initial value, reported in concentration units.}
\label{fig:robertson-mass-all-stages}
\end{figure}

\subsubsection{MAPK Cascade}

As discussed in Subsection \ref{subsubsec:mapk-theory}, the MAPK cascade has two linear conservation laws, defined in $C_1$ and $C_2$ \eqref{eqn:mapk-invariant}.

Numerical tests confirm that only one conservation law is preserved depending on the value of $\alpha$. For $\alpha = 0$, only $C_1$ is preserved. For $\alpha = 1$, only $C_2$ is preserved.

These findings are consistent with the structure of $\mathbf{G}(y)$ and with the prior analysis by \cite{Blanes_2022_positivity}. They showed that, for endpoint values of $\alpha$, one of the conservation laws is lost due to a non-zero left kernel mismatch.

All tests were conducted over the time interval \([0, 200]\). The relative errors in each invariant under different correction strategies are summarized below.

Tables \ref{tab:mapk-alpha-0} and \ref{tab:mapk-alpha-1} summarize the relative invariant
deviations for the MAPK cascade. For $\alpha=0$, only $C_1$ is preserved to machine
precision, while $C_2$ exhibits noticeable drift. Conversely, for $\alpha=1$, $C_2$
is preserved exactly and $C_1$ is not. These results are consistent with the structure
of $\mathbf{G}(t,\mathbf{y})$ and the theoretical analysis in
\cite{Blanes_2022_positivity}.

\begin{table}[htbp]
\centering
\caption{MAPK cascade invariant preservation for $\alpha = 0$ over $[0,200]$.
Reported values are maximum relative deviations $E_I$ defined in
\eqref{eqn:invariant-error}.}
\label{tab:mapk-alpha-0}
\begin{tabular}{lcc}
\toprule
Method & $E_{C_1}$ & $E_{C_2}$ \\
\midrule
Baseline SDIRK        & $9.14 \times 10^{-15}$ & $5.40 \times 10^{-4}$ \\
Final-stage correction & $9.26 \times 10^{-15}$ & $5.39 \times 10^{-4}$ \\
All-stage correction   & $9.26 \times 10^{-15}$ & $5.73 \times 10^{-4}$ \\
\bottomrule
\end{tabular}
\end{table}

\begin{table}[htbp]
\centering
\caption{MAPK cascade invariant preservation for $\alpha = 1$ over $[0,200]$.
Reported values are maximum relative deviations $E_I$ defined in
\eqref{eqn:invariant-error}.}
\label{tab:mapk-alpha-1}
\begin{tabular}{lcc}
\toprule
Method & $E_{C_1}$ & $E_{C_2}$ \\
\midrule
Baseline SDIRK        & $9.48 \times 10^{-4}$ & $3.11 \times 10^{-15}$ \\
Final-stage correction & $9.47 \times 10^{-4}$ & $3.11 \times 10^{-15}$ \\
All-stage correction   & $9.21 \times 10^{-4}$ & $3.11 \times 10^{-15}$ \\
\bottomrule
\end{tabular}
\end{table}

\begin{figure}[htbp]
\centering
\includegraphics[scale=0.36]{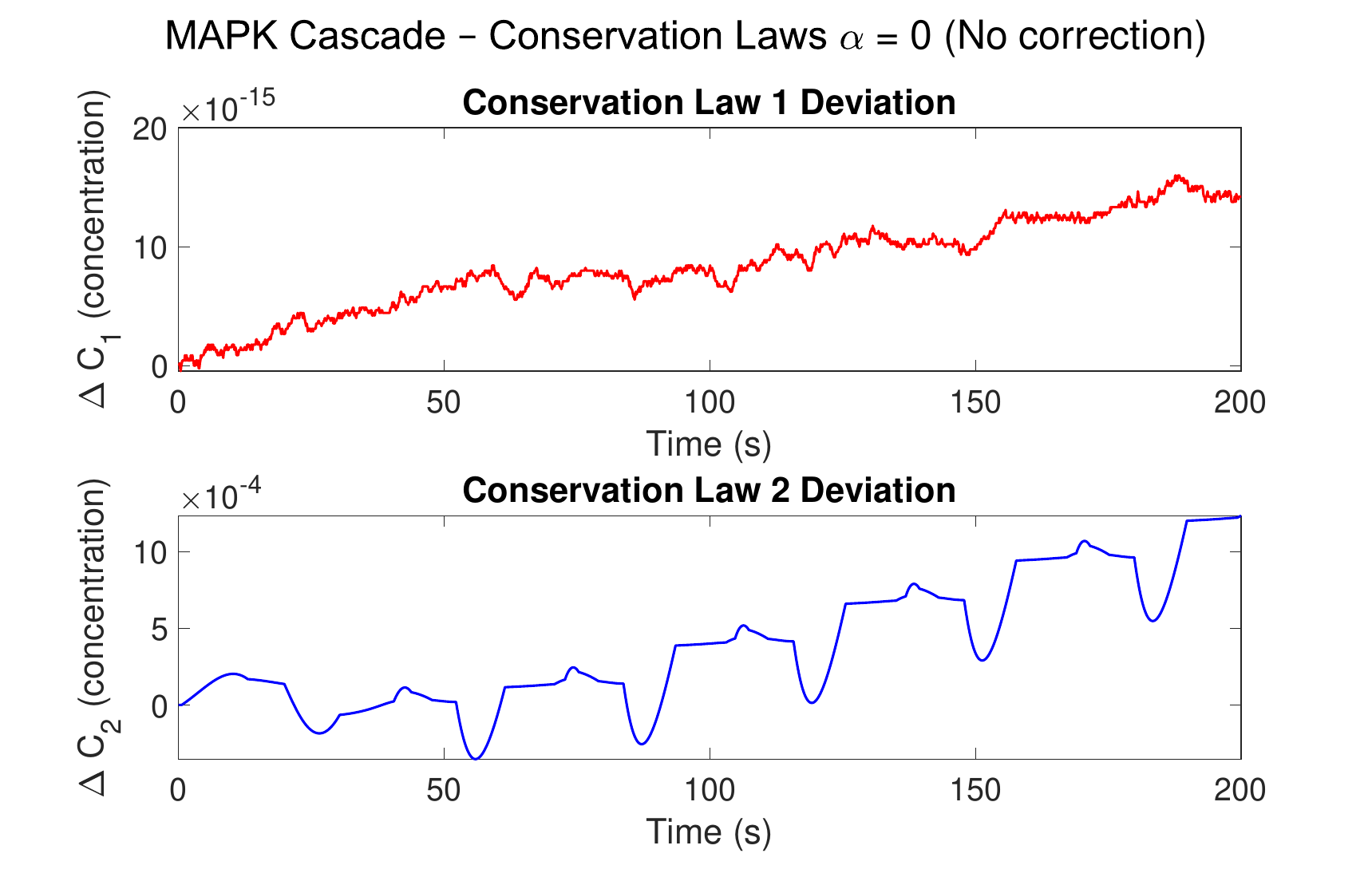}
\caption{MAPK Cascade \eqref{eqn:mapk-system}. Deviations from the conserved quantities for the baseline SDIRK scheme. Here, $C_1(t)=\mathbf{w}_1^\top \mathbf{y}(t)=y_1+y_4+y_6$ and $C_2(t)=\mathbf{w}_2^\top \mathbf{y}(t)=y_2+y_3+y_4+y_5$ denote the linear conservation laws of the MAPK cascade. The plotted quantities are deviations from the initial values, reported in concentration units.}
\label{fig:mapk-mass-0}
\end{figure}

\begin{figure}[htbp]
\centering
\includegraphics[scale=0.36]{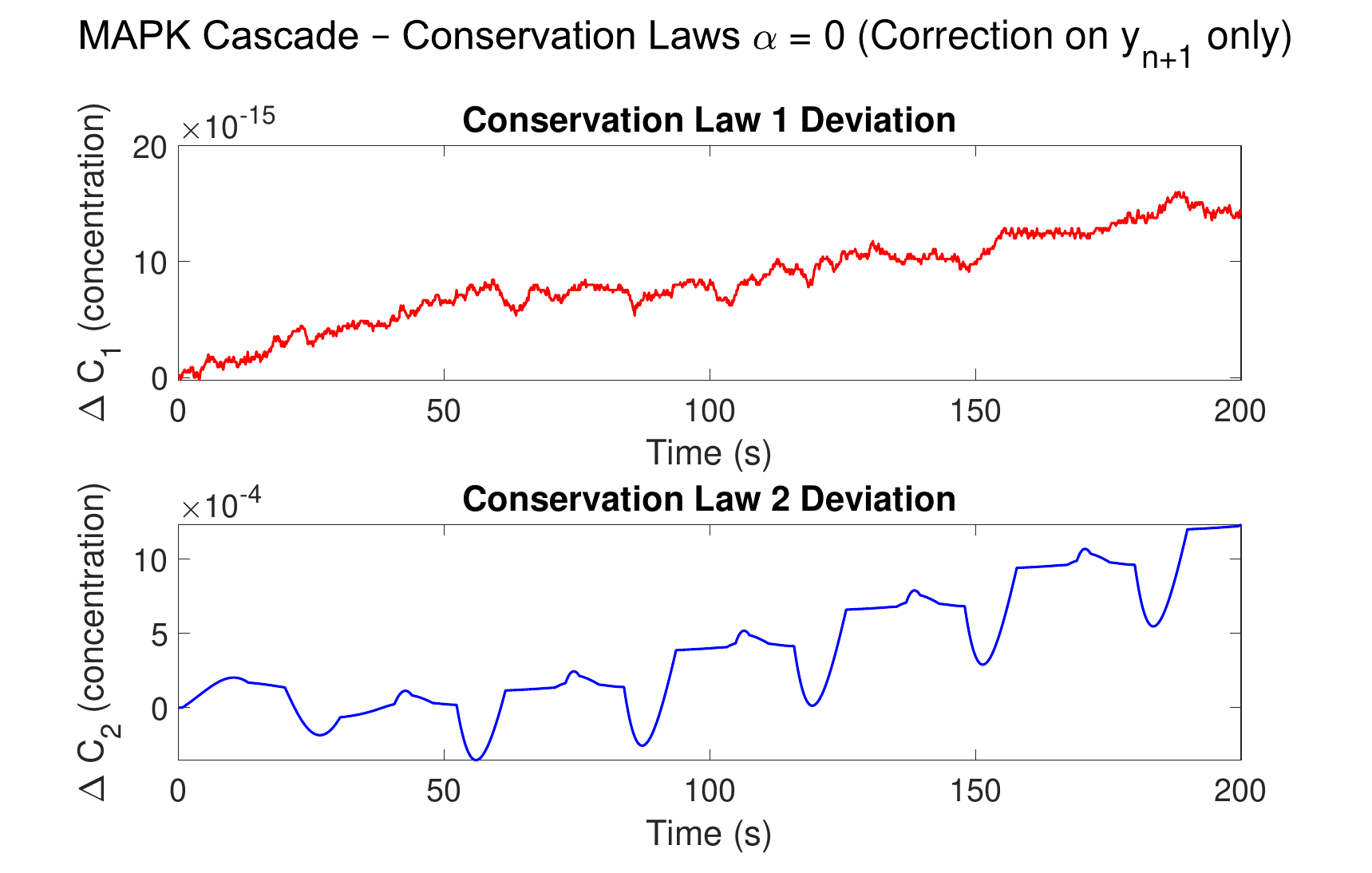}
\caption{MAPK Cascade \eqref{eqn:mapk-system}. Deviations from the conserved quantities with final-stage correction. Here, $C_1(t)=\mathbf{w}_1^\top \mathbf{y}(t)=y_1+y_4+y_6$ and $C_2(t)=\mathbf{w}_2^\top \mathbf{y}(t)=y_2+y_3+y_4+y_5$ denote the linear conservation laws of the MAPK cascade. The plotted quantities are deviations from the initial values, reported in concentration units.}
\label{fig:mapk-mass-0_last-stage}
\end{figure}

\begin{figure}[htbp]
\centering
\includegraphics[scale=0.36]{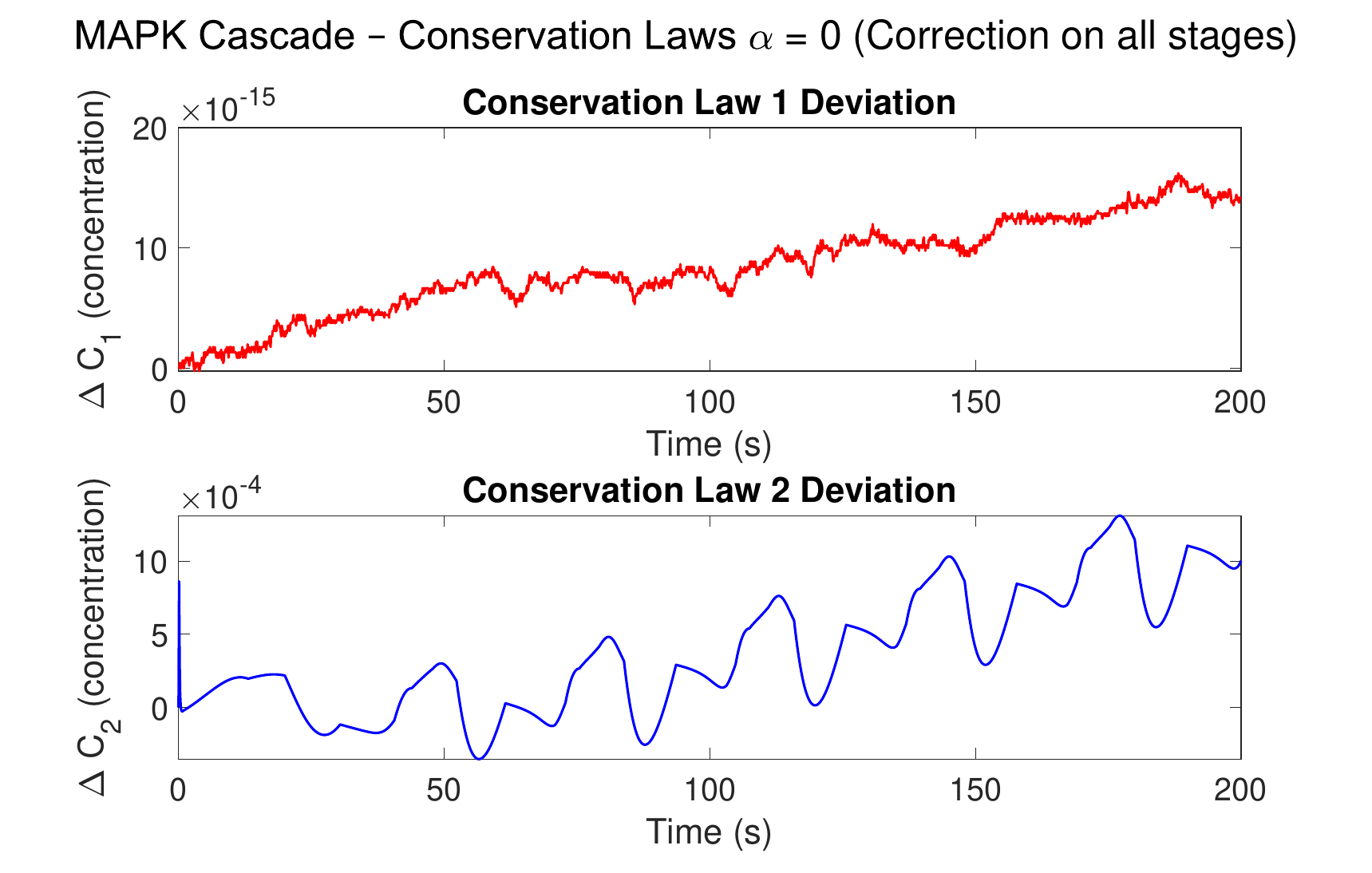}
\caption{MAPK Cascade \eqref{eqn:mapk-system}. Deviations from the conserved quantities with all-stage correction. Here, $C_1(t)=\mathbf{w}_1^\top \mathbf{y}(t)=y_1+y_4+y_6$ and $C_2(t)=\mathbf{w}_2^\top \mathbf{y}(t)=y_2+y_3+y_4+y_5$ denote the linear conservation laws of the MAPK cascade. The plotted quantities are deviations from the initial values, reported in concentration units.}
\label{fig:mapk-mass-0_all-stages}
\end{figure}

\begin{figure}[htbp]
\centering
\includegraphics[scale=0.36]{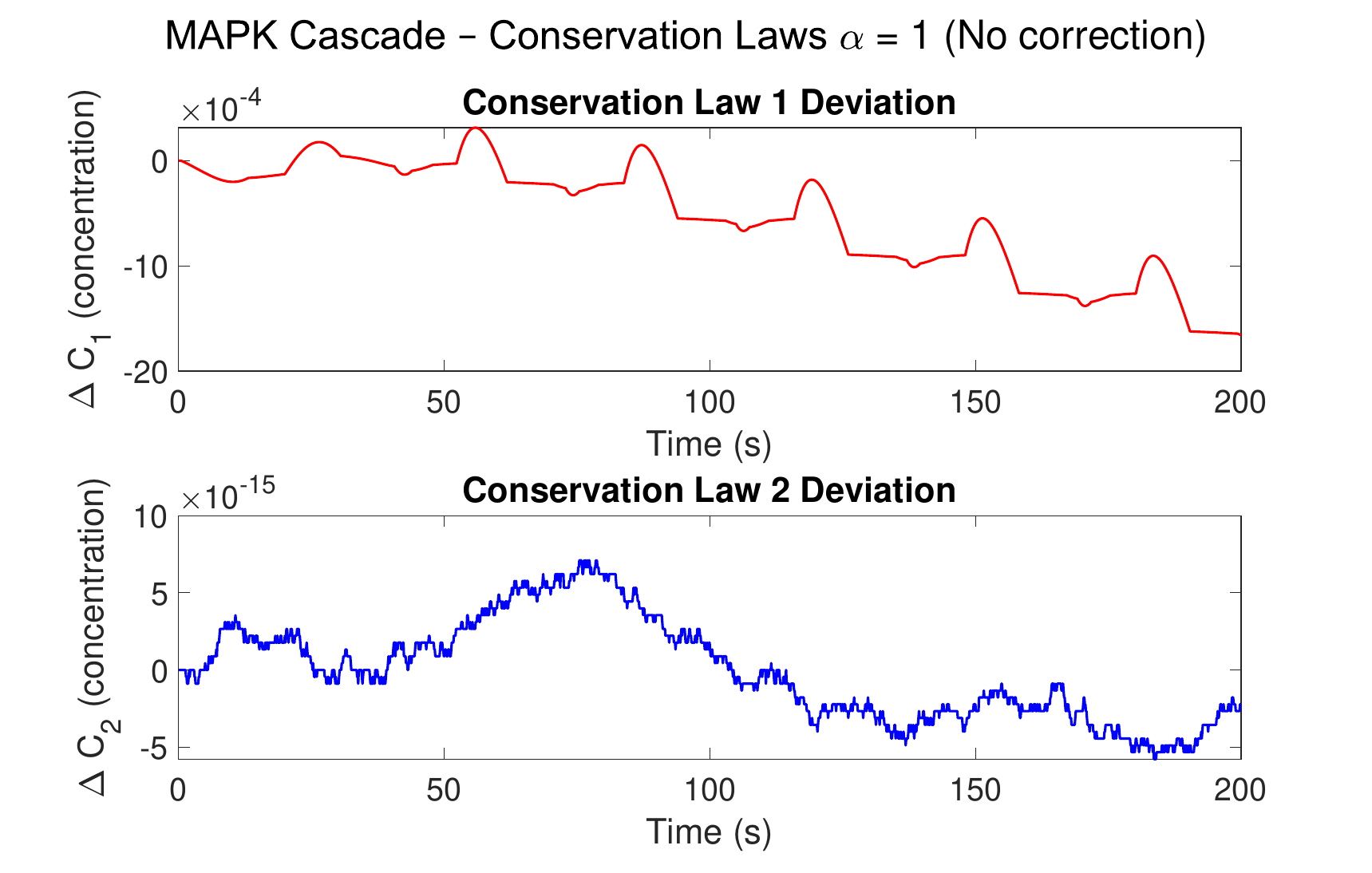}
\caption{MAPK Cascade \eqref{eqn:mapk-system}. Deviations from the conserved quantities for the baseline SDIRK scheme. Here, $C_1(t)=\mathbf{w}_1^\top \mathbf{y}(t)=y_1+y_4+y_6$ and $C_2(t)=\mathbf{w}_2^\top \mathbf{y}(t)=y_2+y_3+y_4+y_5$ denote the linear conservation laws of the MAPK cascade. The plotted quantities are deviations from the initial values, reported in concentration units.}
\label{fig:mapk-mass-1}
\end{figure}

\begin{figure}[htbp]
\centering
\includegraphics[scale=0.36]{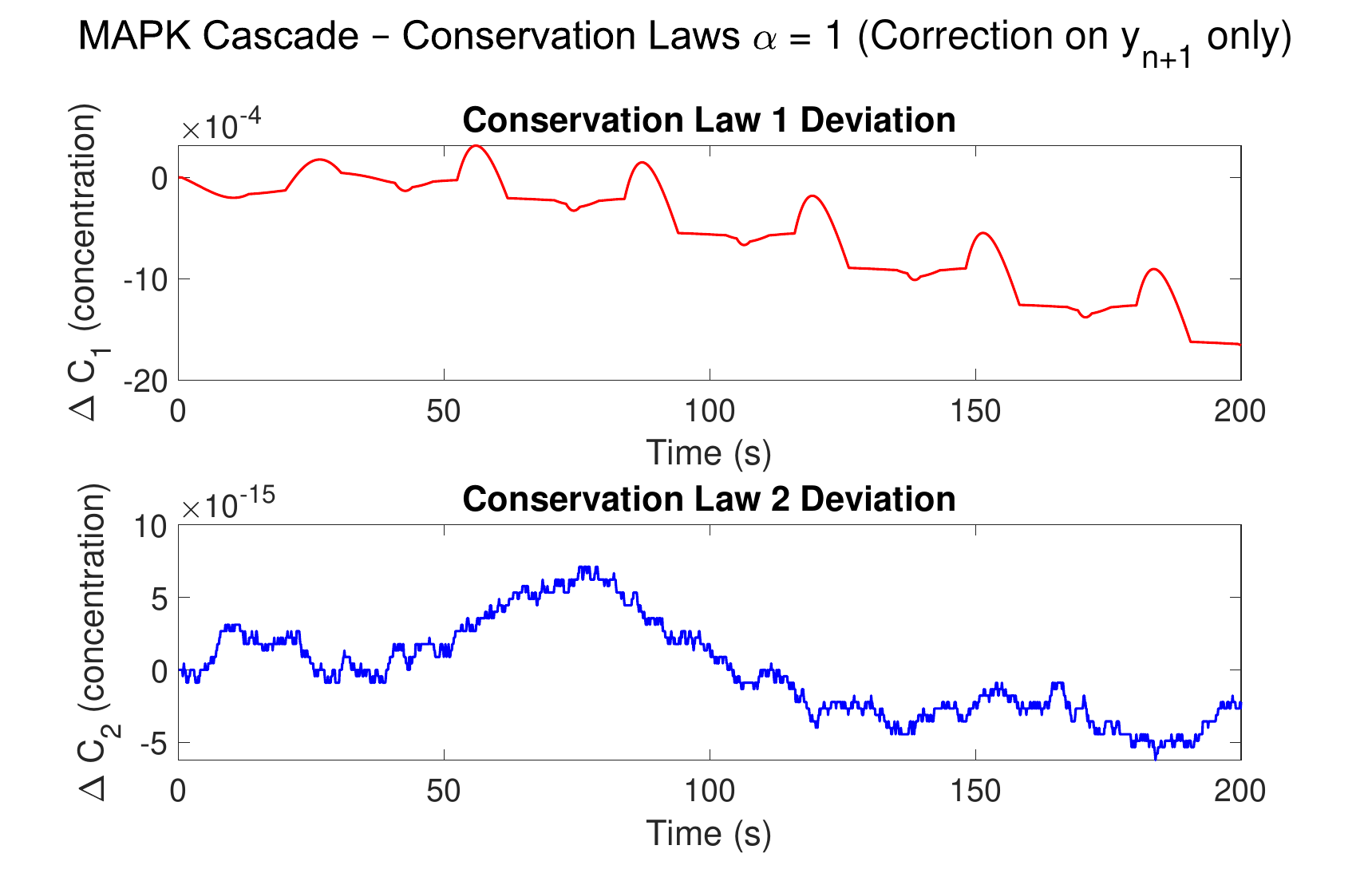}
\caption{MAPK Cascade \eqref{eqn:mapk-system}. Deviations from the conserved quantities with final-stage correction. Here, $C_1(t)=\mathbf{w}_1^\top \mathbf{y}(t)=y_1+y_4+y_6$ and $C_2(t)=\mathbf{w}_2^\top \mathbf{y}(t)=y_2+y_3+y_4+y_5$ denote the linear conservation laws of the MAPK cascade. The plotted quantities are deviations from the initial values, reported in concentration units.}
\label{fig:mapk-mass-1_last-stage}
\end{figure}

\begin{figure}[htbp]
\centering
\includegraphics[scale=0.36]{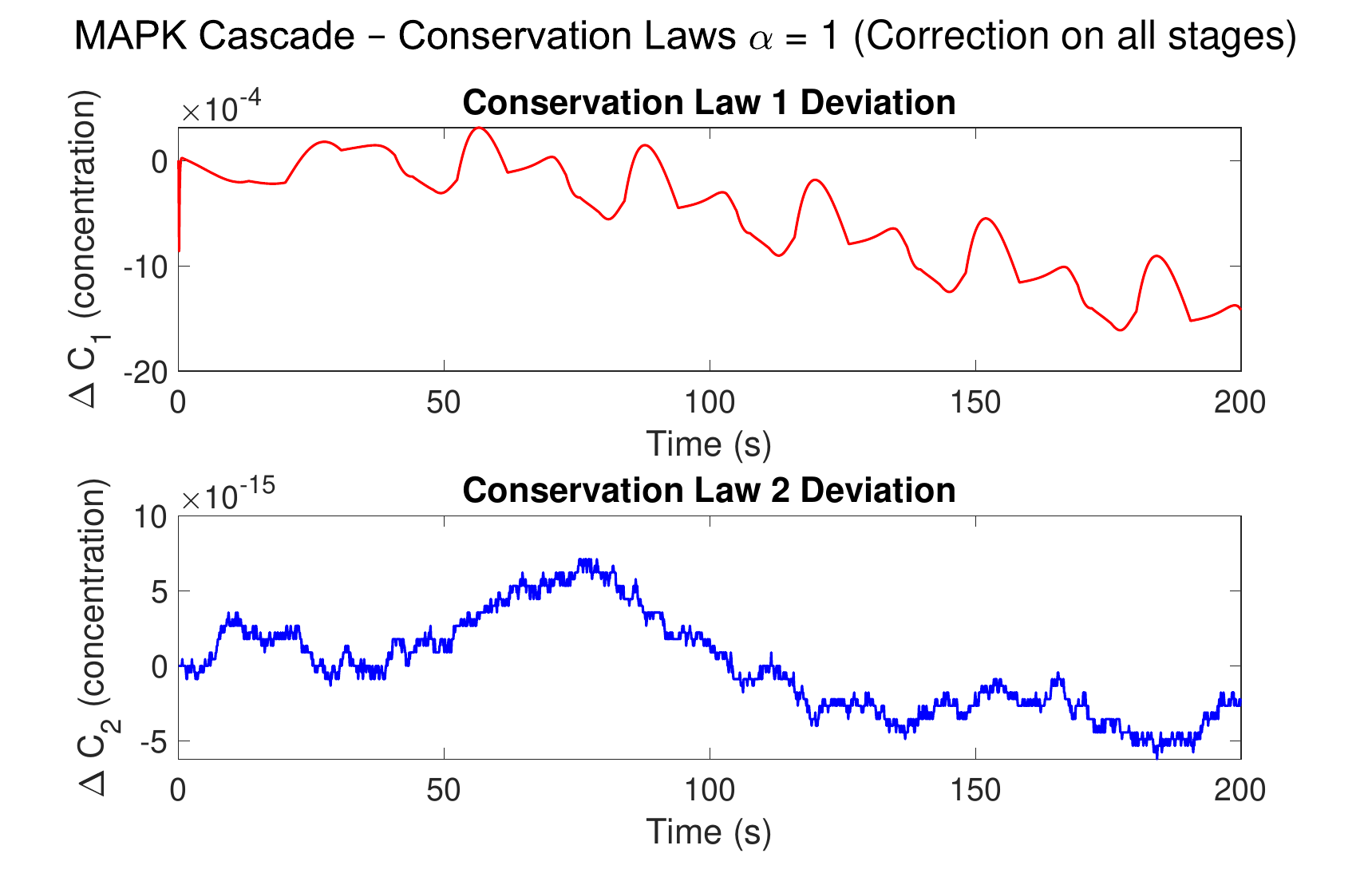}
\caption{MAPK Cascade \eqref{eqn:mapk-system}. Deviations from the conserved quantities with all-stage correction. Here, $C_1(t)=\mathbf{w}_1^\top \mathbf{y}(t)=y_1+y_4+y_6$ and $C_2(t)=\mathbf{w}_2^\top \mathbf{y}(t)=y_2+y_3+y_4+y_5$ denote the linear conservation laws of the MAPK cascade. The plotted quantities are deviations from the initial values, reported in concentration units.}
\label{fig:mapk-mass-1_all-stages}
\end{figure}

\subsubsection{Stratospheric Reaction}

The stratospheric system conserves the total number of oxygen and nitrogen atoms, defined by the invariants \eqref{eqn:strato-invariant}.

Table \ref{tab:stratospheric-invariant} summarizes the relative deviations in the
conserved oxygen and nitrogen atom counts for the stratospheric reaction system.
For all correction strategies, nitrogen conservation is maintained to machine
precision. Oxygen conservation is likewise preserved for the baseline and
final-stage corrected schemes, while a slightly larger deviation is observed when
the correction is applied at all internal stages. This increase remains small and
is attributed to accumulated round-off effects rather than a systematic loss of
conservation.

\begin{table}[htbp]
\centering
\caption{Stratospheric reaction invariant preservation over one day ($[12\cdot3600,\,36\cdot3600]$ seconds). Reported values are the maximum relative deviations $E_I$ defined in \eqref{eqn:invariant-error} for the total oxygen ($M_O$) and nitrogen ($M_N$) atom counts.}
\label{tab:stratospheric-invariant}
\begin{tabular}{lcc}
\toprule
Method & $E_{M_O}$ & $E_{M_N}$ \\
\midrule
Baseline SDIRK        & $4.89 \times 10^{-14}$ & $7.18 \times 10^{-15}$ \\
Final-stage correction & $4.89 \times 10^{-14}$ & $7.39 \times 10^{-15}$ \\
All-stage correction   & $1.49 \times 10^{-12}$ & $7.18 \times 10^{-15}$ \\
\bottomrule
\end{tabular}
\end{table}

\begin{figure}[htbp]
\centering
\includegraphics[scale=0.4]{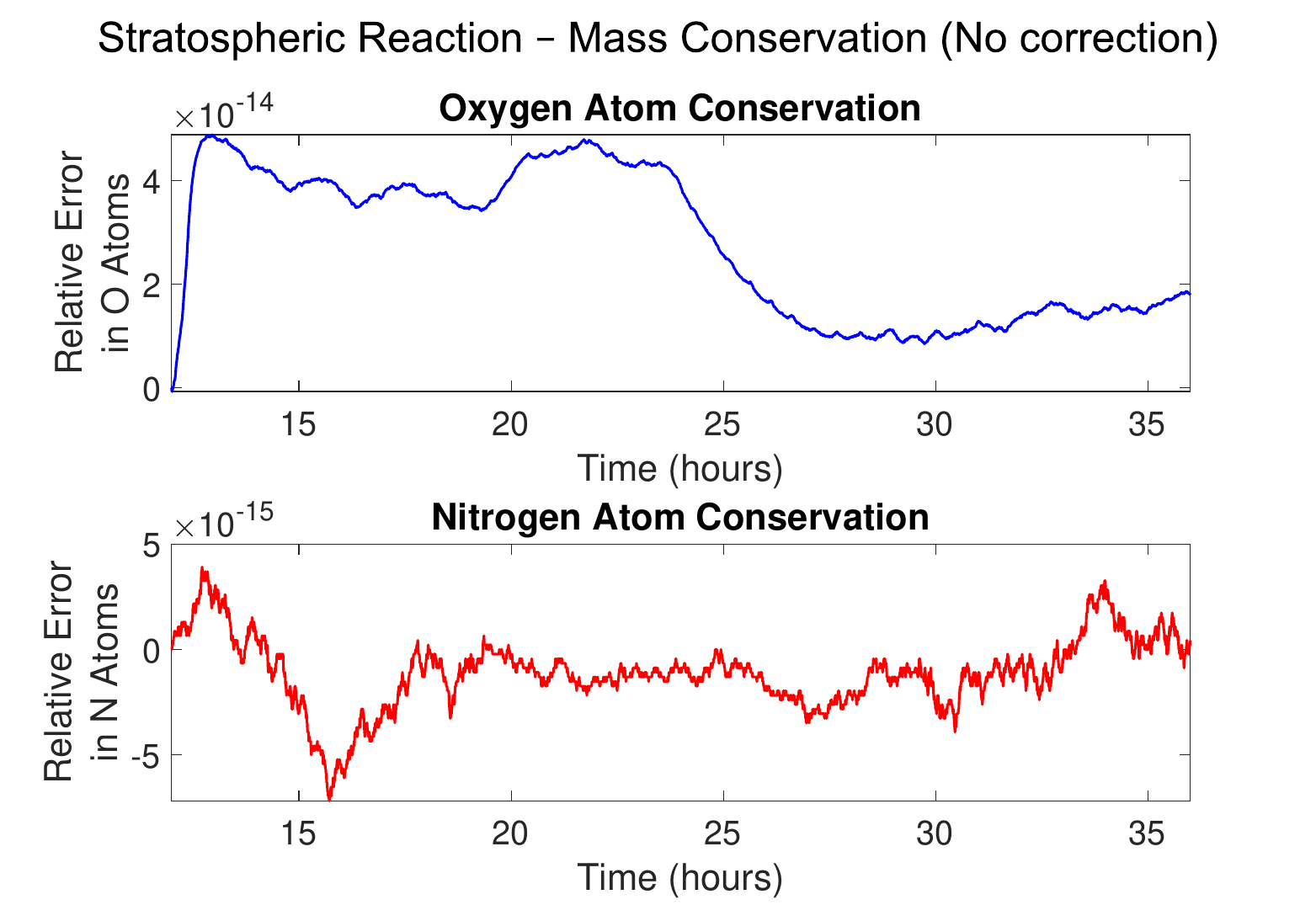}
\caption{Stratospheric Reaction \eqref{eqn:strato-ode}. Atom conservation for the baseline SDIRK scheme. The plotted quantities represent the relative error in total oxygen and nitrogen atom counts, normalized by their initial values.}
\label{fig:stratospheric-mass}
\end{figure}

\begin{figure}[htbp]
\centering
\includegraphics[scale=0.4]{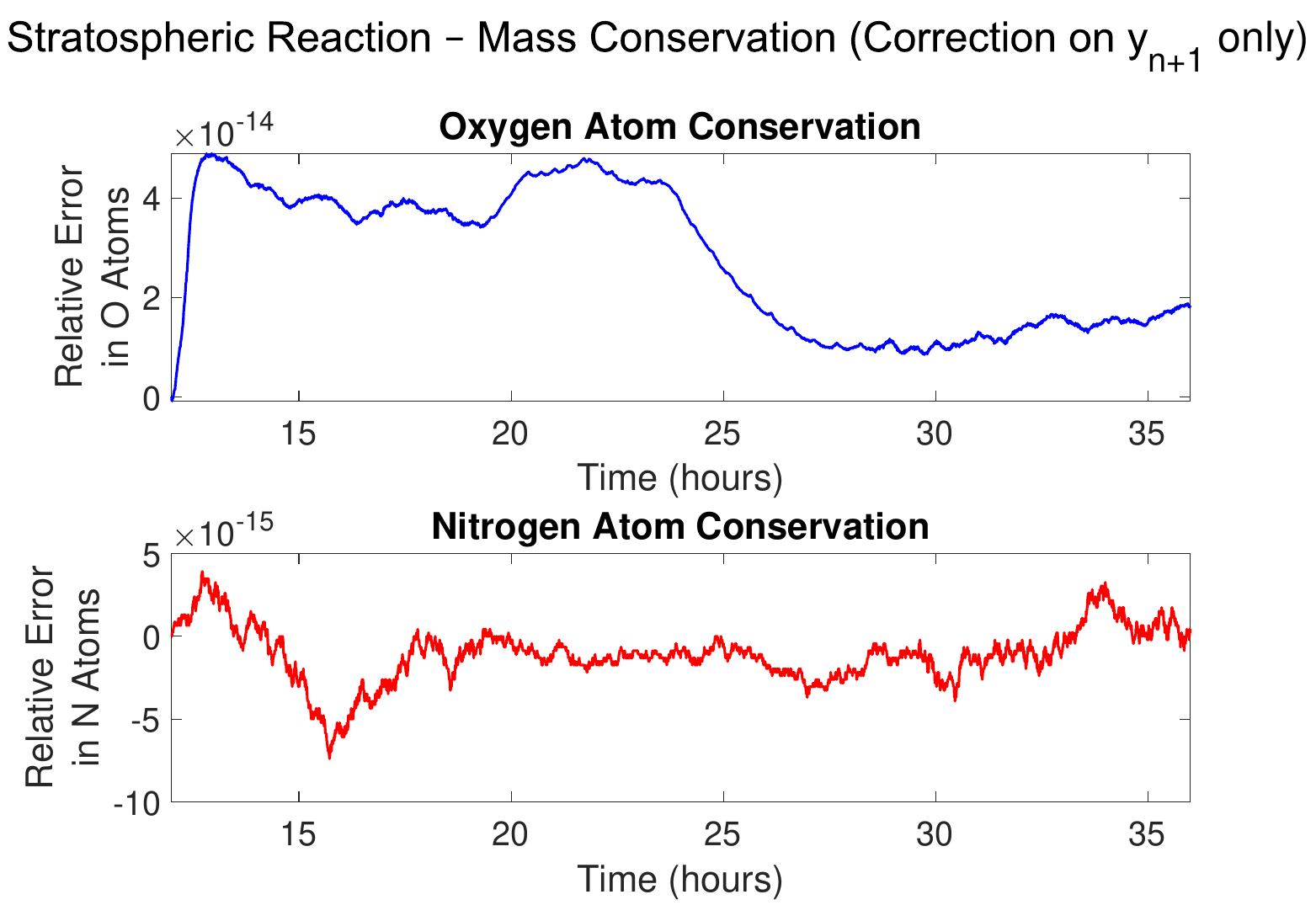}
\caption{Stratospheric Reaction \eqref{eqn:strato-ode}. Atom conservation with final-stage correction. The plotted quantities represent the relative error in total oxygen and nitrogen atom counts, normalized by their initial values.}
\label{fig:stratospheric-mass-last-stage}
\end{figure}

\begin{figure}[htbp]
\centering
\includegraphics[scale=0.4]{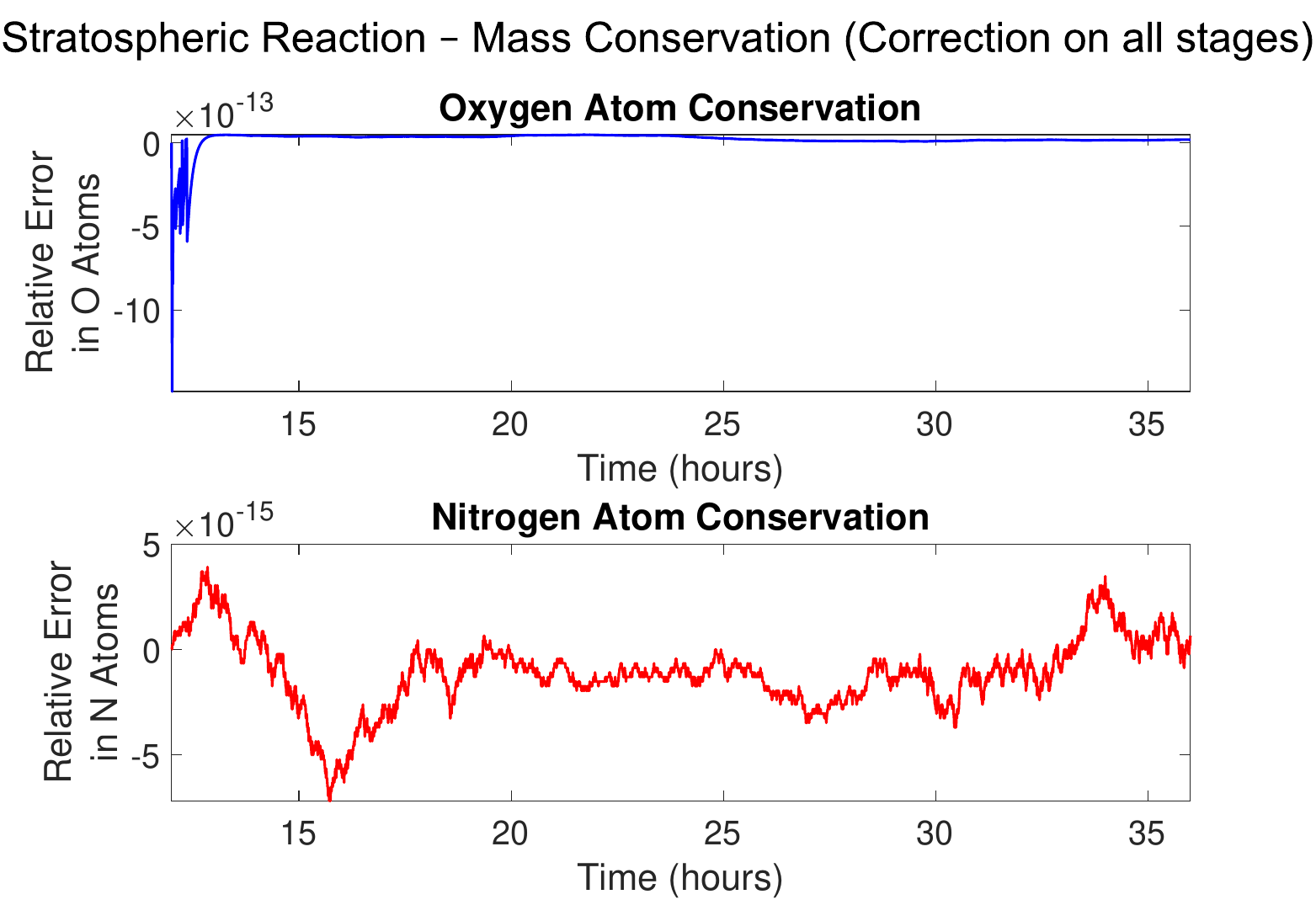}
\caption{Stratospheric Reaction \eqref{eqn:strato-ode}. Atom conservation with all-stage correction. The plotted quantities represent the relative error in total oxygen and nitrogen atom counts, normalized by their initial values.}
\label{fig:stratospheric-mass-all-stages}
\end{figure}

\subsubsection{KdV Equation}
The semi-discrete KdV equation preserves the total mass defined in \eqref{eqn:kdv-invariant}.

Over the simulation interval, the maximum relative mass deviation $E_M$ defined in \eqref{eqn:invariant-error} was:
\begin{itemize}
    \item Baseline SDIRK: $8.88 \times 10^{-16}$.
    \item Final-stage correction: $7.40 \times 10^{-16}$.
\end{itemize}

\begin{figure}[htbp]
\centering
\includegraphics[scale=0.5]{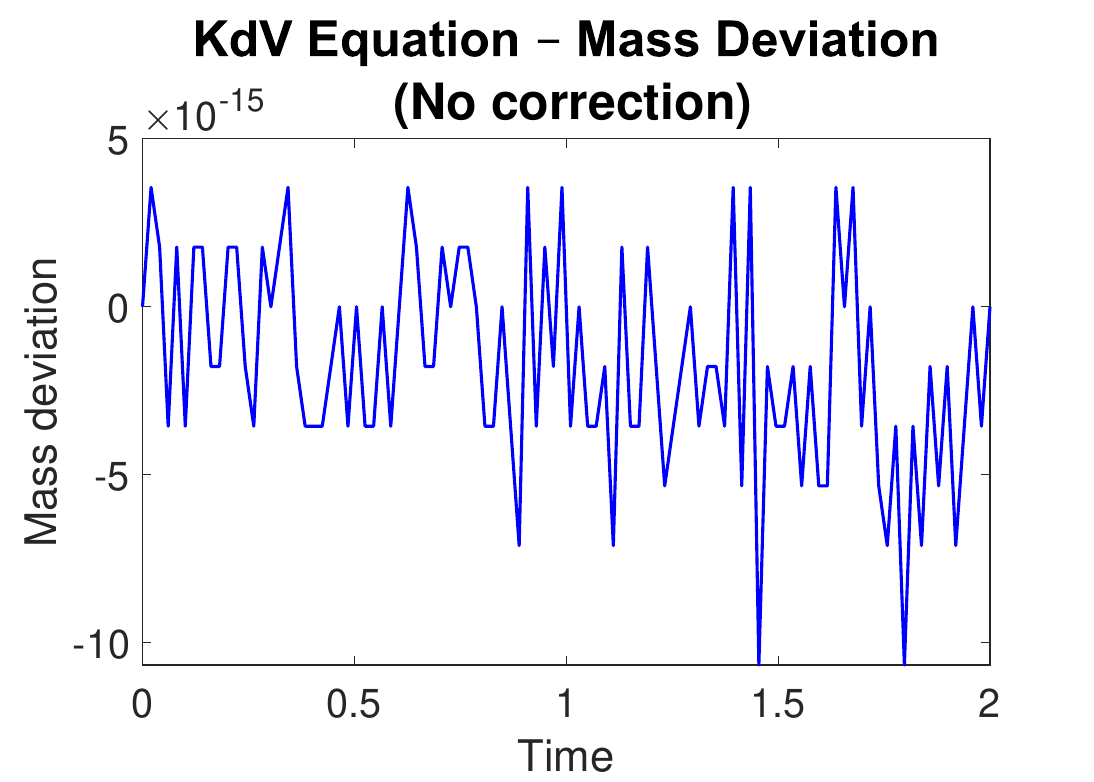}
\caption{KdV equation \eqref{eqn:kdv-pde}. Deviation of the discrete mass invariant \eqref{eqn:kdv-invariant} from its initial value for the baseline SDIRK scheme (nondimensional units).}
\label{fig:kdv-mass}
\end{figure}

\begin{figure}[htbp]
\centering
\includegraphics[scale=0.5]{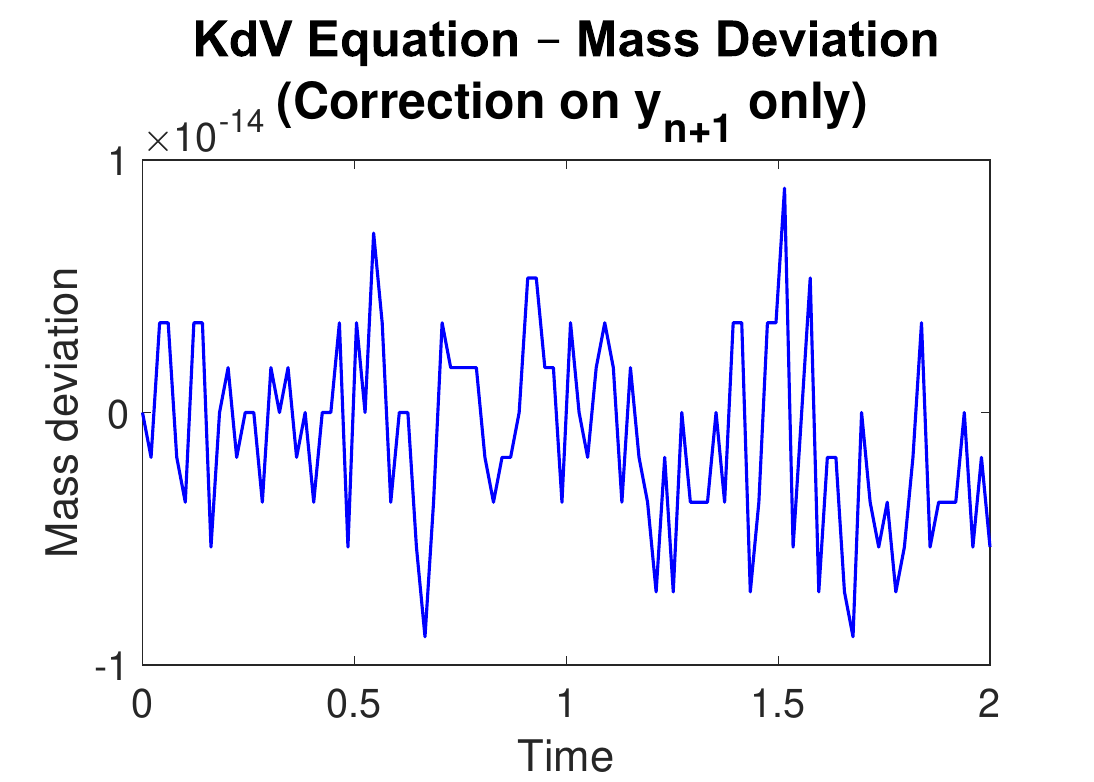}
\caption{KdV equation \eqref{eqn:kdv-pde}. Deviation of the discrete mass invariant \eqref{eqn:kdv-invariant} from its initial value with final-stage correction (nondimensional units).}
\label{fig:kdv-mass-last-stage}
\end{figure}

\subsubsection{Summary}

For all benchmark problems, the deviations in invariant quantities remained at or near machine precision whenever the conservation law is structurally exact. For example, total concentration in the Robertson reaction and total mass in the semi-discrete KdV equation.

For systems with multiple conservation laws where only one can be preserved by the construction of the $\mathbf{G}$ matrix (MAPK cascade and stratospheric chemistry), the preserved invariant remained accurate to machine precision. The other invariants showed small deviations, consistent with truncation errors.

Importantly, applying positivity-preserving corrections did not significantly deteriorate invariant preservation. In some cases (e.g., MAPK), the correction slightly reduced the deviation, while in others (e.g., oxygen conservation in the stratospheric system) it led to a marginal increase due to accumulated round-off. Overall, the proposed corrections retain the conservation properties of the underlying SDIRK schemes while enforcing positivity.

\subsection{Empirical Order Assesment}
\label{subsec:order}

To assess whether the positivity-preserving corrections affect the formal accuracy of the underlying methods, we performed order validation on four representative problems: the Robertson reaction system, the stratospheric reaction mechanism, the MAPK cascade, and KdV equation. These problems span a wide range of stiffness and nonlinear dynamics, providing a thorough testbed for the proposed approach.

The adaptive step-size integration mode was used for the ODE test problems. Relative and absolute tolerances were set equal ($\text{RelTol} = \text{AbsTol}$) and decreased simultaneously over several orders of magnitude until the asymptotic convergence regime was observed. For each tolerance value, the global error at the final time was compared to a reference solution computed with very tight absolute and relative tolerances both equal to ($10^{-14}$). For the KdV equation, a fixed time-step discretization was employed. Observed convergence slopes were obtained from a least-squares linear fit of the log-log data. Throughout this section, SDIRK21, SDIRK32, and SDIRK43 denote second-, third-, and fourth-order methods, respectively. Observed convergence slopes are reported accordingly.

We compared two variants of the positivity-preserving correction:
\begin{enumerate}
    \item \textbf{Final-stage correction:} only the stiffly accurate stage $\mathbf{y}_{n+1}$ is clipped and rescaled.
    \item \textbf{Full-stage correction:} all intermediate stages $\mathbf{Y}_i$ are clipped before being used in the final combination.
\end{enumerate}

\subsubsection{Robertson Reaction}

The Robertson reaction is a classic stiff chemical kinetics model with widely separated time scales. To avoid dominance of long-term truncation error, we restricted the integration interval to $[0, 5000]$ with absolute and relative tolerances from $10^{-5}$ to $10^{-8}$.

For the final-stage correction (Figure \ref{fig:sdirk-rob_last-stage}) and for full-stage correction (Figure \ref{fig:sdirk-rob_all-stages}) these results confirm second-, third-, and fourth-order convergence for SDIRK21, SDIRK32, and SDIRK43, respectively, even in the highly stiff Robertson system.

\begin{figure}[htbp]
    \centering
    \begin{subfigure}[b]{0.45\textwidth}
        \centering
        \includegraphics[width=\textwidth]{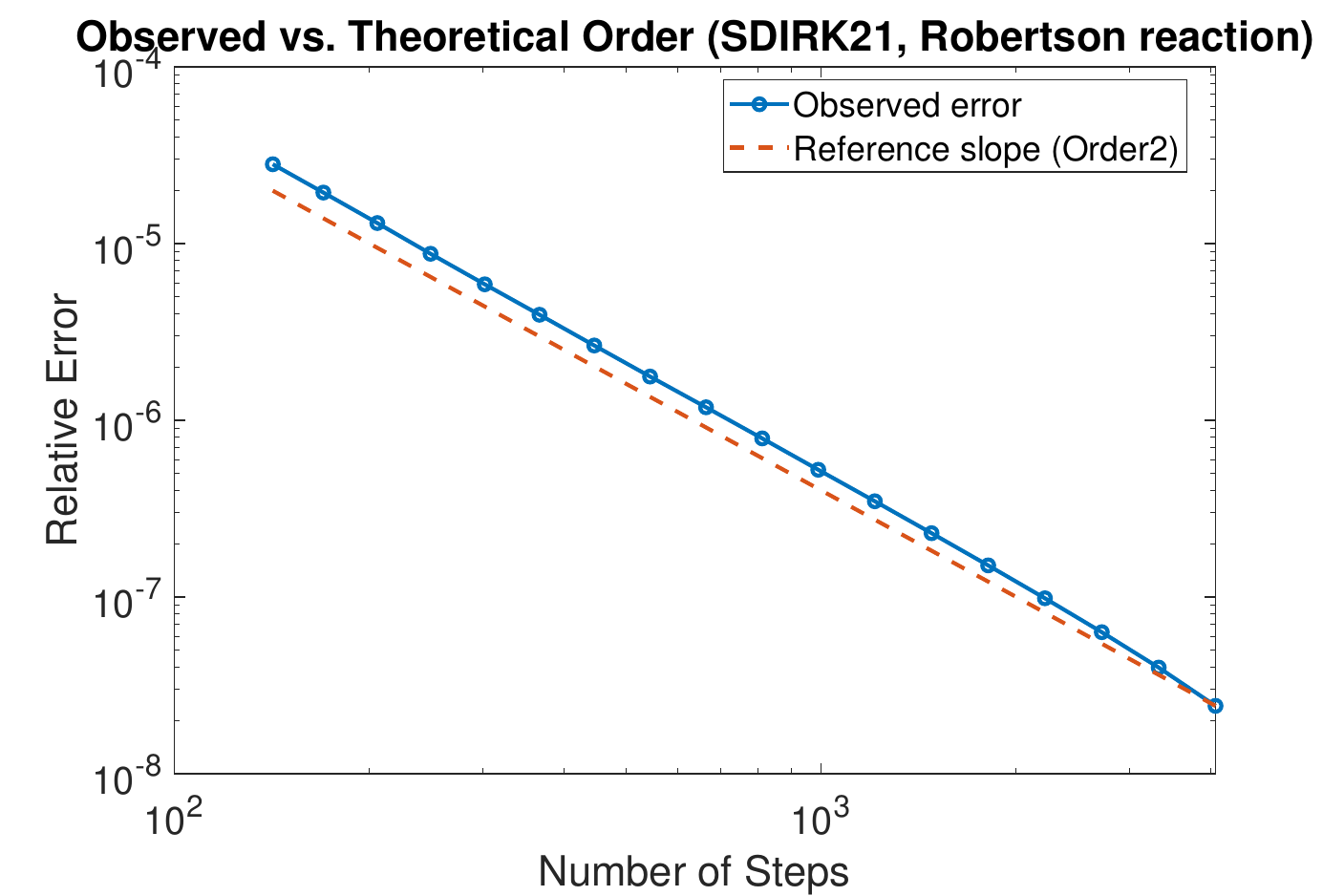}
        \caption{SDIRK21}
        \label{sfig:sdirk-rob-2_last-stage}
    \end{subfigure}
    \begin{subfigure}[b]{0.45\textwidth}
        \centering
        \includegraphics[width=\textwidth]{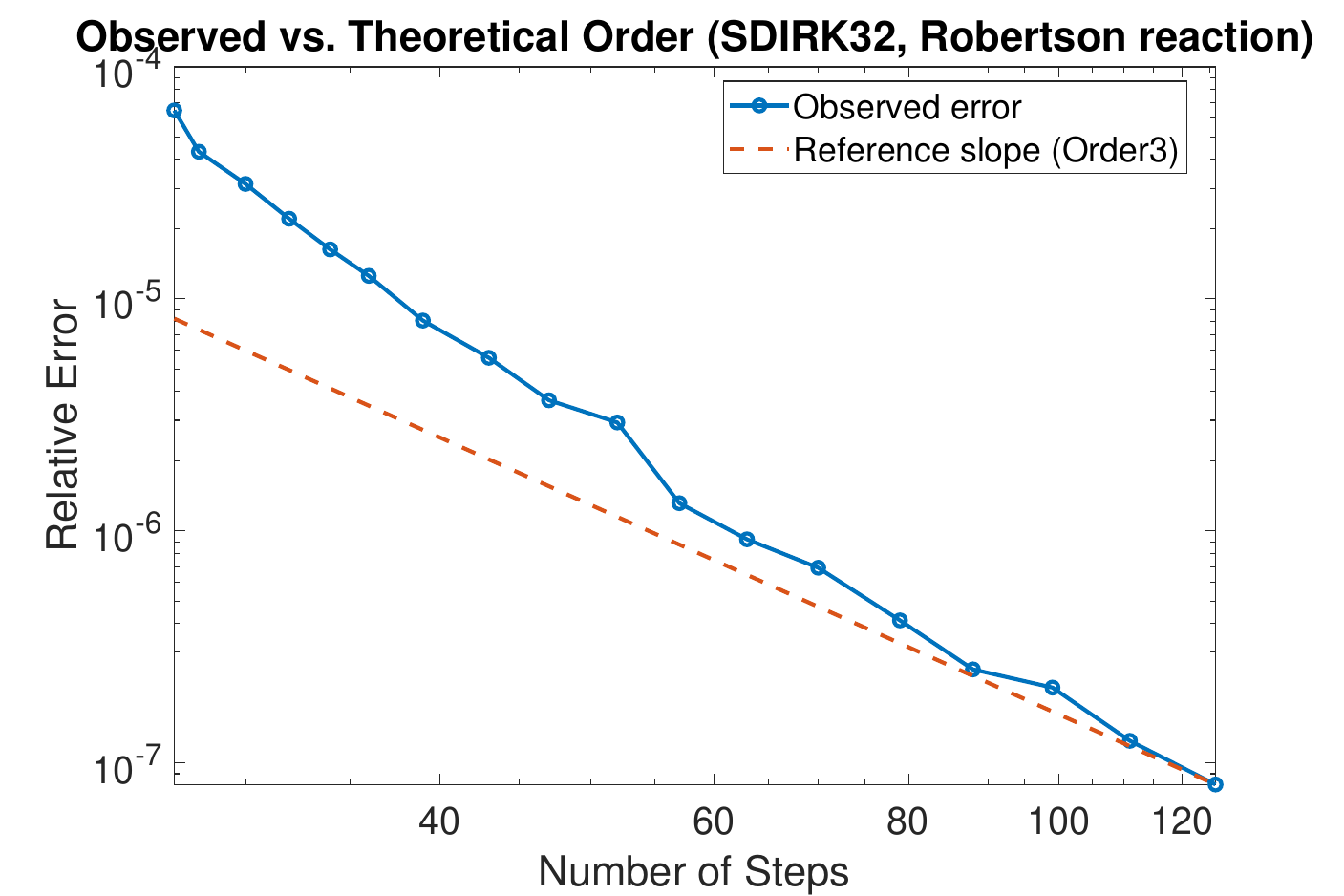}
        \caption{SDIRK32}
        \label{sfig:sdirk-rob-3_last-stage}
    \end{subfigure}
    \begin{subfigure}[b]{0.45\textwidth}
        \centering
        \includegraphics[width=\textwidth]{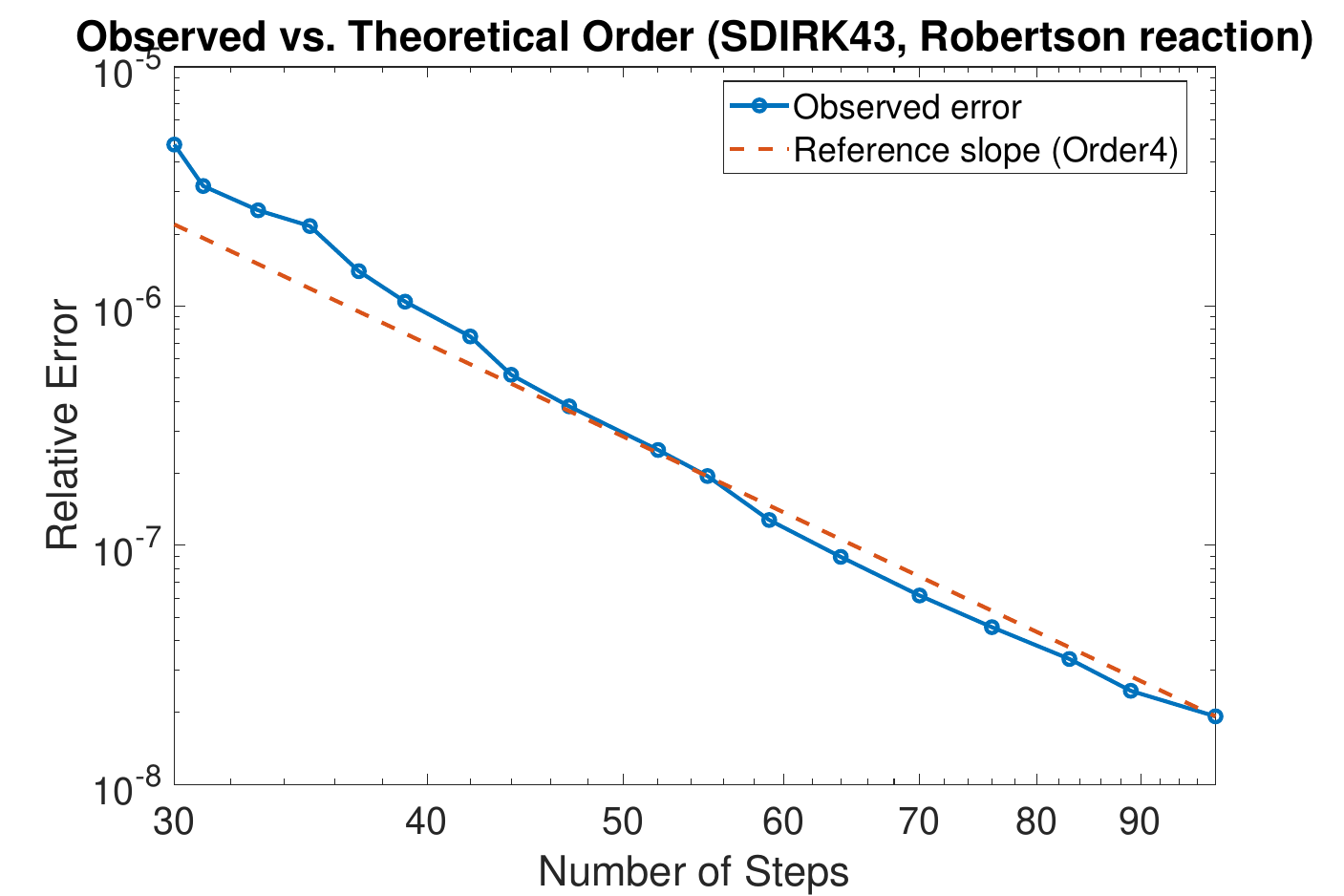}
        \caption{SDIRK43}
        \label{sfig:sdirk-rob-4_last-stage}
    \end{subfigure}
       \caption{Robertson reaction system \eqref{eqn:Robertson}. Observed convergence orders for SDIRK methods with positivity correction applied only to the final stage.}
       \label{fig:sdirk-rob_last-stage}
\end{figure}

\begin{figure}[htbp]
    \centering
    \begin{subfigure}[b]{0.49\textwidth}
        \centering
        \includegraphics[width=\textwidth]{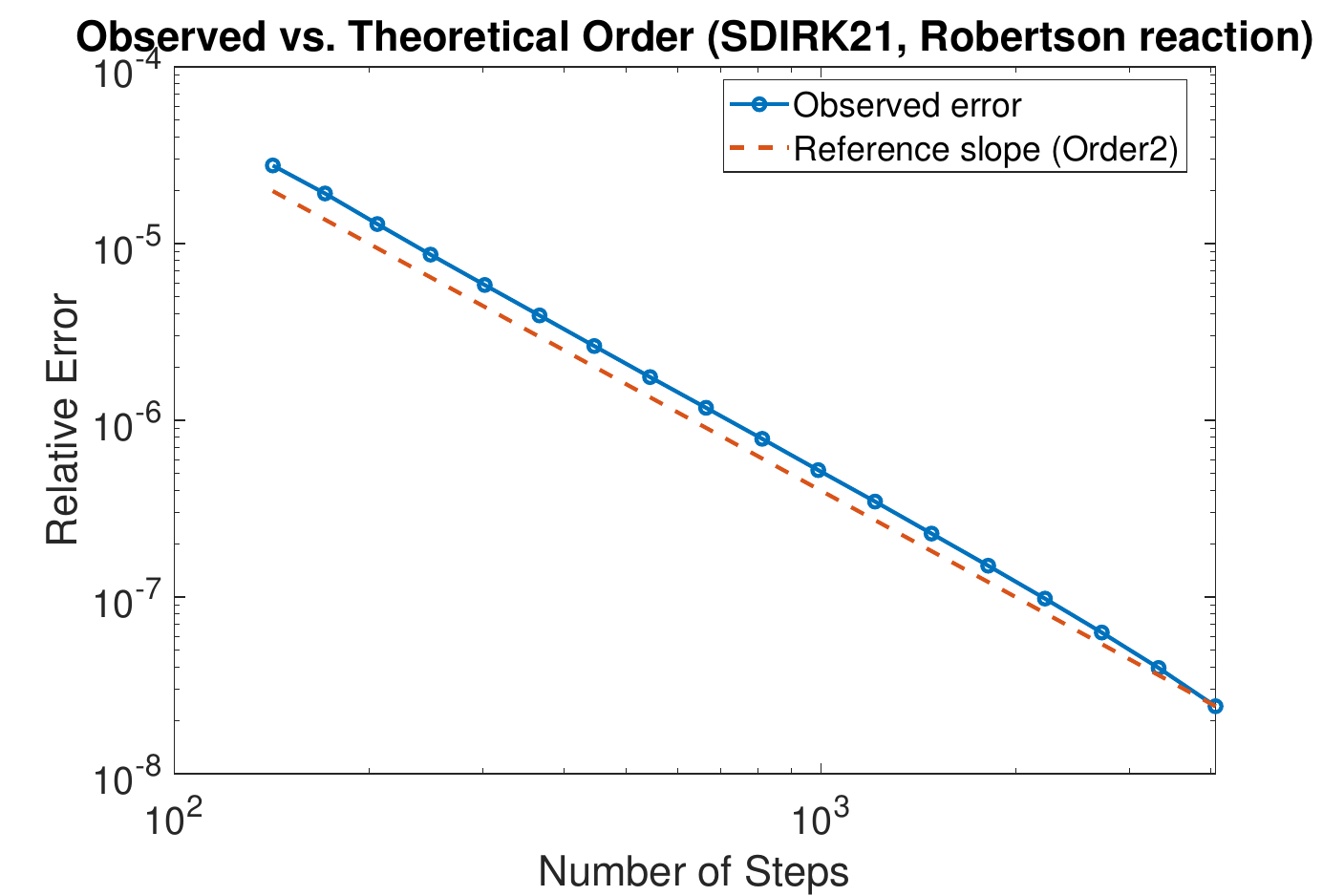}
        \caption{SDIRK21}
        \label{sfig:sdirk-rob-2_all-stages}
    \end{subfigure}
    \begin{subfigure}[b]{0.49\textwidth}
        \centering
        \includegraphics[width=\textwidth]{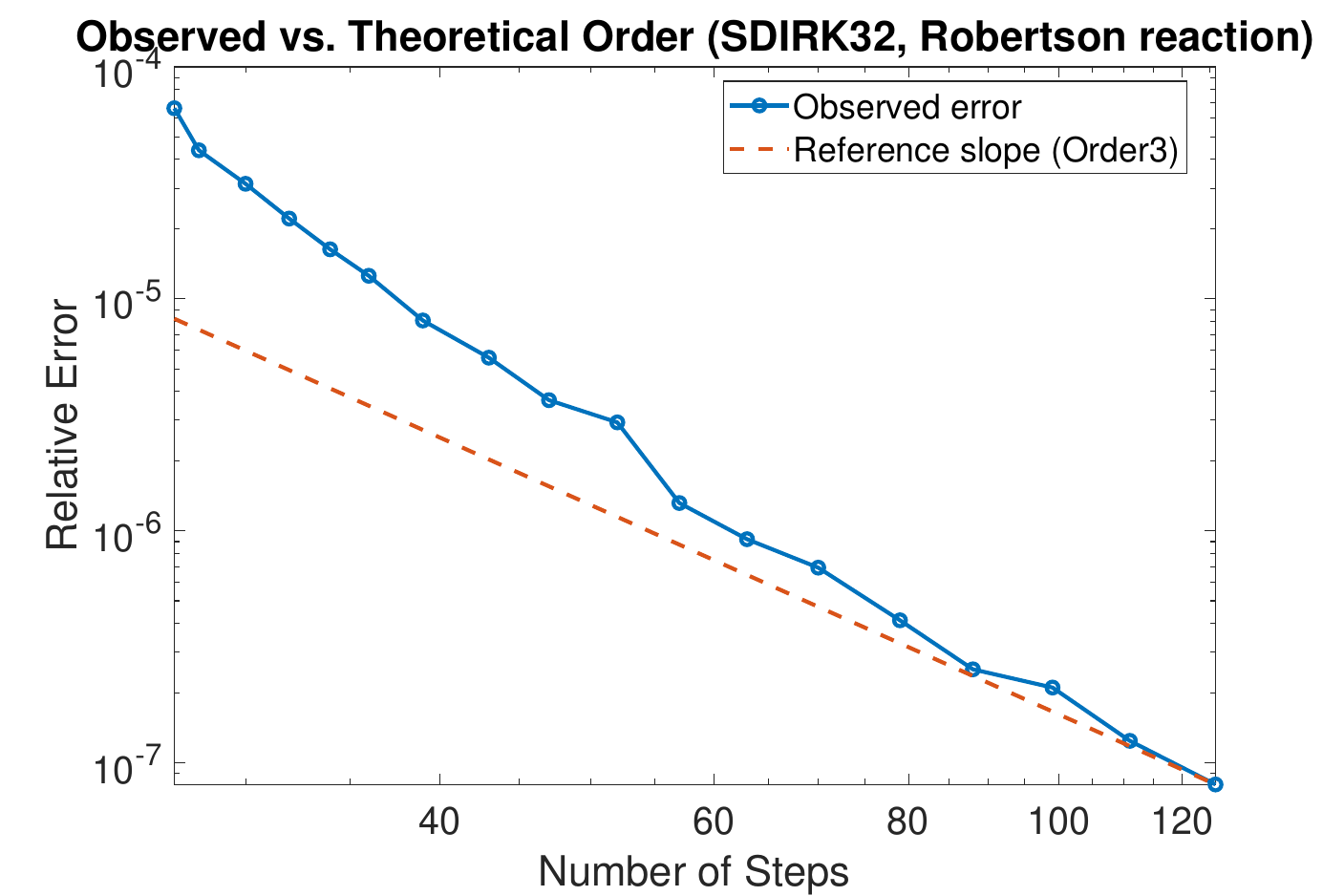}
        \caption{SDIRK32}
        \label{sfig:sdirk-rob-3_all-stages}
    \end{subfigure}
    \begin{subfigure}[b]{0.49\textwidth}
        \centering
        \includegraphics[width=\textwidth]{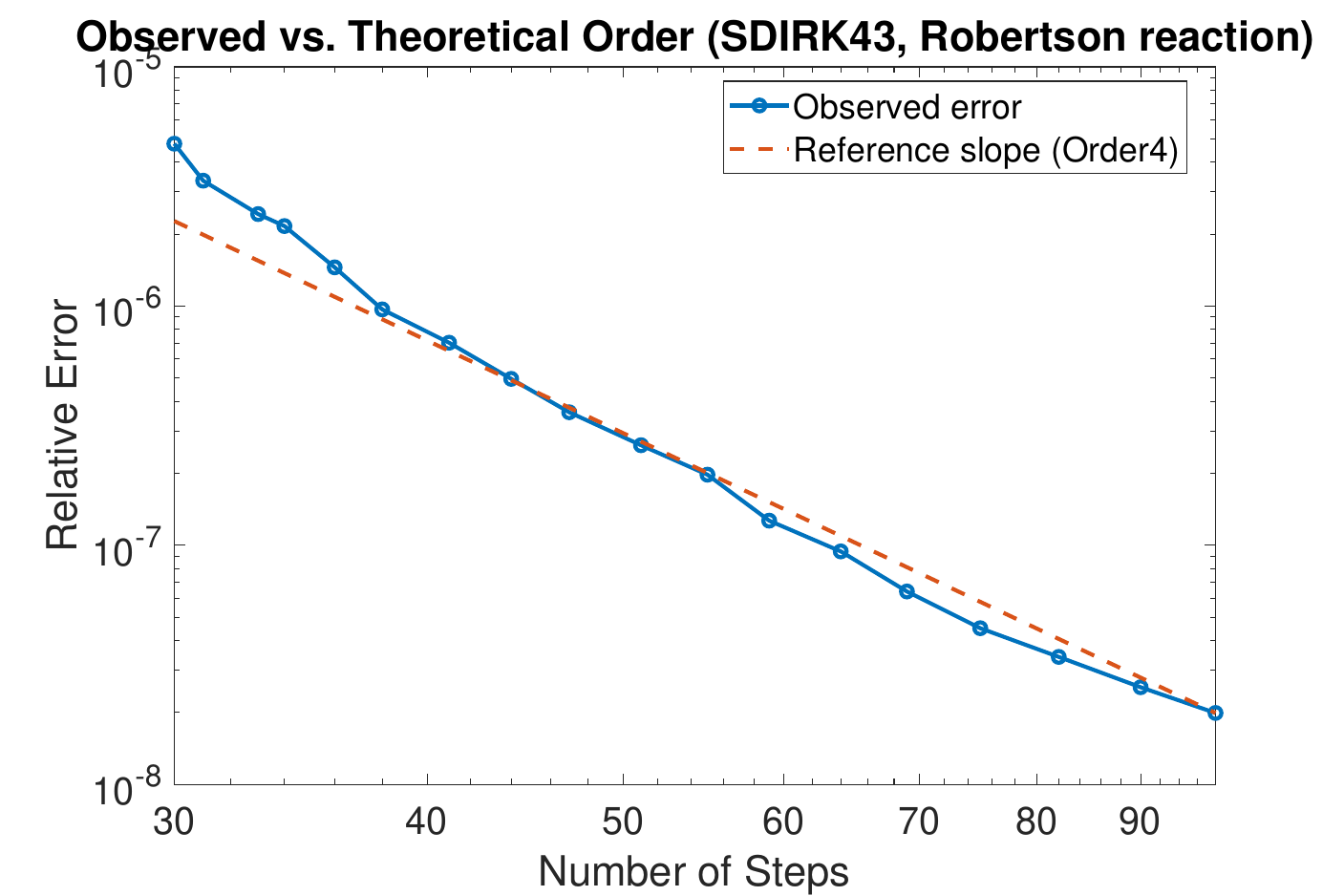}
        \caption{SDIRK43}
        \label{sfig:sdirk-rob-4_all-stages}
    \end{subfigure}
       \caption{Robertson reaction system \eqref{eqn:Robertson}. Observed convergence orders for SDIRK methods with positivity correction applied to all stages.}
       \label{fig:sdirk-rob_all-stages}
\end{figure}

\subsubsection{MAPK Cascade}

The MAPK cascade is a moderately stiff multiscale biochemical signaling network. To observe the asymptotic convergence regime, we integrated over the interval $[0,60]$ using SDIRK21, SDIRK32, and SDIRK43 with absolute ans relative tolerances ranging from $10^{-5}$ to $10^{-8}$; the reference solution was computed with tolerance $10^{-14}$.

With final-stage correction, the measured slopes were close to the theoretical values, confirming that the correction does not alter the convergence rate (Figure \ref{fig:sdirk-mapk_last-stage}). With full-stage correction, the slopes remained essentially unchanged (Figure \ref{fig:sdirk-mapk_all-stages}). The measured slopes are consistent with the nominal orders of the underlying SDIRK methods, indicating no order degradation due to the correction.

\begin{figure}[htbp]
    \centering
    \begin{subfigure}[b]{0.49\textwidth}
        \centering
        \includegraphics[width=\textwidth]{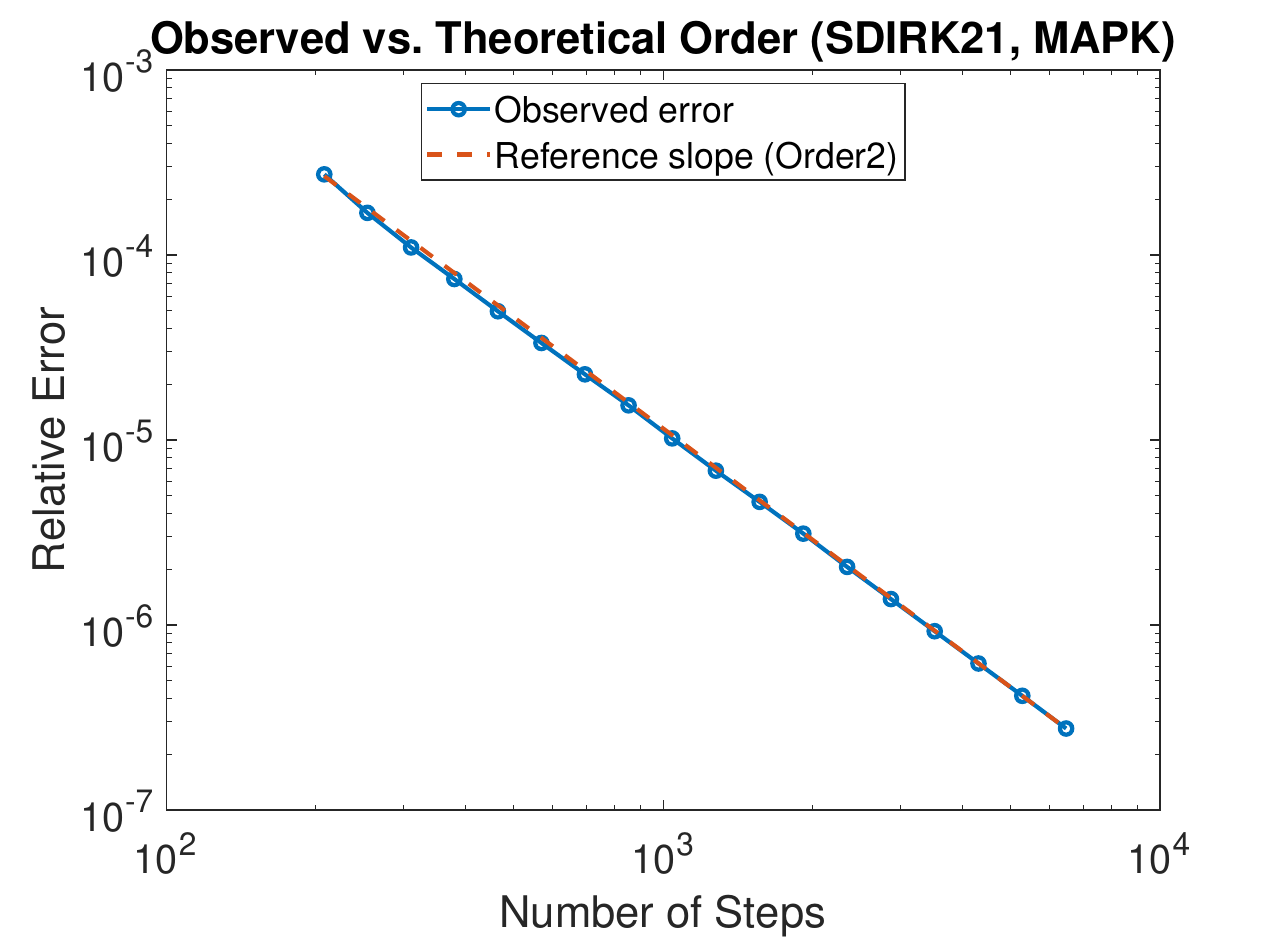}
        \caption{SDIRK21}
        \label{sfig:sdirk-mapk-2_last-stage}
    \end{subfigure}
    \begin{subfigure}[b]{0.49\textwidth}
        \centering
        \includegraphics[width=\textwidth]{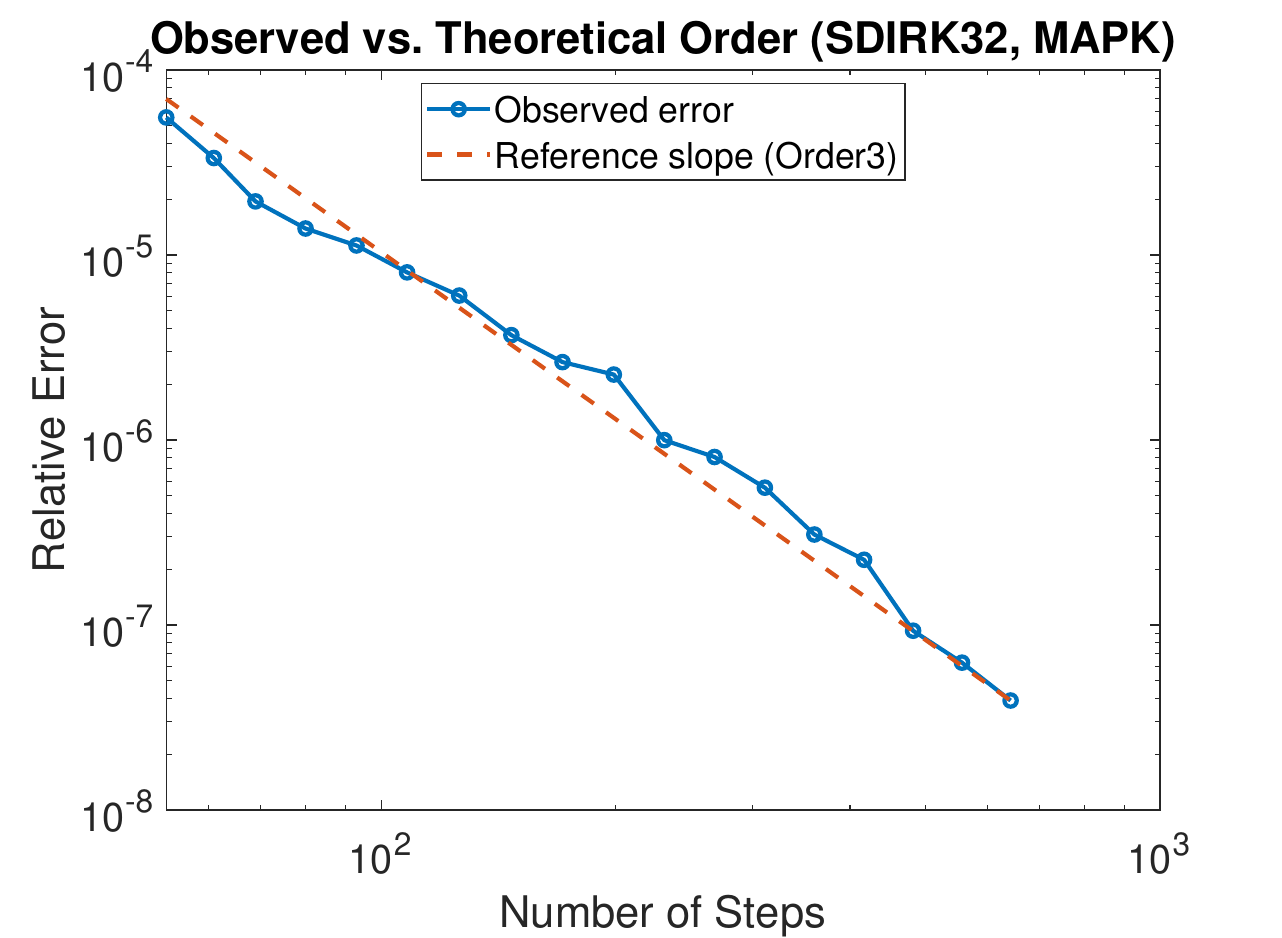}
        \caption{SDIRK32}
        \label{sfig:sdirk-mapk-3_last-stage}
    \end{subfigure}
    \begin{subfigure}[b]{0.49\textwidth}
        \centering
        \includegraphics[width=\textwidth]{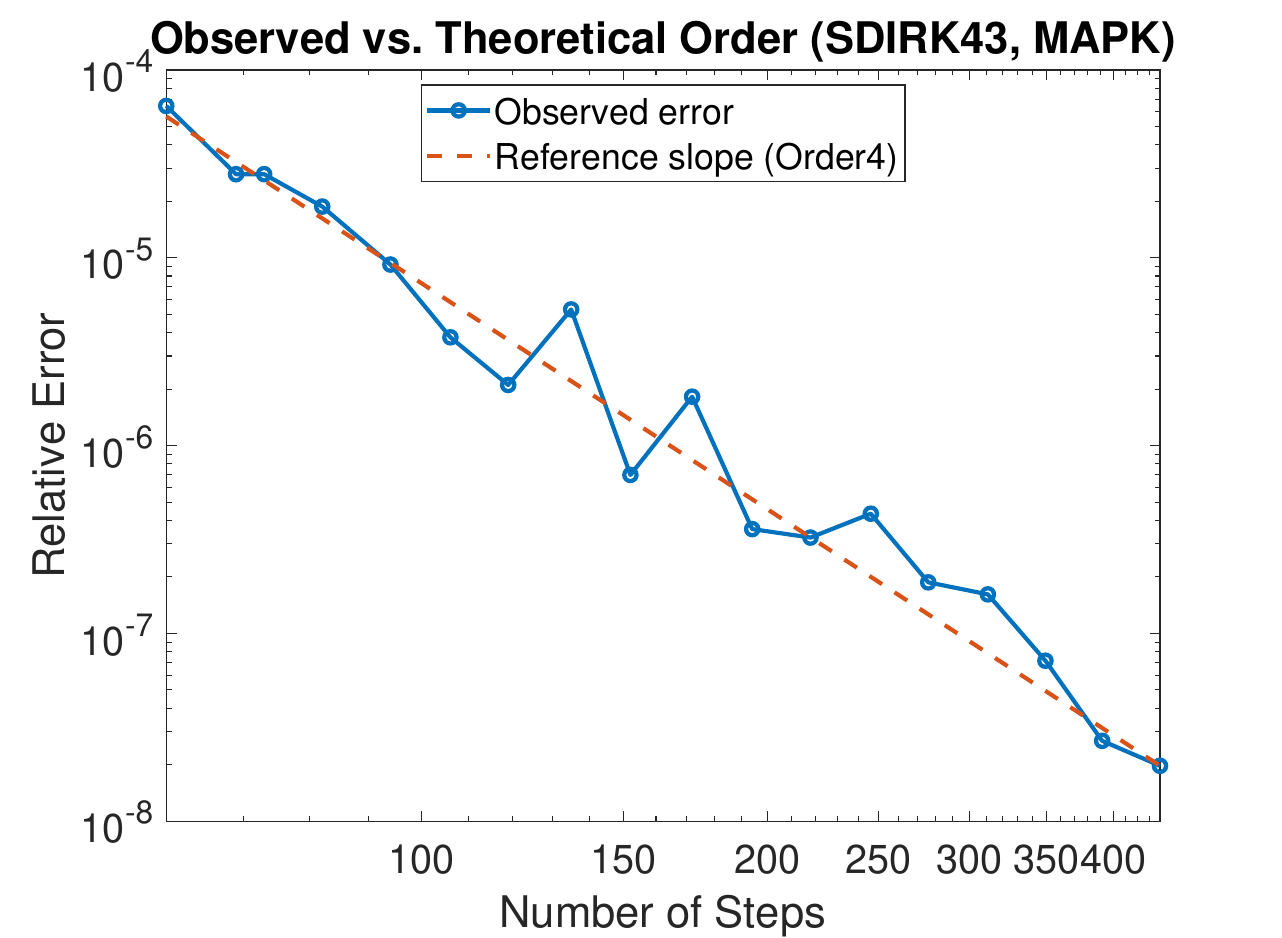}
        \caption{SDIRK43}
        \label{sfig:sdirk-mapk-4_last-stage}
    \end{subfigure}
       \caption{MAPK Cascade \eqref{eqn:mapk-system}. Observed convergence orders for SDIRK methods  with positivity correction applied only to the final stage.}
       \label{fig:sdirk-mapk_last-stage}
\end{figure}

\begin{figure}[htbp]
    \centering
    \begin{subfigure}[b]{0.49\textwidth}
        \centering
        \includegraphics[width=\textwidth]{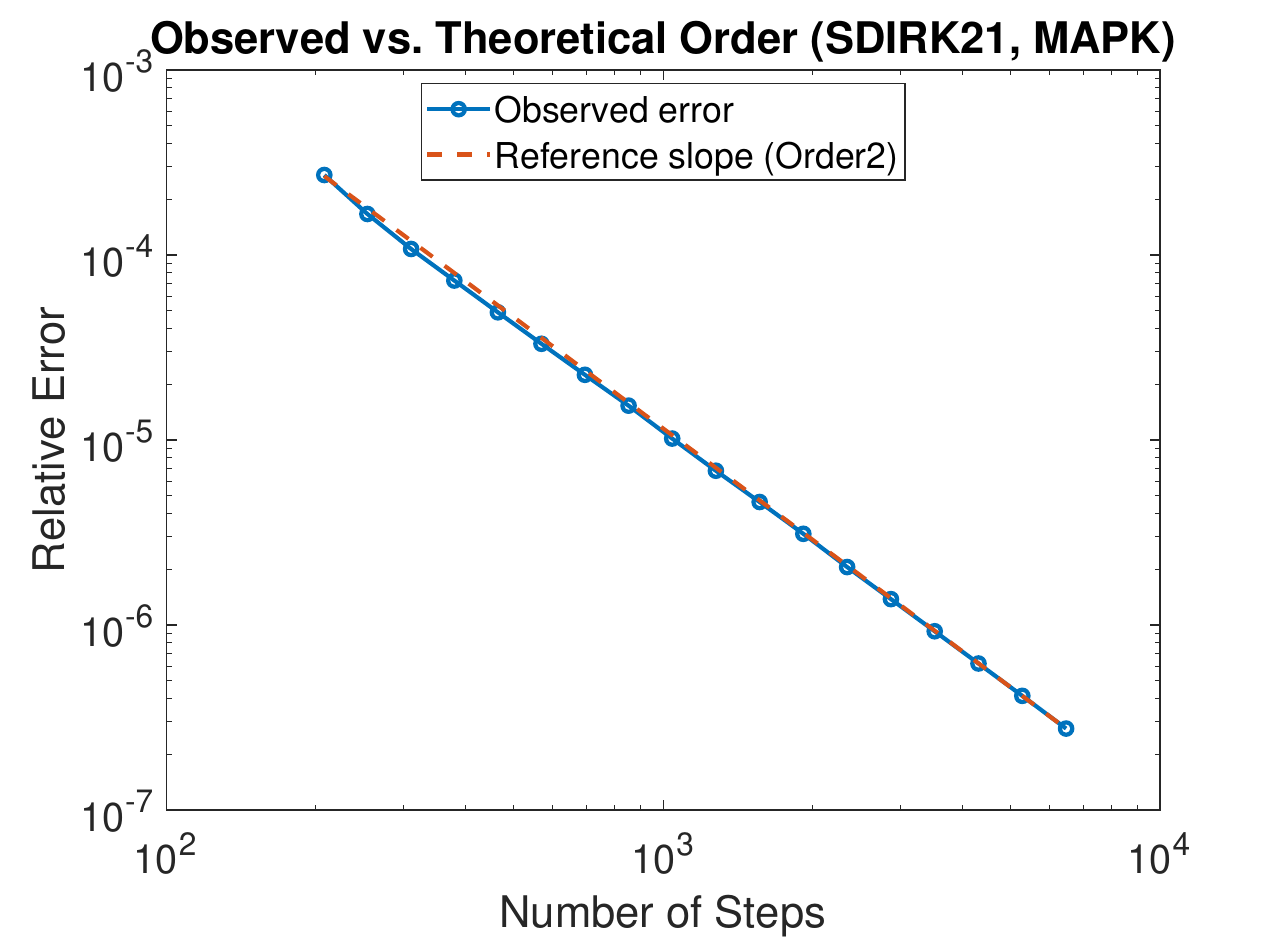}
        \caption{SDIRK21}
        \label{sfig:sdirk-mapk-2_all-stages}
    \end{subfigure}
    \begin{subfigure}[b]{0.49\textwidth}
        \centering
        \includegraphics[width=\textwidth]{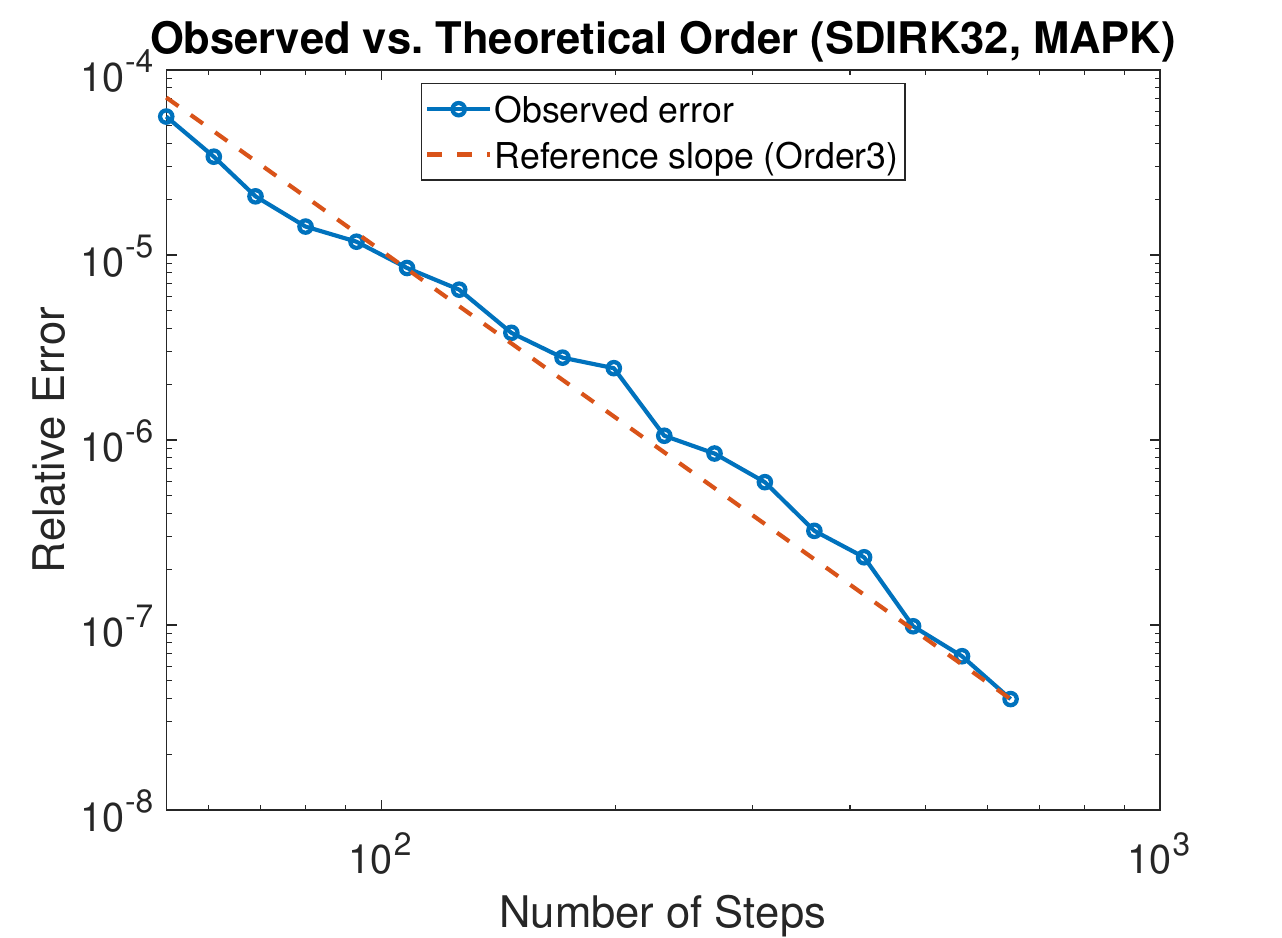}
        \caption{SDIRK32}
        \label{sfig:sdirk-mapk-3_all-stages}
    \end{subfigure}
    \begin{subfigure}[b]{0.49\textwidth}
        \centering
        \includegraphics[width=\textwidth]{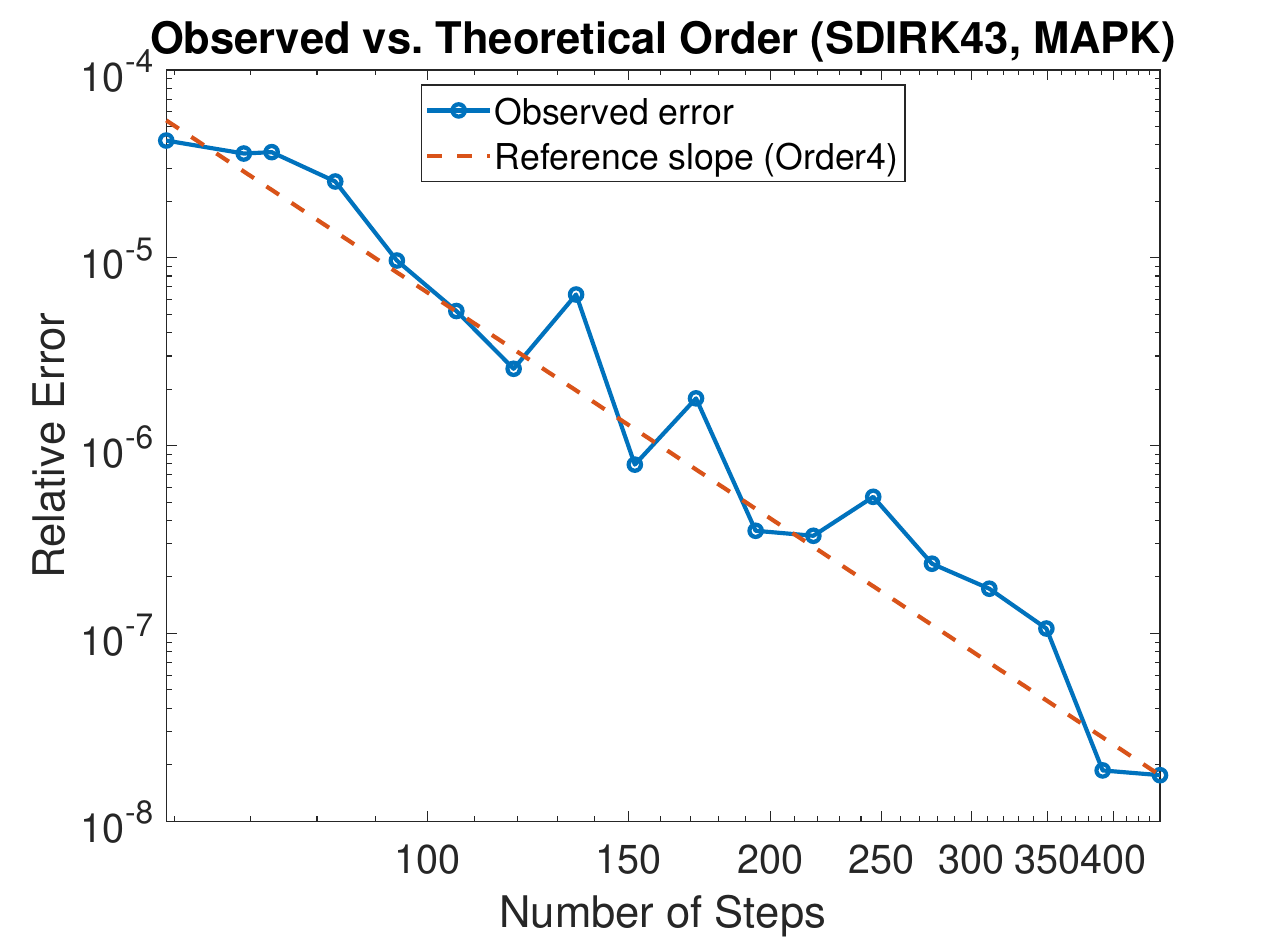}
        \caption{SDIRK43}
        \label{sfig:sdirk-mapk-4_all-stages}
    \end{subfigure}
       \caption{MAPK Cascade \eqref{eqn:mapk-system}. Observed convergence orders for SDIRK methods with positivity correction applied only to all stages.}
       \label{fig:sdirk-mapk_all-stages}
\end{figure}

\subsubsection{Stratospheric Reaction}

The stratospheric reaction mechanism is an extremely stiff, positivity-sensitive atmospheric chemistry model. We integrated over the interval $[19 \cdot 3600, 29 \cdot 3600]$ seconds to capture a representative stiff regime.

For the final-stage correction (Figure \ref{fig:sdirk-strat_last-stage}) the higher-than-expected slopes for the higher-order methods reflect rapid error damping typical of extremely stiff systems.

When applying full-stage correction, the slopes slightly decreased but still followed the expected order trend  (Figure \ref{fig:sdirk-strat_all-stages}). Thus, even in extreme stiffness, full-stage clipping maintains the nominal order without introducing significant degradation.

\begin{figure}[htbp]
    \centering
    \begin{subfigure}[b]{0.49\textwidth}
        \centering
        \includegraphics[width=\textwidth]{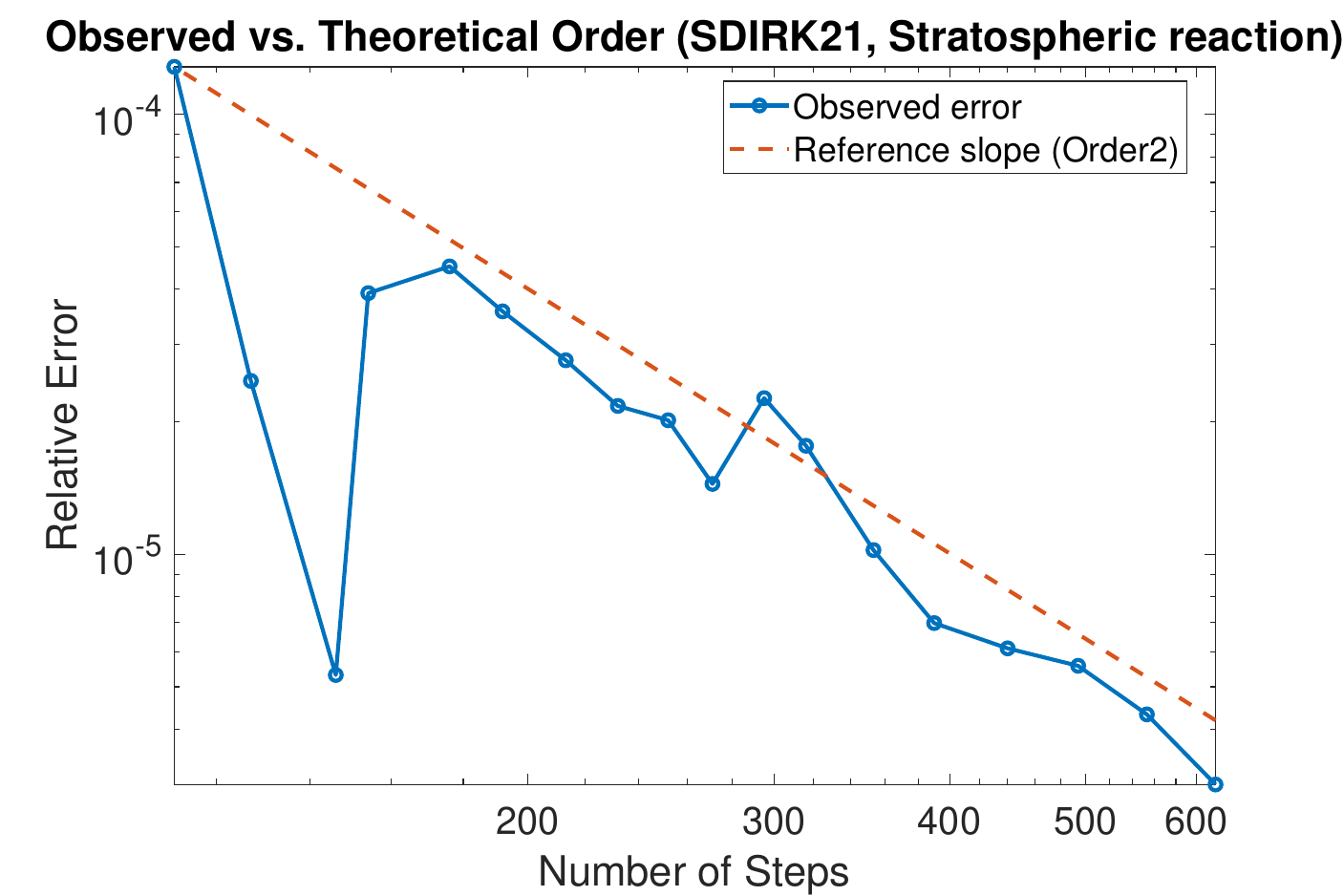}
        \caption{SDIRK21}
        \label{sfig:sdirk-strat-2_last-stage}
    \end{subfigure}
    \begin{subfigure}[b]{0.49\textwidth}
        \centering
        \includegraphics[width=\textwidth]{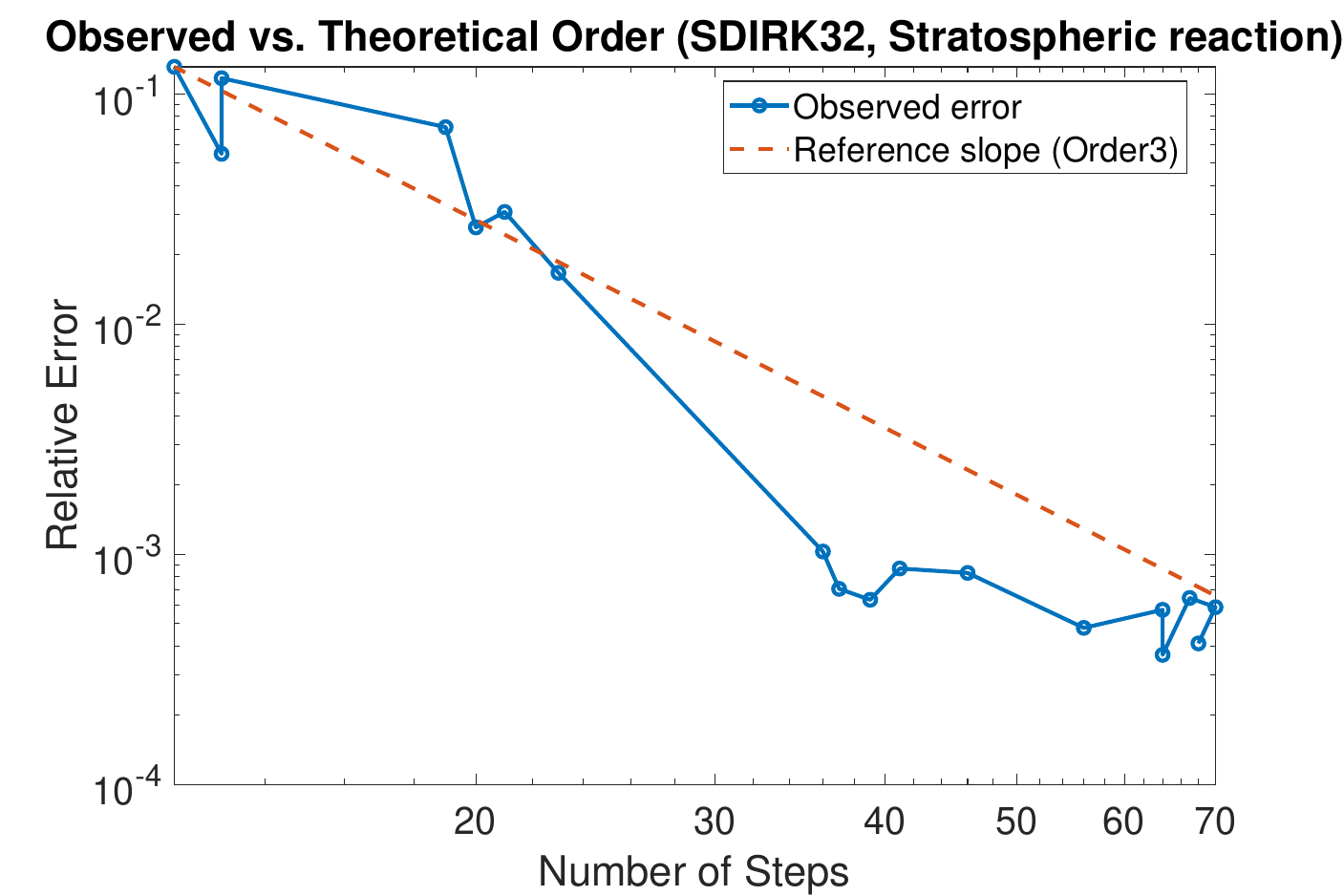}
        \caption{SDIRK32}
        \label{sfig:sdirk-strat-3_last-stage}
    \end{subfigure}
    \begin{subfigure}[b]{0.49\textwidth}
        \centering
        \includegraphics[width=\textwidth]{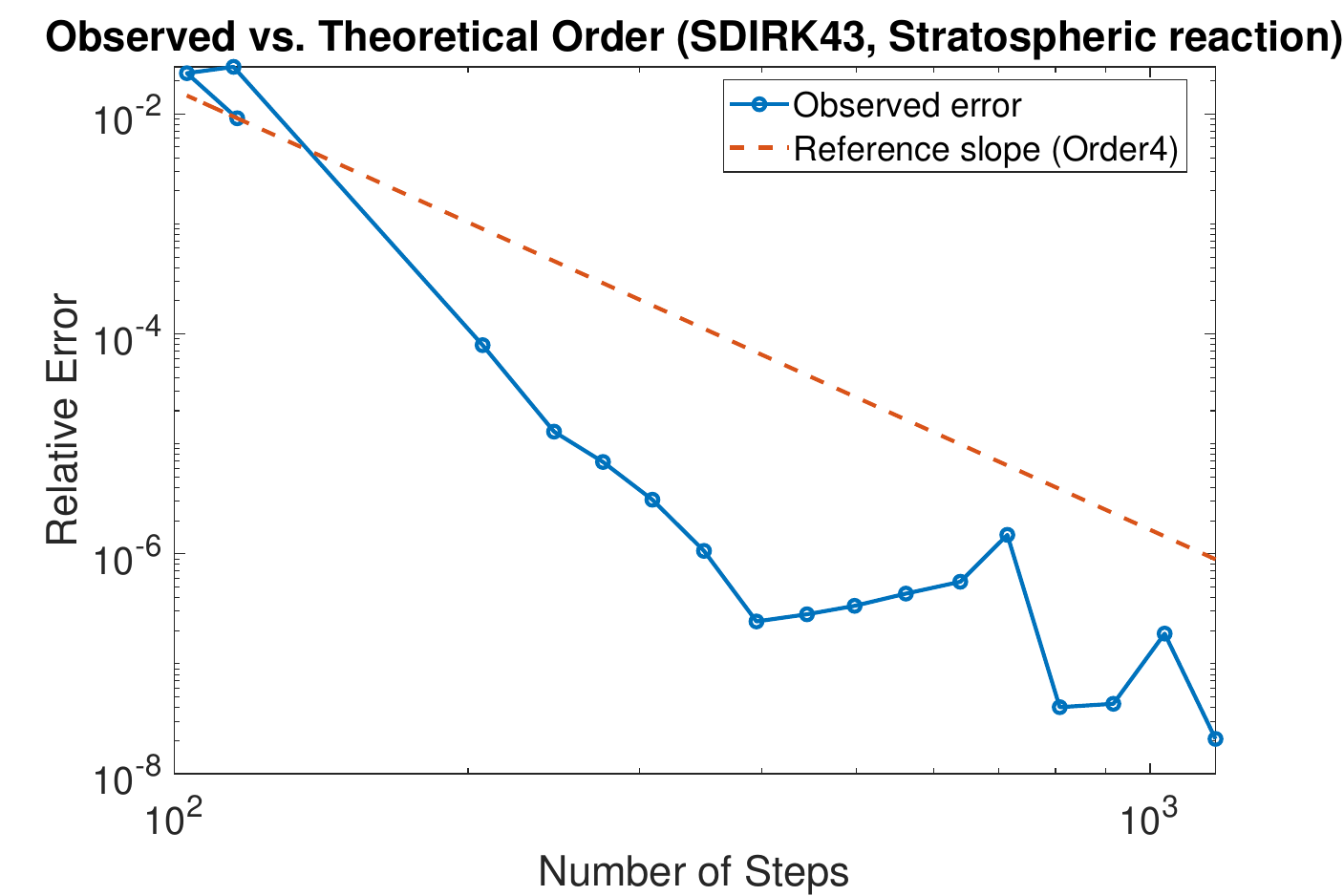}
        \caption{SDIRK43}
        \label{sfig:sdirk-strat-4_last-stage}
    \end{subfigure}
        \caption{Stratospheric Reaction \eqref{eqn:strato-ode}. Observed convergence orders for SDIRK methods with positivity correction applied only to the final stage.}
        \label{fig:sdirk-strat_last-stage}
\end{figure}

\begin{figure}[htbp]
    \centering
    \begin{subfigure}[b]{0.49\textwidth}
        \centering
        \includegraphics[width=\textwidth]{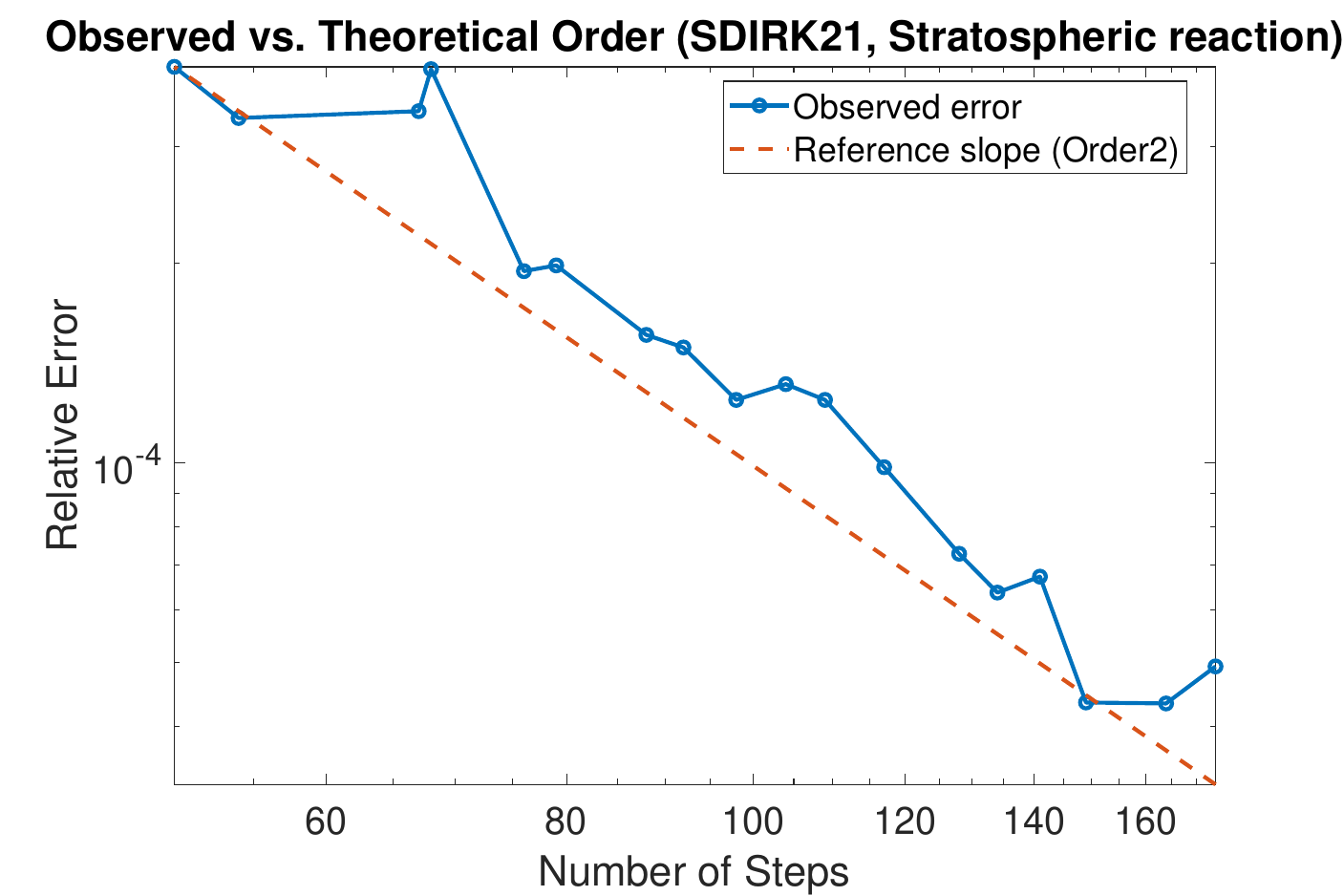}
        \caption{SDIRK21}
        \label{sfig:sdirk-strat-2_all-stages}
    \end{subfigure}
    \begin{subfigure}[b]{0.49\textwidth}
        \centering
        \includegraphics[width=\textwidth]{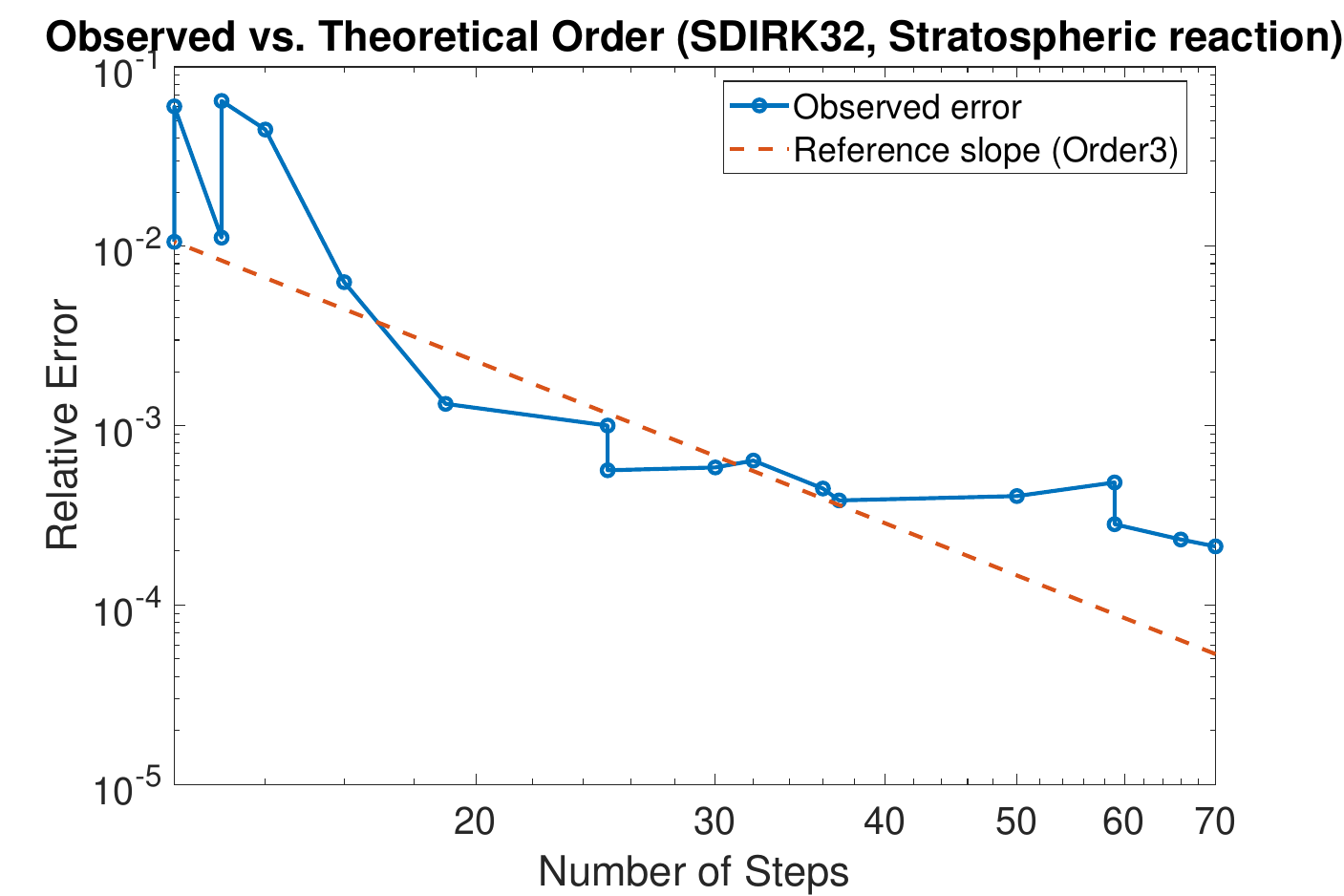}
        \caption{SDIRK32}
        \label{sfig:sdirk-strat-3_all-stages}
    \end{subfigure}
    \begin{subfigure}[b]{0.49\textwidth}
        \centering
        \includegraphics[width=\textwidth]{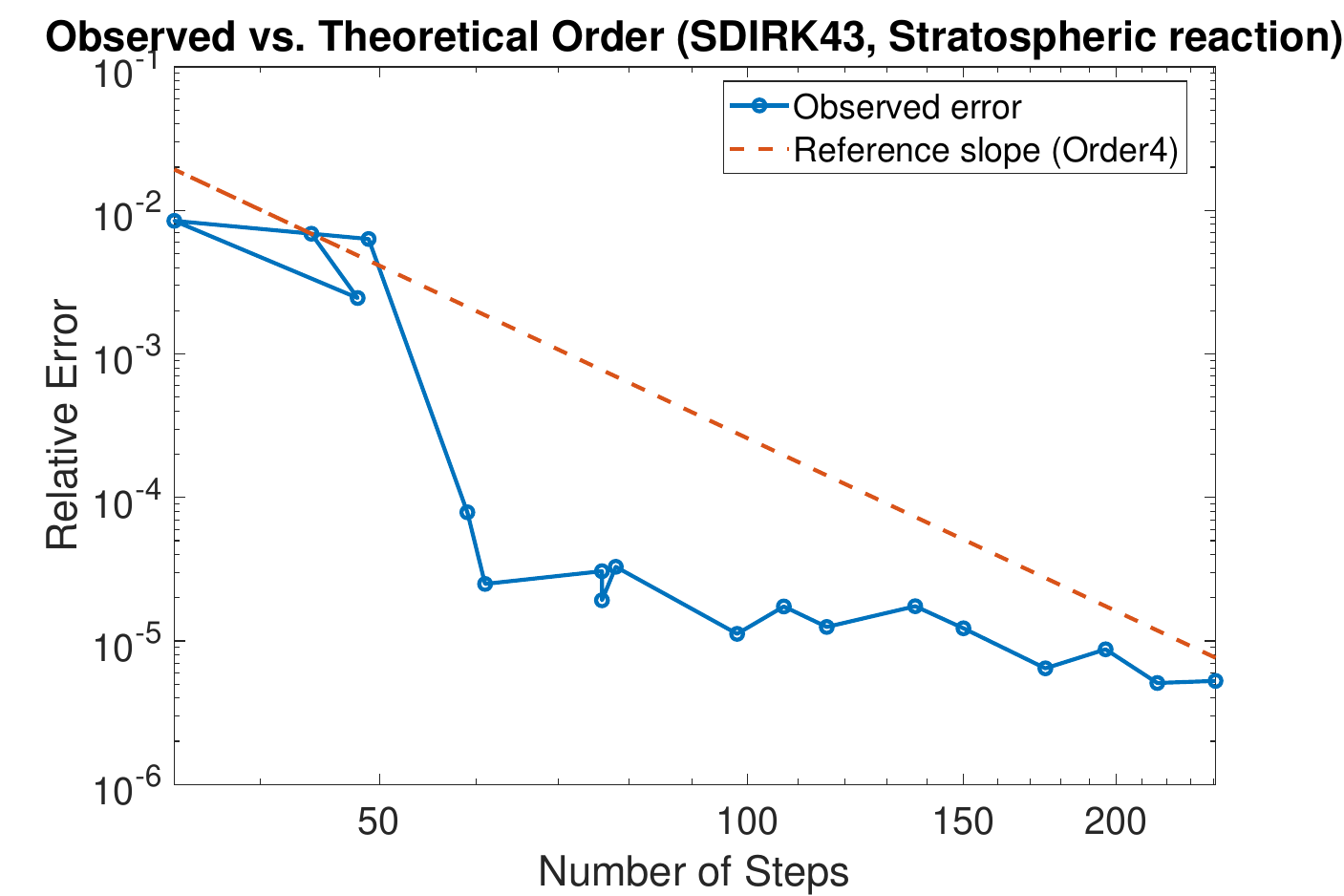}
        \caption{SDIRK43}
        \label{sfig:sdirk-strat-4_all-stages}
    \end{subfigure}
        \caption{Stratospheric Reaction \eqref{eqn:strato-ode}. Observed convergence orders for SDIRK methods with positivity correction applied to all stages.}
        \label{fig:sdirk-strat_all-stages}
\end{figure}

\subsubsection{KdV Equation}

We applied the SDIRK integrators in fixed-step mode to directly measure the convergence rate with respect to the time step size $h$. The log-log plots of error versus time step (Figure \ref{fig:sdirk-kdv_last-stage}) gives the expected slopes. The fixed-step results confirm the expected asymptotic orders of
the SDIRK methods for the dispersive PDE setting.

\begin{figure}[htbp]
    \centering
    \begin{subfigure}[b]{0.49\textwidth}
        \centering
        \includegraphics[width=\textwidth]{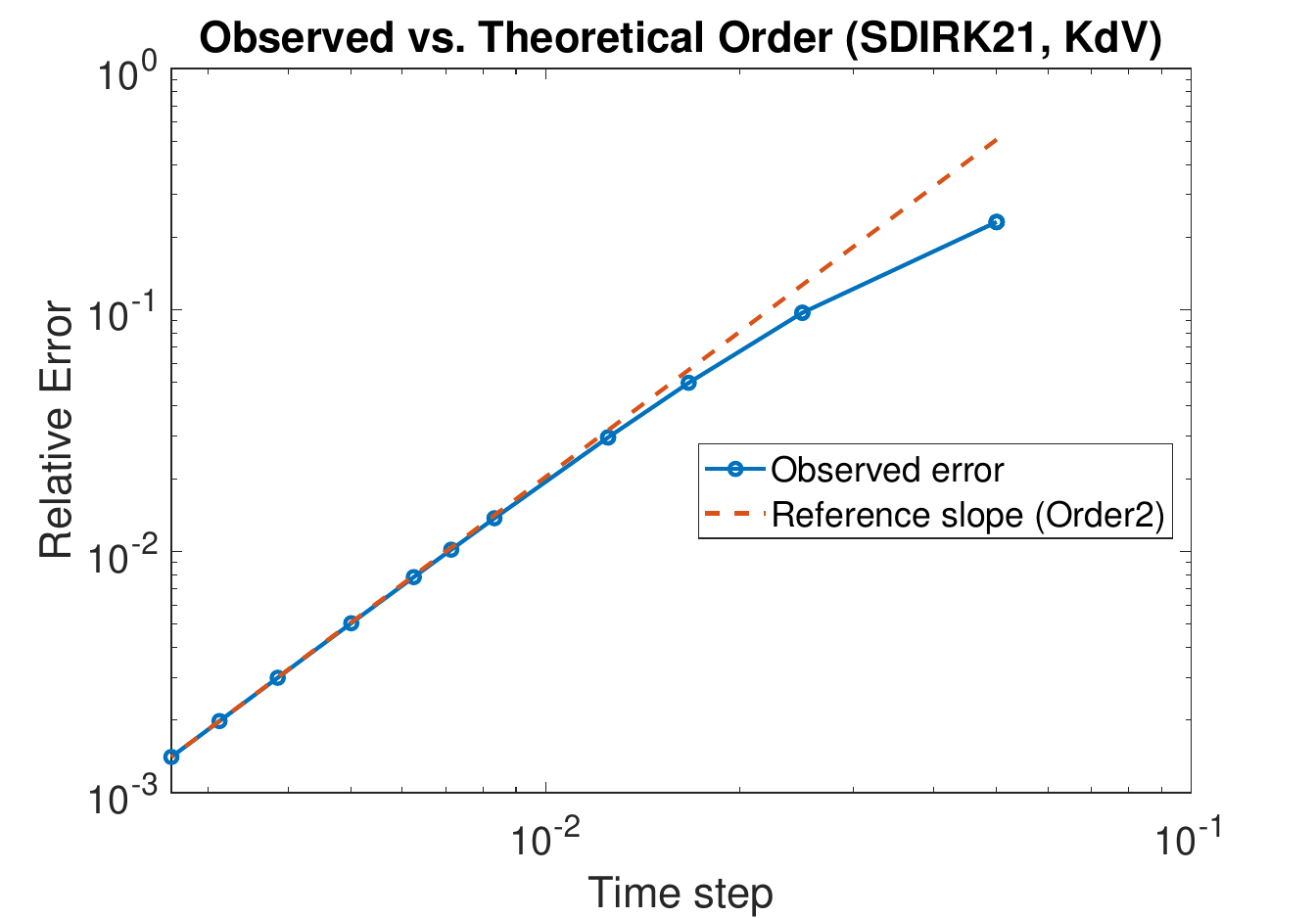}
        \caption{SDIRK21}
        \label{sfig:sdirk-kdv-2_last-stage}
    \end{subfigure}
    \begin{subfigure}[b]{0.49\textwidth}
        \centering
        \includegraphics[width=\textwidth]{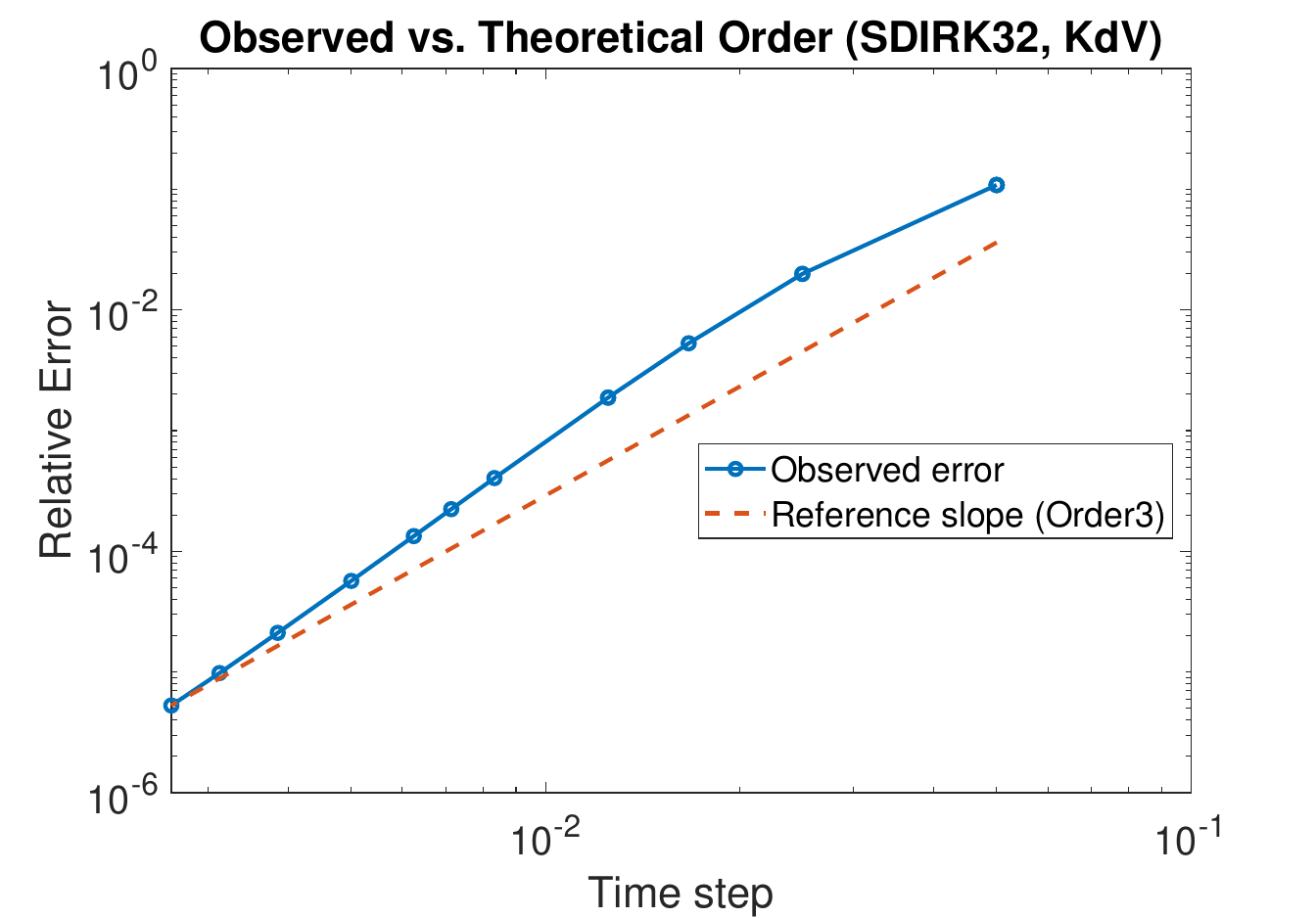}
        \caption{SDIRK32}
        \label{sfig:sdirk-kdv-3_last-stage}
    \end{subfigure}
    \begin{subfigure}[b]{0.49\textwidth}
        \centering
        \includegraphics[width=\textwidth]{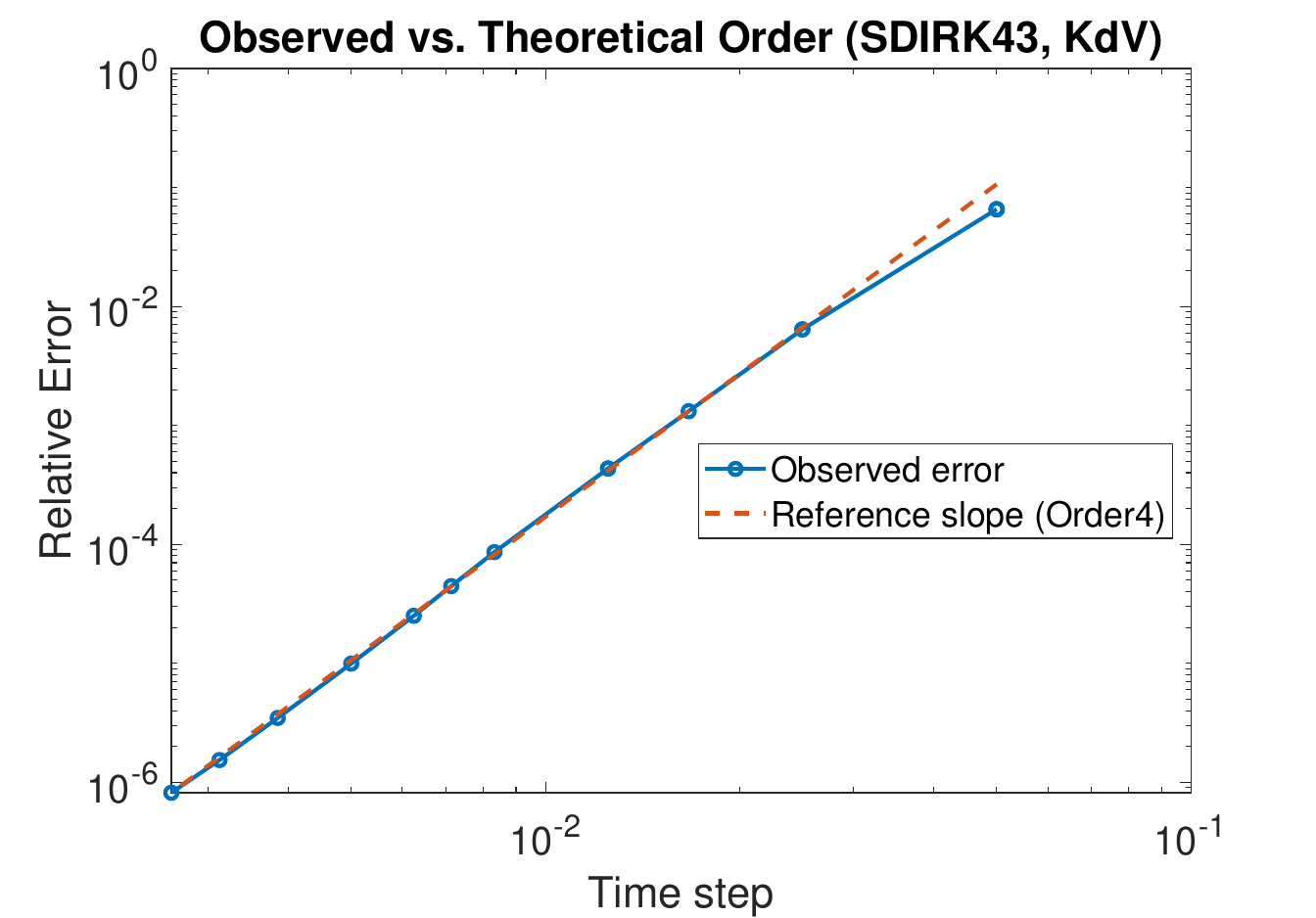}
        \caption{SDIRK43}
        \label{sfig:sdirk-kdv-4_last-stage}
    \end{subfigure}
       \caption{KdV equation \eqref{eqn:kdv-pde}. Observed convergence orders for SDIRK methods with positivity correction applied only to the final stage.}
       \label{fig:sdirk-kdv_last-stage}
\end{figure}

\subsubsection{Summary}

The estimated order results are summarized in Table \ref{tab:orderFinal} and Table \ref{tab:orderStage}. Across all test problems, the results demonstrate that the proposed correction mechanism preserves, in practice, the formal order of the implicit SDIRK methods across diverse classes of problems. These include stiff chemical kinetics, multiscale biochemical signaling networks, and nonlinear dispersive systems, both in adaptive and fixed step-size settings.

\begin{table}[t]
\centering
\caption{Empirical orders of the methods with correction at the final stage only.}
\label{tab:orderFinal}
\resizebox{\linewidth}{!}{%
\begin{tabular}{lccccc}
\toprule
Model & SDIRK21 & SDIRK32 & SDIRK43 \\
\midrule
Stratospheric reaction & $-2.07$ &  $-4.27$ & $-4.70$ \\
MAPK cascade           & $-1.98$ & $-2.76$ & $-3.81$  \\
Robertson reaction     & $-1.51$ & $-3.47$ & $-5.47$  \\
KdV equation           & $-1.72$ & $-3.37$  & $-3.85$  \\
\bottomrule
\end{tabular}}
\end{table}

\begin{table}[t]
\centering
\caption{Empirical orders of the methods with correction at each stage.}
\label{tab:orderStage}
\resizebox{\linewidth}{!}{%
\begin{tabular}{lccccc}
\toprule
Model & SDIRK21 & SDIRK32 & SDIRK43 \\
\midrule
Stratospheric reaction & $-2.07$ &  $-4.28$ & $-4.62$ \\
MAPK cascade           & $-1.98$ & $-2.75$ & $-3.86$  \\
Robertson reaction     & $-1.95$ & $-2.84$ & $-3.76$  \\
\bottomrule
\end{tabular}}
\end{table}

\subsection{Computational Overhead Analysis}
\label{subsec:efficiency}

The computational overhead of the proposed positivity-preserving corrected integrator was evaluated against the base Runge-Kutta scheme without corrections. The comparison was conducted using our four representative production-destruction models.

All simulations were run using fixed relative and absolute tolerances of $10^{-7}$ to ensure consistent accuracy across methods. For each problem, we measured the wall-clock time required to complete the integration over its simulation time span: 3 days ($[12 \cdot 3600, 84 \cdot 3600]$) for the stratospheric reaction, $[0, 200]$ for the MAPK cascade, $[0, 10^4]$ for the Robertson reaction, and $[0, 1.7]$ for the KdV equation, corresponding to approximately one soliton crossing of the computational domain. Each experiment was repeated five times, and average runtimes were computed to mitigate fluctuations due to background processes. All computations were performed on a laptop equipped with an AMD Ryzen\texttrademark\ 7 5825U CPU (8 cores), 16 GB of RAM, running Windows 11, using MATLAB R2023b. Timing was recorded using MATLAB's \texttt{tic} and \texttt{toc} commands. All tests were conducted using the SDIRK43 scheme, a fourth-order singly diagonally implicit Runge-Kutta method, both with and without the proposed positivity-preserving correction.

Table \ref{tab:timing} summarizes the results for the four test problems.
The stratospheric reaction, a highly stiff system with many positivity-sensitive
interactions, exhibited a substantial runtime increase when corrections were
applied. In contrast, the MAPK cascade showed negligible overhead, while the
Robertson reaction experienced a significant speedup. The KdV equation displayed
moderate overhead due to the additional correction steps.

These results demonstrate that while positivity-preserving corrections can introduce overhead in some stiff or dispersive problems, they can also lead to performance improvements in other situations. Overall, the proposed correction mechanism is a practical and effective solution for a wide range of stiff ODE systems.

\begin{table}[t]
\centering
\caption{Average wall-clock runtimes for the SDIRK43 scheme with and without
positivity-preserving correction. Each experiment was repeated five times and
average runtimes are reported.}
\label{tab:timing}
\resizebox{\linewidth}{!}{%
\begin{tabular}{lccccc}
\toprule
Model & Time interval & No correction (s) & With correction (s) & Overhead (\%) \\
\midrule
Stratospheric reaction & $[12\cdot3600,\,84\cdot3600]$ & 30.2  & 110.5 & +266 \\
MAPK cascade           & $[0,\,200]$                  & 0.560 & 0.516  & $-$7.9 \\
Robertson reaction     & $[0,\,10^{4}]$               & 0.485 & 0.090  & $-$81.4 \\
KdV equation           & $[0,\,1.7]$                  & 10.9  & 21.2   & +94.5 \\
\bottomrule
\end{tabular}}
\end{table}

\section{Conclusions}
\label{sec:conclusions}
This paper addressed the development and evaluation of a positivity-preserving Patankar correction strategy for the numerical integration of stiff systems of ordinary differential equations representing production-destruction processes. Recognizing that standard integration methods may violate fundamental structural properties such as non-negativity and conservation, we proposed a post-processing approach. This approach is compatible with general time integrators, including singly diagonally implicit Runge-Kutta (SDIRK) schemes. The correction procedure, based on clipping negative entries and applying structure-preserving scaling, was shown to maintain positivity and preserve linear invariants without requiring changes to the base integration algorithm.

While the focus here is on stiff systems and implicit Runge-Kutta schemes, the corrections procedure could be applied in principle to explicit Runge-Kutta schemes. However, the corrections requires solving the linear system and the advantage of explicit computations can be lost.

Extensive numerical experiments involving the stratospheric reaction system, the MAPK cascade, the Robertson reaction, and the semi-discrete Korteweg-De Vries equation demonstrated that the proposed correction reliably enforces positivity and invariant conservation. This holds across a wide range of stiff and multiscale problems. For the ODE benchmarks, both final-stage and all-stage corrections performed well, whereas for the KdV equation, only the final-stage correction preserved the nonlinear wave dynamics. Order validation tests confirmed that the formal accuracy of SDIRK methods was retained in practice.
Efficiency benchmarks further revealed that the additional correction step introduces only modest computational overhead relative to implicit time stepping; while in stiff regimes, it improved robustness by preventing solver breakdown.
This overhead could be alleviated if one uses the correction matrix \(\bar{\mathbf{G}}\) computed in the current for the correction of $\mathbf{y}_{n+1}$ as the approximate Jacobian for the nonlinear solver in the subsequent step to compute $\mathbf{y}^{(p)}_{n+2}$. A future study will consider this strategy to increase computational efficiency.

In conclusion, the presented positivity-preserving correction strategy provides a practical and effective way to enforce physical constraints such as positivity, mass conservation, and accuracy. It ensures these properties in the numerical solution of stiff ODEs and related PDEs. It extends the applicability of classical time integration methods to a broad class of scientific and engineering problems, from atmospheric chemistry to nonlinear wave propagation, and offers a foundation for further work on higher-stage-order DIRK methods, multi-invariant systems, and large-scale PDE applications.

\section*{Acknowledgment}

The work of Amit N. Subrahmanya was supported by the U.S. Department of Energy, Office of Science, Advanced Scientific Computing Research Program under contract number DE-AC02-06CH11357. The work of Adrian Sandu and Reid J. Gomillion was supported by the National Science Foundation (NSF) awards DMS-2436357, DMS-2411069, and by the Computational Science Laboratory.

The authors also thank the anonymous reviewers for their constructive feedback and insightful suggestions such as the property \eqref{eq:strict-positivity}, which helped improve the work in this paper.

During the preparation of this work the author(s) used ChatGPT in order to polish and refine the text of the manuscript. After using this tool/service, the author(s) reviewed and edited the content as needed and take(s) full responsibility for the content of the published article.

\bibliographystyle{elsarticle-num} 
\bibliography{ref}

\end{document}